\documentclass[DIV=9, headings=small,10pt,british,captions=tableheading]{scrartcl}
\pdfoutput=1

\title{\vspace{-2cm}Shuffling functors and spherical twists on $\Db(\CatO_0)$}
\author{\href{mailto:fabian.lenzen@tum.de}{Fabian Lenzen}, TU Munich}
\date{26th March 2021}

\usepackage{babel}
\usepackage[utf8]{inputenc}
\usepackage[T1]{fontenc}
\usepackage{csquotes}
\usepackage{chars}
\usepackage{microtype}

\usepackage{booktabs,array,multirow,blkarray,tabularx}

\usepackage[shortcuts]{extdash}					%

\usepackage{afterpage}
\usepackage[all, defaultlines=3]{nowidow}
\usepackage{trimclip}
\usepackage[backend=biber,style=alphabetic,citetracker=true, giveninits=true]{biblatex}

\usepackage[hidelinks,breaklinks=true,unicode]{hyperref}
\usepackage[noabbrev]{cleveref}

\setkomafont{disposition}{\fontseries{b}\selectfont\boldmath}
\setkomafont{part}{\Large}
\addtokomafont{partentry}{\normalsize}
\setkomafont{sectionentry}{}
\setkomafont{partnumber}{\Large}
\setkomafont{section}{\large}
\setkomafont{subsection}{\normalsize}
\setkomafont{subsubsection}{\normalsize}
\setkomafont{caption}{\small}
\setkomafont{captionlabel}{\itshape}
\setkomafont{title}{\large\bfseries}
\setkomafont{author}{\normalsize}
\setkomafont{date}{\normalsize}

\setcapindent{0em}            %

\SetSymbolFont{operators}{bold}{OT1}{lmr}{b}{n}
\SetMathAlphabet{\mathbf}{normal}{OT1}{lmr}{b}{n}
\SetMathAlphabet{\mathbf}{bold}{OT1}{lmr}{bx}{n}
\SetMathAlphabet{\mathsf}{bold}{OT1}{lmss}{b}{n}
\SetMathAlphabet{\mathit}{bold}{OT1}{lmr}{b}{it}

\DeclareFontFamily{U}{mathc}{}
\DeclareFontShape{U}{mathc}{m}{it}{<->s*[1.03] mathc10}{}
\DeclareMathAlphabet{\mathcalx}{U}{mathc}{m}{it}

\let\include\input

\RedeclareSectionCommand[
  beforeskip=-8.5ex plus -1ex minus -.5ex,
  afterskip=2.3ex plus .2ex
]{section}

\renewcommand*{\bibfont}{\small}                             %
\DeclareBibliographyAlias{preprint}{misc}                    %
\DeclareFieldFormat[preprint]{title}{\mkbibquote{#1\isdot}}  %
\DeclareFieldFormat{eprint:HAL}{\textsc{hal}:~\href{http://hal.archives-ouvertes.fr/#1}{#1}}
\DeclareFieldFormat{eprint:EuDML}{\textsc{EuDML}:~\href{http://eudml.org/doc/#1}{#1}}
\newcolumntype{Q}[1]{>{\centering\arraybackslash$}p{#1}<{$}} %
\newcolumntype{C}{>{$}Sc<{$}}                                %

\usepackage[inline, shortlabels]{enumitem}
\setlist[itemize]{itemsep=0pt}

\usepackage{amsthm,thmtools}
\newtheoremstyle{theorem}{}{}{}{}{\itshape}{.}{ }{}
\theoremstyle{theorem}

\newtheorem*{theorem*}{Theorem}
\newtheorem{theorem}{Theorem}[section]
\newtheorem{lemma}[theorem]{Lemma}
\newtheorem*{proposition*}{Proposition}
\newtheorem{proposition}[theorem]{Proposition}
\newtheorem{corollary}[theorem]{Corollary}

\newtheorem*{definition*}{Definition}
\newtheorem{definition}[theorem]{Definition}
\newtheorem{definition-lemma}[theorem]{Definition and Lemma}

\newtheorem{example}[theorem]{Example}
\newtheorem{notation}[theorem]{Notation}

\newtheorem{remark}[theorem]{Remark}
\newtheorem{caveat}[theorem]{Caveat}

\newcommand{\thmtitle}[1]{(\textit{#1})}
\patchcmd{\thmhead}{(#3)}{#3}{}{}

\crefname{lemma}{Lemma}{Lemmas}
\crefname{observation}{Observation}{Observations}
\crefname{notation}{Notation}{Notations}
\crefname{caveat}{Caveat}{Caveats}
\crefformat{footnote}{#2\footnotemark[#1]#3}
\crefformat{equation}{(#2#1#3)}
\crefrangeformat{equation}{(#3#1#4--#5#2#6)}
\crefmultiformat{equation}{(#2#1#3)}{ and~(#2#1#3)}{, (#2#1#3)}{ and~(#2#1#3)}

\newlist{thmlist}{enumerate}{1}
\setlist[thmlist]{label=\textup{(\roman*)}, ref=\thetheorem.\textup{(\roman*)}, noitemsep}

\newlist{prooflist}{enumerate}{1}
\setlist[prooflist]{label=\textup{(\roman{prooflisti})}, ref=(\roman{prooflisti}),noitemsep}
\crefname{prooflisti}{part}{parts}

\newlist{thmlist*}{enumerate*}{1}
\setlist[thmlist*]{label=\textup{(\roman{thmlist*i})}, ref=\thetheorem.(\roman{thmlist*i}),noitemsep}

\makeatletter
\newcommand\strongnopagebreak{\nopagebreak\@afterheading} 
\makeatother

\addtotheorempostheadhook[theorem]{\crefalias{thmlisti}{theorem}}
\addtotheorempostheadhook[remark]{\crefalias{thmlisti}{remark}}
\addtotheorempostheadhook[lemma]{\crefalias{thmlisti}{lemma}}
\addtotheorempostheadhook[recall]{\crefalias{thmlisti}{recall}}
\addtotheorempostheadhook[definition]{\crefalias{thmlisti}{definition}}
\addtotheorempostheadhook[example]{\crefalias{thmlisti}{example}}
\addtotheorempostheadhook[fact]{\crefalias{thmlisti}{fact}}
\addtotheorempostheadhook[proposition]{\crefalias{thmlisti}{proposition}}
\addtotheorempostheadhook[notation]{\crefalias{thmlisti}{notation}}

\numberwithin{equation}{section}
\numberwithin{table}{section}
\numberwithin{figure}{section}

\usepackage{atbegshi}
\newcommand\showtimer{%
	\PackageWarning{Timer}{Page \thepage: \the\numexpr\the\pdfelapsedtime*1000/65536\relax ms}%
	\pdfresettimer}
\AtBeginDocument{\showtimer}
\AtBeginShipout {\showtimer}

\newcommand{\nonumberfootnote}[1]{%
	\begingroup%
	\renewcommand\thefootnote{}\footnote{#1}%
	\addtocounter{footnote}{-1}%
	\endgroup%
}

\usepackage{chars}
\usepackage{amsmath,amsthm,amssymb,amsfonts,mathtools,thmtools,needspace}
\usepackage[only,mapsfrom,longmapsfrom,mapsfromchar]{stmaryrd}
\allowdisplaybreaks

\usepackage{bm,oubraces}
\usepackage[math]{cellspace}
\usepackage{ifthen,xifthen,xargs,calc}

\makeatletter
\newcommand\ie{i.\,e\@ifnextchar.{}{.\@}}
\newcommand\cf{cf\@ifnextchar.{}{.\@}}
\newcommand\eg{e.\,g\@ifnextchar.{}{.\@}}
\newcommand\wrt{w.\,r.\,t\@ifnextchar.{}{.\@}}
\newcommand{\resp}{resp\@ifnextchar.{}{.\@}}
\newcommand{\Wlog}{W.\,l.\,o.\,g\@ifnextchar.{}{.\@}}
\newcommand{\wlofg}{w.\,l.\,o.\,g\@ifnextchar.{}{.\@}}
\newcommand{\ses}{s.\,e.\,s\@ifnextchar.{}{.\@}}
\newcommand{\st}{s.\,t\@ifnextchar.{}{.\@}}

\DeclareMathOperator{\Hom}{Hom}
\DeclareMathOperator{\Ext}{Ext}
\DeclareMathOperator{\End}{End}
\newcommand{\id}{\mathrm{id}}%
\DeclareMathOperator{\im}{im}
\DeclareMathOperator{\Sh}{Sh}
\DeclareMathOperator{\Csh}{Csh}
\DeclareMathOperator{\cone}{cone}
\DeclareMathOperator{\cocone}{cocone}
\DeclareMathOperator{\coker}{coker}
\DeclareMathOperator{\can}{can}
\DeclareMathOperator{\Ch}{Ch}

\DeclareMathOperator{\Res}{Res}
\DeclareMathOperator{\Ind}{Ind}
\newcommand{\ev}{\mathit{ev}}
\DeclareMathOperator{\rad}{rad}
\DeclareMathOperator{\soc}{soc}
\newcommand{\lin}{\operatorname{lin}_\mathbf{C}}%

\newcommand{\unit}{\eta}
\newcommand{\counit}{\varepsilon}

\newcommand{\isom}{\cong}
\newcommand{\ldot}{\mathbin{.}}

\newcommand{\qis}{\simeq}

\newcommand{\Db}{D^\mathrm{b}}
\newcommand{\Kb}{K^\mathrm{b}}

\newcommand{\SL}{\mathfrak{sl}}

\newcommand{\degZero}[2][0]{\underset{\mathsf{#1}}{#2}}

\newcommand{\SPol}{\mathop{\Lambda\operator@font Pol}\nolimits}

\newcommand{\SPolC}{\mathop{\Lambda{\operator@font Pol}\mathfrak C}\nolimits}

\newcommand{\CatO}{\mathcal{O}}
\newcommand{\gCatO}{\mathcal{O}^{\mathbf Z}}
\newcommand{\Aut}{\operatorname{Aut}}
\newcommand\Sph{\mathcal{S}\mkern-3mu\mathcalx{ph}}

\renewcommand{\degZero}[1]{\underset{0}{#1}}

\newcommand\Mtrx[1]{%
	\begin{psmallmatrix}%
		#1%
	\end{psmallmatrix}%
}

\newenvironment{smallcases}{\left\{\begin{smallmatrix*}[l]}{\end{smallmatrix*}\right.}

\newcommand{\subalign}[1]{%
	\vcenter{%
		\Let@ \restore@math@cr \default@tag
		\baselineskip\fontdimen10 \scriptfont\tw@
		\advance\baselineskip\fontdimen12 \scriptfont\tw@
		\lineskip\thr@@\fontdimen8 \scriptfont\thr@@
		\lineskiplimit\lineskip
		\ialign{\hfil$\m@th\scriptstyle##$&&$\m@th\scriptstyle##$\hfil\crcr
			#1\crcr
		}%
	}
}
\newenvironment{scriptaligned}[1][c]{%
	\,\hbox\bgroup%
	\fontsize{\sf@size}{\dimexpr\sf@size pt+1pt}\selectfont%
	$\!\aligned[#1]%
}{%
	\endaligned$\egroup
}

\newcommand\restrict[2]{{%
	\left.\kern-\nulldelimiterspace %
	#1 %
	\right|_{#2} %
}}

\newcommand{\xmapsfrom}[2][]{%
	\ext@arrow3095\leftarrowfill@{#1}{#2}\mapsfromchar
}
\newcommand{\xtofrom}[2][]{%
	\mathrel{%
		\raise.4ex\hbox{%
			$\ext@arrow 0395\rightarrowfill@{\phantom{#1}}{#2}$}%
		\setbox0=\hbox{%
			$\ext@arrow 0395\leftarrowfill@{#1}{\phantom{#2}}$}%
		\kern-\wd0 \lower.4ex\box0%
	}%
}
\newcommand{\mapstofrom}{%
	\mathrel{\ooalign{%
		\raise.4ex\hbox{$\mapsto$}\cr%
		\raise-.4ex\hbox{$\mapsfrom$}\cr}%
	}%
}
\renewcommand{\rightleftarrows}{\xtofrom[]{}}
\newcommand{\longrightleftarrows}{\xtofrom[]{\quad}}

\DeclareRobustCommand{\longinto}{%
  \mathrel{\mathpalette\shook@rightarrow\relax}%
}

\newcommand{\shook@rightarrow}[2]{%
  \sbox\z@{$\m@th#1\lhook$}%
  {\lhook}%
  \kern-.25\wd\z@
  {\smash{\clipbox{{.75\wd\z@} 0pt 0pt {-\width}}{$\m@th#1\longrightarrow$}}}%
}

\newcommand*{\relrelbarsep}{.386ex}
\newcommand*{\relrelbar}{%
  \mathrel{%
    \mathpalette\@relrelbar\relrelbarsep
  }%
}
\newcommand*{\@relrelbar}[2]{%
  \raise#2\hbox to 0pt{$\m@th#1\relbar$\hss}%
  \lower#2\hbox{$\m@th#1\relbar$}%
}
\providecommand*{\rightrightarrowsfill@}{%
  \arrowfill@\relrelbar\relrelbar\rightrightarrows
}
\providecommand*{\xrightrightarrows}[2][]{%
  \ext@arrow 0359\rightrightarrowsfill@{#1}{#2}%
}

\let\into\hookrightarrow
\let\onto\twoheadrightarrow

\let\xto\xrightarrow
\let\from\leftarrow
\let\longto\longrightarrow

\DeclareRobustCommand\vdots{%
  \mathpalette\@vdots{}%
}
\newcommand*{\@vdots}[2]{%
	\sbox0{$#1\cdotp\cdotp\cdotp\m@th$}%
	\sbox2{$#1.\m@th$}%
	\vbox{%
	\dimen@=\wd0 %
	\advance\dimen@ -3\ht2 %
	\kern.5\dimen@
	\dimen@=\wd2 %
	\advance\dimen@ -\ht2 %
	\dimen2=\wd0 %
	\advance\dimen2 -\dimen@
	\vbox to \dimen2{%
		\offinterlineskip
		\copy2 \vfill\copy2 \vfill\copy2 %
	}%
	}%
}
\DeclareRobustCommand\ddots{%
	\mathinner{%
		\mathpalette\@ddots{}%
		\mkern\thinmuskip
	}%
}
\newcommand*{\@ddots}[2]{%
	\sbox0{$#1\cdotp\cdotp\cdotp\m@th$}%
	\sbox2{$#1.\m@th$}%
	\vbox{%
		\dimen@=\wd0 %
		\advance\dimen@ -3\ht2 %
		\kern.5\dimen@
		\dimen@=\wd2 %
		\advance\dimen@ -\ht2 %
		\dimen2=\wd0 %
		\advance\dimen2 -\dimen@
		\vbox to \dimen2{%
			\offinterlineskip
			\hbox{$#1\mathpunct{.}\m@th$}%
			\vfill
			\hbox{$#1\mathpunct{\kern\wd2}\mathpunct{.}\m@th$}%
			\vfill
			\hbox{$#1\mathpunct{\kern\wd2}\mathpunct{\kern\wd2}\mathpunct{.}\m@th$}%
		}%
	}%
}

\let\term\emph

\newcommand\trivpath[1]{\varepsilon_{#1}} %

\newcommand{\eqsp}{\mathrel{\phantom{=}}}

\newcommand{\noloc}{\nobreak\mskip6muplus1mu\mathpunct{}\nonscript
  \mkern-\thinmuskip{:}\mskip2mu\relax}

\newcommand\casesGap{\hphantom{\left\{\rule{0pt}{1cm}\right. \kern-\nulldelimiterspace}}
\newlength{\algnRef}

\mathchardef\mhyphen="2D
\newcommandx{\Mod}[3][1={},3={}]{%
	\ifthenelse{\isempty{#2}}{}{#2\mhyphen}%
	{\operator@font#1Mod}%
	\ifthenelse{\isempty{#3}}{}{\mhyphen#3}%
}

\newcommand{\Proj}[1]{%
	#1\mhyphen{\operator@font Proj}
}
\newcommand{\VS}[1]{%
	#1\mhyphen{\operator@font Vect}
}

\newcommand\hShift[1]{[#1]}
\renewcommand\Sh[1]{\operatorname{Sh}_{#1}}
\renewcommand\Csh[1]{\operatorname{Csh}_{#1}}
\newcommand\LSh[1]{\mathbf{L}\Sh{#1}}
\newcommand\RCsh[1]{\mathbf{R}\Csh{#1}}
\newcommand\Complex{\mathbf{C}}
\newcommand{\p}{\mathfrak{p}}

\let\epsilon\varepsilon

\makeatother

\usepackage{tikz}[2015/08/07]
\usepackage{tikz-cd}
\usetikzlibrary{
	calc,
	fit,
	shapes.geometric,
	decorations.pathreplacing,
	angles,quotes,
	backgrounds,
	positioning,
	calligraphy,
	intersections,
	patterns,
	quotes
}

\tikzset{
	|/.tip={Bar[width=.8ex,round]},
	back line/.style={densely dotted},
	cross line/.style={%
		preaction={%
			draw=white,
			-, 
			line width=6pt
		}
	},
	brace tip/.style = {
		sloped, allow upside down, yshift=+1ex
	},
	brace tip'/.style = {
		sloped, allow upside down, yshift=-1ex
	},
	brace/.style={
		decoration={calligraphic brace, amplitude=.8ex},
		decorate,
		line width=.3ex
	},
	brace'/.style={
		decoration={calligraphic brace, amplitude=.8ex, mirror},
		decorate,
		line width=.3ex
	},
	braced box/.style = {
		rectangle,
		inner sep=0mm,
		append after command = {
			\pgfextra{
				\draw[brace] (\tikzlastnode.north east) to node[xshift=1ex] (#1){} (\tikzlastnode.south east);
			}
		}
	},
	braced box'/.style = {
		rectangle,
		inner sep=0mm,
		append after command = {
			\pgfextra{
				\draw[brace] (\tikzlastnode.south west) to node[xshift=-1ex] (#1){} (\tikzlastnode.north west);
			}
		}
	},
	braced box''/.style = {
		rectangle,
		inner sep=0mm,
		append after command = {
			\pgfextra{
				\draw[brace] (\tikzlastnode.north west) to node[yshift=1ex] (#1){} (\tikzlastnode.north east);
			}
		}
	},
	braced box'''/.style = {
		rectangle,
		inner sep=0mm,
		append after command = {
			\pgfextra{
				\draw[brace] (\tikzlastnode.south east) to node[yshift=-1ex] (#1){} (\tikzlastnode.south west);
			}
		}
	},
	mth/.style = {
		commutative diagrams/every diagram,
		every path/.style={
			commutative diagrams/.cd, every arrow, every label, arrows=dash, tips=on proper draw
		},
		matrix of nodes/.append style={
			nodes={font=\normalsize}
		}
	}
}
\tikzset{
	diag/.style={
		execute at begin picture={%
			\let\oldblacklozenge\blacklozenge%
			\let\oldlozenge\lozenge%
			\newcommand{\blacklozengeX}{\scalebox{.65}{$\oldblacklozenge$}}%
			\newcommand{\lozengeX}{\scalebox{.65}{$\oldlozenge$}}%
			\newcommand{\decofactor}{1}
			\let\oldbullet\bullet%
			\let\oldcirc\circ
			\renewcommand{\blacklozenge}{\scalebox{\decofactor}{$\blacklozengeX$}}%
			\renewcommand{\lozenge}{\scalebox{\decofactor}{$\lozengeX$}}%
			\renewcommand{\bullet}{\scalebox{\decofactor}{$\oldbullet$}}%
			\renewcommand{\circ}{\scalebox{\decofactor}{$\oldcirc$}}%
		},
		x=5mm,
		y=5mm,
		baseline=-.5ex,
		curved/.style={looseness=1.7},
		rot/.style={sloped},
		deco/.style={
			font=\normalsize,
			every deco,
			inner sep=0pt,
			label distance=0pt
		},
		every node/.append style={
			inner sep=0pt,
			font=\scriptsize
		},
		below/.default={
			1.2mm
		},
		above/.default={
			1.2mm
		},
		every label/.append style={
			font=\scriptsize,
			inner sep=0,
			label distance=1mm
		},
		smaller/.style={
			scale=.8,
			every deco/.append style={
				execute at begin node={
					\renewcommand{\decofactor}{.9}
				}
			}	
		},
		inline/.style={
			x=1ex, y=1ex,
			every deco/.append style={
				execute at begin node={
				}
			},
			left/.default={.5ex},
			right/.default={.5ex}
		},
		every diag,
		every inline diag,
	},
	every deco/.style={},
	every diag/.style={},
	every inline diag/.style={}
}
\tikzset{
  bezier/controls/.code args={(#1) and (#2)}{
    \def\mystartcontrol{#1}
    \def\mytargetcontrol{#2}
  },
  bezier/limit/.store in=\mylimit,
  bezier/limit=1cm,
  bezier/.code={
    \tikzset{bezier/.cd,#1}
    \tikzset{
      to path={
        let
        \p0=(\tikztostart),    \p1=(\mystartcontrol),
        \p2=(\mytargetcontrol), \p3=(\tikztotarget),
        \n0={veclen(\x1-\x0,\y1-\y0)},
        \n1={veclen(\x3-\x2,\y3-\y2)},
        \n2={\mylimit}
        in  \pgfextra{
          \pgfmathtruncatemacro\ok{max((\n0>\n2),(\n1>\n2))}
        }
        \ifnum\ok=1 %
        let
        \p{01}=($(\p0)!.5!(\p1)$), \p{12}=($(\p1)!.5!(\p2)$), \p{23}=($(\p2)!.5!(\p3)$),
        \p{0112}=($(\p{01})!.5!(\p{12})$), \p{1223}=($(\p{12})!.5!(\p{23})$),
        \p{01121223}=($(\p{0112})!.5!(\p{1223})$)
        in
        to[bezier={controls={(\p{01}) and (\p{0112})}}]
        (\p{01121223})
        to[bezier={controls={(\p{1223}) and (\p{23})}}]
        (\p3)
        \else
        [overlay=false] .. controls (\p1) and (\p2) ..  (\p3) [overlay=true]
        \fi
      },
    },
  },
  limit bb/.style n args={2}{
    overlay,
    decorate,
    decoration={
      show path construction,
      moveto code={},
      lineto code={\path[#2] (\tikzinputsegmentfirst) -- (\tikzinputsegmentlast);},
      curveto code={
        \path[#2]
        (\tikzinputsegmentfirst)
        to[bezier={limit=#1,controls={(\tikzinputsegmentsupporta) and (\tikzinputsegmentsupportb)}}]
        (\tikzinputsegmentlast);
      },
      closepath code={\path[#2] (\tikzinputsegmentfirst) -- (\tikzinputsegmentlast);},
    },
  },
  limit bb/.default={1mm}{draw},
}

\usepackage[]{ifdraft}
\tikzcdset{
	slightly cramped/.style = {
		every matrix/.append style={inner sep=+-.2em},
		every cell/.append style={inner sep=+.3em}
	}
}
\ifdraft{
	
	\tikzset{
		limit bb/.style n args={2}{},
		out/.style={},
		in/.style={},
		in looseness/.style={},
		out looseness/.style={},
		looseness/.style={},
		brace/.style={},
		brace'/.style={}
	}
}{}

\RedeclareSectionCommand[beforeskip=-2.3ex plus -1ex minus -.5ex, afterskip=-.5em]{section}
\RedeclareSectionCommand[beforeskip=-2.3ex plus -1ex minus -.5ex, afterskip=-.5em]{subsection}
\setkomafont{section}{\normalsize}
\setkomafont{subsection}{\normalsize}
\setkomafont{title}{\large\bfseries\boldmath}
\setkomafont{author}{\normalsize}
\setkomafont{date}{\normalsize}

\renewcommand{\tableofcontents}{
	\section*{Contents}
	\leavevmode
	\par
	\listoftoc*{toc}
}

\begin{document}
\maketitle
\nonumberfootnote{
	Published in Journal of Algebra, 
	doi: \href{http://dx.doi.org/10.1016/j.jalgebra.2021.02.024}{10.1016/j.jalgebra.2021.02.024}.
}

\begin{abstract}
	\noindent
	{\bfseries Abstract}\enspace
	For a semisimple complex Lie algebra $\mathfrak g$,
	the BGG category $\CatO$ is of particular interest in representation theory.
	It is known that Irving's shuffling functors $\Sh{w}$,
	indexed by elements $w\in W$ of the Weyl group,
	induce an action of the braid group $B_W$ associated to $W$
	on the derived categories $\Db(\CatO_\lambda)$ of blocks of $\CatO$.
	
	We show that for maximal parabolic subalgebras $\p$ of $\SL_n$
	corresponding to the parabolic subgroup $W_\p=S_{n-1}\times S_1$ of $S_n$,
	the derived shuffling functors $\LSh{s_i}$
	are instances of Seidel and Thomas' spherical twist functors.
	Namely, we show that certain parabolic indecomposable projectives $P^\p(w)$
	are spherical objects,
	and the associated twist functors are naturally isomorphic to $\LSh{w}\hShift{1}$
	as auto-equivalences of $\Db(\CatO^\p)$.
	
	We give an overview of the main properties of the BGG category $\CatO$,
	the construction of shuffling and spherical twist functors,
	and give some examples how to determine images of both.
	To this end, we employ the equivalence of blocks of $\CatO$ 
	and the module categories of certain path algebras.
\end{abstract}

\thispagestyle{empty}
\enlargethispage*{4cm}
\tableofcontents

\pagebreak

\section{Introduction}
Consider a finite dimensional semisimple complex Lie algebra $\mathfrak g$
with Cartan subalgebra $\mathfrak h$ that gives rise to a root system $\Phi$
and a root space decomposition $\mathfrak g = \bigoplus_{\alpha \in \Phi} \mathfrak g_\alpha$.
A choice of simple roots $\Delta$ fixes the positive roots $\Phi^+$,
the corresponding subalgebra $\mathfrak n \coloneqq \bigoplus_{\alpha \in \Phi^+} \mathfrak g_\alpha$
and the corresponding Borel subalgebra $\mathfrak b = \mathfrak h \oplus \mathfrak n$.

Representations of $\mathfrak g$ are equivalent to modules
over the universal enveloping algebra $U(\mathfrak g)$ \autocite[§V]{Humphreys:Lie}.
The \emph{BGG category} $\CatO$ of $\mathfrak g$
is the full subcategory of $\Mod{U(\mathfrak g)}$
consisting of modules that
\begin{thmlist*}[label=($𝓞$\arabic*)]
	\item are finitely generated,
	\item have a weight space decomposition $M=⨁_{𝜆∈𝔥^*} M_𝜆$ and
	\item are locally $𝔫$-finite; \ie, for every $v∈M$, the orbit $U(𝔫^+)v$ is finite dimensional.
\end{thmlist*}

\subsection{Category \texorpdfstring{$\CatO$}{𝓞}, blocks and shuffling functors}
Denote the Weyl group of $\mathfrak g$ by $W$.
The half-sum of positive roots $\rho \coloneqq \frac{1}{2}\sum_{\alpha \in \Phi^+} \alpha$
gives rise to the dot-action $w\cdot\colon \lambda \mapsto w(\lambda+\rho)-\rho$.
The $\Complex$-span of $\Phi$ is a Euclidean space
and its inner product $(-,-)$ fixes the coroots $\check{\alpha} \coloneqq \frac{2\alpha}{(\alpha,\alpha)}$
and the \term{fundamental weights} $\varpi_{\alpha}$ for $\alpha \in \Delta$
defined by $\frac{(\varpi_\alpha, \beta)}{(\beta, \beta)}$ for $\alpha, \beta \in \Delta$.
We call a weight $\lambda \in \Complex\Phi$ \emph{integral}
if $\frac{2(\lambda, \alpha)}{(\alpha,\alpha)} \in \mathbf{Z}$
and denote by $\Lambda$ the set of all integral weights.
We call a $\mathbf{Z}_{\geq 0}$-linear combination of the $\varpi_{\alpha}$'s \emph{dominant}
and denote the set of all dominant weights by $\Lambda^+$.
We call a weight $\lambda$ \term{$\rho$-dominant}
if $\frac{2(\lambda + \rho, \alpha)}{(\alpha, \alpha)} \notin \mathbf{Z}_{<0}$ for all $\alpha \in \Delta$.
To a weight $\lambda$ we associate the subgroup
$W_\lambda \coloneqq \langle s_\alpha \mid \frac{2(\lambda + \rho, \alpha)}{(\alpha,\alpha)} \in \mathbf{Z} \rangle \leq W$;
See \autocites{Humphreys:CatO}{Jantzen:Moduln-mit-hoechstem-Gewicht} for details.

The category $\CatO$ has a decomposition
$\CatO=\bigoplus_{\lambda}\CatO_\lambda$ into \term{blocks} $\CatO_\lambda$,
indexed by the $\rho$-dominant weights $\lambda$,
each of which consists of the modules of highest weight in $W_\lambda\cdot\lambda$
\autocite[thm\ 4.9]{Humphreys:CatO}%
\footnote{Humphreys indexes blocks by $\rho$-antidominant weights, which results in slightly different formulation for some statements.}.
In particular, Each block $\CatO_\lambda$ contains the simple modules $L(w\cdot \lambda)$,
the indecomposable projectives $P(w\cdot\lambda)$
and the Verma modules $M(w\cdot\lambda)$
of highest weight $w\in W/W_\lambda$.
If a block $\CatO_\lambda$ is fixed,
we just write $L(w)$, $P(w)$ and $M(w)$ for the respective objects therein.
Each block $\CatO_\lambda$ is Morita equivalent 
to modules over a quasi-hereditary algebra \autocite{BGG}.

A weight $\lambda$ is called \term{regular} 
if its \term{stabilizer subgroup} $W_\lambda$ \wrt\ the dot-action is trivial;
\ie, if $\lambda$ does not lie on any reflection plane.
All blocks $\CatO_\lambda$ associated to regular weights are equivalent as categories;
in the following we shall thus work in the block $\CatO_0$
containing the trivial $\mathfrak g$-representation $L(e\cdot 0)=\Complex$,
which is called the \term{principal block}.

\begin{definition}
	A \emph{Coxeter system} consists of a group $W$, a fixed set $S$ of generators and a presentation 
	$W=⟨s∈S \mid s^2 = e, sts\dotsm = tst\dotsm ⟩$ with $m_{st}$ factors $s,t$ on both sides.
	The $s\in S$ are called \term{simple reflections}.
	The matrix $(m_{st})_{s,t\in W}$ is called the \term{Coxeter matrix} of $W$.
	To $W$, there is the associated \term{braid group}
	$B_W=⟨s∈S \mid sts\dotsm = tst\dotsm⟩$,
	such that there is a natural quotient map $B_W\onto W$.
	A finite Coxeter system has a \term{length function} $\ell\colon W \to \mathbf{N}_0$,
	which assigns to an element $w \in W$ the length of a shortest expression for $w$
	in terms of $S$.
	With respect to $\ell$, there is a unique longest element $w_0 \in W$.
\end{definition}

\begin{example}
	The Weyl group of $\mathfrak g$ is a Coxeter group.
	In particular, the symmetric group $S_n$,
	which is the Weyl group of $\SL_n$,
	is a Coxeter group, 
	generated by the simple reflections $s_1,\dotsc,s_{n-1}$.
	Its Coxter matrix has entries
	$m_{s_i,s_j} = \begin{smallcases}
	1 & \text{if $i=j$}\\ 3 & \text{if $\lvert i-j\lvert = 1$}\\2 & \text{otherwise}
	\end{smallcases}$.
	For the symmetric group, $B_n ≔ B_{S_n}$ is the well-known Artin braid group.
\end{example}

For weights $\lambda, \mu \in \mathfrak{h}^*$ with $\lambda - \mu \in \Lambda$,
there is a unique $\nu \in W(\mu-\lambda)$ (\wrt\ the ordinary $W$-action)
such that the simple module $L(\nu)$ is finite-dimensional,
which is the case if and only if $\nu$ is dominant
\autocite[thm.\ 1.6]{Humphreys:CatO}.

\begin{definition}[{\autocite[54]{Jantzen:Moduln-mit-hoechstem-Gewicht}}]
	The \term{translation functor} $T_\lambda^\mu\colon \CatO_\lambda \to \CatO_\mu$
	assigns to $M$ the direct summand of $M \otimes L(\nu)$ lying in $\CatO_\mu$.
\end{definition}

The functor $T_\lambda^\mu$ is exact, preserves projectives, commutes with duality,
and is biadjoint to $T_\mu^\lambda$.
From now on, we assume that $\lambda, \mu$ are integral.
For most of what we need, this is stricter than necessary,
but sufficient for our needs;
see \autocite[§7]{Humphreys:CatO} for a more general treatment
and for an overview of properties of $T_\lambda^\mu$.

\begin{definition}
	Let $s\in W$ be a simple reflection
	and $\mu$ be an integral weight with stabiliser $W_\mu=\{e,s\}$.
	The \term{translation through the $s$-wall} is the composition 
	$\Theta_s\coloneqq T_\mu^0 T_0^\mu\colon \CatO_0 \to \CatO_0$.
\end{definition}

As notation suggests, $\Theta_s$ is independent of the choice of $\mu$.
It is an exact self-adjoint auto-equivalence of the block $\CatO_0$
\autocite[§2.10]{Jantzen:Moduln-mit-hoechstem-Gewicht}.
It is uniquely determined by the existence of short exact sequences
\begin{equation}
	\label{eqn:translation-defining-ses}
	0\to M(w)\to \Theta_s M(w) \to M(ws) \to 0
	\quad\text{and}\quad
	\Theta_s M(w)\isom\Theta_s M(ws)
\end{equation}
for $w<ws$ \autocite[Satz 2.10]{Jantzen:Moduln-mit-hoechstem-Gewicht}.
Furthermore, $\Theta_s^2 = \Theta_s ⊕ \Theta_s$ \autocite[§7.14]{Humphreys:CatO}.

\begin{definition}[{\autocites[§2]{Carlin:Extensions-of-Vermas}[§3]{Irving:Shuffled-Verma-Modules}}]
	\label{def:shuffling-functor}
	From the adjunctions $T_\mu^0 \dashv T_0^\mu$ and $T_0^\mu \dashv T_\mu^0$ 
	we get adjunction maps $\unit_s\colon \id \Rightarrow \Theta_s$ 
	and $\counit_s\colon \Theta_s \Rightarrow \id$,
	and we define the mutually adjoint \term{shuffling} and the \term{coshuffling functor} $\Sh{s}$ and $\Csh{s}$ as
	\begin{equation*}
		\Sh{s}  ≔ \coker(\unit_s)\colon \CatO_0 \to \CatO_0,
		\qquad
		\Csh{s} ≔ \ker(\counit_s)\colon \CatO_0 \to \CatO_0.
	\end{equation*} 
\end{definition}

In particular, \cref{eqn:translation-defining-ses} implies that for $w < ws$,
\begin{equation}
	\label{eq:shuffling-on-Vermas}
	\Sh{s} M(w) = M(ws) \quad \text{and} \quad \Csh{s} M(ws) = M(w).
\end{equation}
By the snake lemma,
$\Sh{s}$ is right exact and $\Csh{s}$ is left exact,
so we can consider the mutually inverse \autocite[thm.\ 5.7]{stroppel:translation-and-shuffling} derived functors
$\LSh{s}$ and $\RCsh{s}$.

In general, we can see an adjunction morphism $\unit_s$ on an abelian category $\mathcal C$
as a functor from $\mathcal C$ to the arrow category $\mathcal C^{[1]}$.
By the snake lemma, the cokernel is a right exact functor functor $\mathcal C^{[1]} \to \mathcal C$,
whose derived functor $\mathbf L \coker\colon \Db(\mathcal C^{[1]}) \to \Db(\mathcal{C})$
is the mapping cone \autocite[thm\ 3.5.6]{groth:derivators}.%
\footnote{%
	As the inclusion $\Db(\mathcal C^{[1]}) \into \Db(\mathcal C)^{[1]}$ is no equivalence in general,
	one has to pay attention to the fact that
	the mapping cone is a functor on coherent diagrams $\Db(\mathcal C^{[1]})$,
	but not necessarily on general diagrams $\Db(\mathcal C)^{[1]}$.
}
We thus obtain that $\LSh{s} \isom\cone(\unit_s)$ 
and dually that $\RCsh{s}\isom\cocone(\counit_s) = \cone(\counit_s)\hShift{1}$.
The snake lemma also implies that $\mathbf L^i\Sh{s} = 0\ \text{for $i\neq 0,1$}$ 
and $\mathbf R_j \Csh{s} = 0\ \text{for $j\neq -1,0$}$.

\begin{definition}
	We say that an object $M \in \mathcal C$ is $F$-acyclic
	for a left or right exact functor $F\colon \mathcal C \to \mathcal D$ of abelian categories
	if $\mathbf R_i F M = 0$ or $\mathbf L^i M = 0$ for $i \neq 0$, respectively.
\end{definition}

Projective (respectively, injective) objects are acyclic for every right (resp., left) exact functor.
Since the adjunction unit $\unit_s\colon M(w) \to \Theta_s$ 
is injective for all $w \in W$ \autocite[thm\ 12.2]{Humphreys:CatO},
every $M(w)$ is both $\Sh{s}$-acyclic.

\subsection{\texorpdfstring{Actions of $W$ on $K_0(\CatO_0)$ and $\Db(\CatO_0)$}{Actions of 𝑊 on 𝐾₀(𝓞₀) and Dᵇ(𝓞₀)}}
\begin{definition}
	The \term{Grothendieck group} $K_0(\mathcal C)$ of an abelian category $\mathcal C$
	is the abelian group generated by symbols $[M]$ for the objects $M\in\mathcal C$,
	subject to the relation $[E]=[A]+[B]$ 
	whenever $E$ is an extension of $A$ by $B$.
	Similarly, the Grothendieck group $K_0(\mathcal T)$ of a triangulated category $\mathcal T$
	is the abelian group generated by symbols $[M]$ for the objects $M\in\mathcal{T}$,
	subject to the relation $[E]=[A]+[B]$ whenever $A \to E \to B \to A\hShift{-1}$ is a triangle.
\end{definition}

An exact (respectively, triangulated) functor induces a group homomorphism
on $K_0(\mathcal C)$ (respectively, $K_0(\mathcal T)$).
The inclusion $\mathcal C \into \Db(\mathcal C)$
induces an isomorphism $K_0(\mathcal C)\isom K_0(\Db(\mathcal C))$ 
of abelian groups \autocite{Grothendieck:Tohoku}.

For the category $\CatO_0$, the group $K_0(\CatO_0)$ is free,
and each of the collections $\{L(w)\}_{w\in W}$, $\{M(w)\}_{w\in W}$ and $\{P(w)\}_{w\in W}$
is a $\mathbf Z$-basis of it.
For the simple modules, this follows from the fact
that every object an $\CatO_0$ has a composition series
(see \cref{sec:composition-series}),
and for the others from the fact that each $[L(w)]$
can be written uniquely in terms of the $M(u)$'s or $P(v)$'s,
due to the composition series of the latter.

The shuffling functors induce a right action of $W$ on $K_0(\CatO_0)$,
given for a $w \in W$ and a simple reflection $s \in W$ by the assignment
\begin{equation}
	W \longto \Aut(K_0(\CatO_0)),\quad
	s \longmapsto [\Sh{s}]\colon [M(w)] \mapsto [M(ws)];
\end{equation}
indeed, we even have $\Sh{s} M(w) = M(ws)$ for $w \leq ws$
and from \cref{eqn:translation-defining-ses} and from the $\Sh{s}$-acyclicity of Verma modules
we learn that $M(w)$ and $\Sh{s}(ws)$ have the same composition factors
such that $[M(w)] = \sum_{v} [M(w) : L(v)] [L(v)] = [M(ws)]$;
hence, the above action is well-defined on $W$.
This action actually comes from an action on the derived category $\CatO_0$:

\begin{theorem}[{\autocite[thm.\ 10.4]{rouquier:categorification}}]
	\label{thm:roquier-braid-group-action}
	The assignment 
	\begin{equation}
		B_W\to\Aut\bigl(\Db(\CatO_0)\bigr),\quad
		s\mapsto\LSh{s},\quad 
		s^{-1}\mapsto\RCsh{s}
	\end{equation}
	defines a weak action of $B_W$ on the derived category $\Db(\CatO_0)$,
	meaning that for simple reflections $s,t \in W$,
	there are natural isomorphisms
	$\LSh{s} \LSh{t} \isom \LSh{t} \LSh{s}$ whenever $st = ts$
	and $\LSh{s} \LSh{t} \LSh{s} \isom \LSh{t} \LSh{s} \LSh{t}$ whenever $sts = tst$,
	and the following square commutes:
	\begin{equation*}
		\begin{tikzcd}
			B_W \ar[d, "\LSh{(-)}"'] \ar[r, "\can", two heads]  & W \dar["{[\Sh{(-)}]}"] \\
			\operatorname{Auteq}(\Db(\CatO_0)) \ar[r, "{[-]}"'] & \Aut(K_0(\CatO_0)).
		\end{tikzcd}
	\end{equation*}
\end{theorem}

\subsection{Seidel and Thomas' spherical twist functors}
Let $k$ be a field and $\mathcal C$ be a $k$-linear abelian category $\mathcal C$ of finite global dimension.
Seidel and Thomas have constructed
an action of the braid group $B_n\coloneqq B_{S_n}$ of the symmetric group $S_n$ on $\mathcal C$
in terms of \term{spherical objects} \autocite{seidel-thomas:braid-group-actions}.
We give a short summary of their construction.

\begin{definition}
	Let $X, Y\in\Ch(\mathcal C)$ be chain complexes.
	We employ the convention that the homological shift upwards; \ie, $X\hShift{1}_j \coloneqq X_{j-i}$.
	We define the graded $k$-vector space $\Hom^*_{\Db(\mathcal C)}(X,Y)$
	with graded components $\Hom^j_{\Db(\mathcal C)}(X,Y) \coloneqq \Hom_{\Db(\mathcal C)}(X\hShift{j}, Y)$.
	
	We denote by $\hom^\bullet_\mathcal C(X, Y) \in \Ch(\VS{k})$ the chain complex with homological components
	$\hom^j_\mathcal C(X, Y) \coloneqq \Hom_k(X\hShift{j}, Y)$ the graded $k$-linear morphisms $X\to Y$ 
	(not necessarily chain maps),
	and differential $\mathrm d_{\hom}$
	defined on homogeneous elements of degree $j$
	by $\mathrm d_{\hom}(f)=\mathrm d_Y f-(-1)^j f \mathrm d_X$.
	Its cohomology $H^j \hom^\bullet_\mathcal C(X,Y) = \Hom_{\Kb(\mathcal C)}(Y\hShift{j}, Y)$
	is the (honest) the Hom-space in the bounded homotopy category $\Kb(\mathcal C)$ \autocite[§2.7.5]{Weibel:Hom-Alg}.
\end{definition}

\begin{definition}
	Let $X \in \Ch(\mathcal C)$ and $V \in \Ch(\VS{k})$.
	Since $\mathcal C$ is a $k$-linear category,
	we can regard $X$ as a complex of vector spaces,
	on which a $\lambda \in k$ acts  by $\lambda x \coloneqq (\lambda\id_X)(x)$ for $x \in X$.
	We define the complex $\lin^\bullet(V, X) \in \Ch(\mathcal C)$ 
	of \term{$k$-linear maps} from $V$ to $X$
	with differential $(\mathrm d_{\lin}f) v \coloneqq (-1)^{\deg v}[\mathrm dfv - f\mathrm dv]$
	and the \term{tensor product} complex $V \otimes X$
	with differential $\mathrm d_{V \otimes X}(v \otimes x) \coloneqq \mathrm dv \otimes x + (-1)^{\deg v} v \otimes \mathrm dx$ for homogeneous $v$.
\end{definition}

\begin{remark}
	\label{rmk:lin-and-tensor}
	For a homogeneous basis $B$ of $V$ and $v \in V$,
	let $\beta_b(v)$ be the coefficients such that
	$v = \sum_{b \in B} \beta_b(v) b$.
	There are homogeneous isomorphisms of graded vector spaces
	\begin{align*}
		V \otimes X           &\stackrel\cong\longto \textstyle\bigoplus_{b \in B} X\hShift{-\deg b}_b,&
		\lin^\bullet(V, X)    &\stackrel\cong\longto \textstyle\prod_{b \in B} X \hShift{\deg b}_b, \\
		v\otimes x            &\longmapsto \textstyle\sum_b \beta_b(v) x_b &
		f                     &\longmapsto (f(b))_b;
	\end{align*}
	The differential on $V \otimes X$ and $\lin^\bullet(V, X)$ is specified by
	\begin{align*}
		\mathrm d_{V \otimes X}(x_b) &= 
			\textstyle\sum_{b' \in B} \beta_{b'}(\mathrm db) x_{b'} + (-1)^{\deg b} (\mathrm d x)_b,\\
		\mathrm d_{\lin^\bullet(V, X)}\bigl((x_b)_{b \in B}\bigr) &= \textstyle
			\bigl((-1)^{\deg b} \bigl(\mathrm dx_b - (\sum_{b' \in B}\beta_{b}(\mathrm db') x_b) \bigr) \bigr)_{b \in B};
	\intertext{
		if $V$ is degree-wise finite dimensional,
		the complex $\lin(V,X) \cong V^*\otimes X$
		is spanned by maps $x_b \colon b' \mapsto \delta_{bb'} x$
		(where $\delta$ denotes the Kronecker-$\delta$),
		and we get
	}
		\mathrm d_{\lin^\bullet(V, X)}(x_b) &= \textstyle
			(-1)^{\deg b} \bigl( (\mathrm dx_b) - \sum_{b' \in B}\beta_{b}(\mathrm db') x_{b'}\bigr).
	\end{align*}
\end{remark}

\begin{definition}
	\label{def:spherelike-and-spherical}
	\Needspace{3\baselineskip}
	An object $E\in\Db(\mathcal C)$ is called \term{$d$-spherelike} for $d\geq 0$
	if
	\begin{thmlist}[label={(S\arabic*)}, ref={\thetheorem(S\arabic*)}]
		\item\label{def:spherical-object:finite-total-dimension}
			for any $F∈\Db(𝓒)$, 
			the spaces $\Hom^*_{\Db(𝓒)}(E, F)$ and $\Hom^*_{\Db(𝓒)}(F, E)$ are of finite total dimension, and
		\item\label{def:spherical-object:spherical-cohomology}
			$\Hom^*_{\Db(𝓒)}(E, E) ≅ k[x]/(x^2)$ as graded algebras,
			where $x$ is a morphism of degree $d$.
	\end{thmlist}
	It is called \term{$d$-spherical} if it also has the \term{Calabi-Yau-property}: 
	\begin{thmlist}[resume*]
		\item\label{def:spherical-object:calabi-yau}
			For all $F$ and $i$,
			composition of morphisms gives a non\-/degenerate pairing\\
			$\Hom^i_{\Db}(F, E) ⊗ \Hom^{d-i}_{\Db}(E, F) \to \Hom^d_{\Db}(E, E) \to kx$.%
			\footnote{
				If $d\neq 0$, we already have $\Hom^d_{\Db}(E, E) = kx$,
				but if $d=0$, we have $\Hom^d_{\Db}(E, E) = kx \oplus k\id_E$.
				In this case, we mean that 
				$\circ\colon\Hom^i_{\Db}(F, E) ⊗ \Hom^{d-i}_{\Db}(F, E) \to \Hom^d_{\Db}(E, E)/k\id_E ≅ k$
				be non\-/degenerate.
			}
	\end{thmlist}
\end{definition}

\begin{definition}
	\label{def:twist-cotwist}
	For $\Xi_E F ≔ \hom^\bullet_\mathcal C(E, F)⊗E $ 
	and $\Xi'_E F \coloneqq \lin^\bullet(\hom^\bullet_\mathcal C(F, E), E)$,
	we consider the evaluation and coevaluation maps
	\begin{align*}
		\ev\colon \Xi_E F &\to F,\ \phi⊗f  \mapsto\phi(f); &
		\ev'\colon F &\to \Xi'_E F,\ f \mapsto\phi\mapsto\phi(f).
	\end{align*}
	If $E$ is $d$-spherelike,
	we define its \term{spherical twist} and \emph{cotwist functor}
	\begin{align*}
			T_E &≔ \cone(\ev\colon\Xi_E \Rightarrow\id_\mathcal C), &
			T'_E &≔ \cocone(\ev'\colon\id_\mathcal C\Rightarrow\Xi'_E);
	\end{align*}
	here, the grading of the (co)cone is such that $\id_\mathcal C$
	is in degree zero.
\end{definition}
Isomorphic objects in $\Db(\mathcal C)$ give rise
to naturally isomorphic twist and cotwist functors \autocite[prop.\ 2.6]{seidel-thomas:braid-group-actions},
hence $T_E$ and $T'_E$ are well-defined.
There is an adjunction $T'_E\dashv T_E$ \autocite[lem.\ 2.8]{seidel-thomas:braid-group-actions}.

\begin{remark}
	\label{rmk:evaluation-in-terms-of-shifted-copies}%
	\label{rmk:induced-map-in-terms-of-shifted-copies}%
	In terms of the descriptions of $\lin^\bullet$ and $\otimes$
	\wrt\ fixed bases $A$ and $B$ of $\hom^\bullet_\mathcal{C}(E,F)$
	and $\hom^\bullet_\mathcal{C}(F,E)$ from \cref{rmk:lin-and-tensor},
	the evaluation and coevaluation map are the homogeneous maps
	\begin{align*}
		\ev\colon \Xi_E F = \bigoplus_{\mathclap{a \in A}} E \hShift{\deg a}_a & \longto F
		& \ev' \colon F & \longto  \prod_{\mathclap{b \in B}} F \hShift{-\deg b}_b = \Xi'_E F
		\\
		(e_a)_a &\longmapsto \textstyle \sum_a a(e_a)
		& v &\longmapsto (b(v))_b.
	\end{align*}
	Similarly, a morphism $f\colon F \to F'$ of complexes induces maps
	\begin{align*}
		  f_*\colon \Xi_E F &\longto \Xi_E F'
		& f_*\colon \Xi'_E F &\longto \Xi'_E F'
		\\
		g\otimes e &\longmapsto fg\otimes e &
		\psi &\longmapsto h \mapsto \psi(hf), \\
	\intertext{which correspond to}
		 \bigoplus_{a \in A}E\hShift{\deg a}
		&\longto \bigoplus_{\mathclap{a' \in A'}} E\hShift{\deg a'}
		&\prod_{\mathclap{b \in B}} E\hShift{-\deg a}
		&\longto \prod_{\mathclap{b' \in B'}} E\hShift{-\deg b'}
		\\
		e_a &\longmapsto \textstyle\sum_{a' \in A'} (\beta_{a'}(af) e_{a})_{a'}
		& (e_{b})_{b}
		&\longmapsto \bigl(\textstyle\sum_{b \in B} \beta_b(b'f) e_b\bigr)_{b'}.
	\end{align*}
\end{remark}

\begin{proposition}[{\autocite[prop.\ 2.10]{seidel-thomas:braid-group-actions}}]
	If $E$ is $d$-spherical, then $T_E$ and $T'_E$ are mutually inverse auto-equivalences of $\Db(\mathcal C)$.
\end{proposition}

\begin{notation}
	\label{notation:total-complexes}
	Given a double (or triple) complex $X^{\bullet\bullet}$, we denote its $⊕$-total complex by $\{X^{\bullet\bullet}\}$.
	In particular, since we can regard a morphism $f\colon X\to Y$ of chain complexes
	trivially as a double complex with $Y$ placed in degree zero,
	the mapping cone can be written as $\cone(f)=\{f\colon X\to Y\}$.
	We occasionally write a zero below the complex placed in degree zero, so $\cone(f)=\{f\colon X\to \degZero{Y}\}$.
\end{notation}

\begin{remark}
	\label{rmk:twisting-E-by-itself}
	Since for every $d$-spherical object $E$,
	$\hom^\bullet_\mathcal C(E, E) \isom ⟨\id⟩_k⊕⟨x⟩_k \hShift{d}$ as complexes of vector spaces,
	we have
	\begin{equation*}
		T'_E E = \{\begin{psmallmatrix}1\\x\end{psmallmatrix}\colon \degZero{E}\to E⊕E \hShift{-d}\} \qis E\hShift{1-d}.
	\end{equation*}
\end{remark}

\begin{definition}
	A collection $\{E_1,\dotsc, E_n\}$ of $d$-spherical objects is an \term{$\mathrm{A}_n$-configuration}
	if $\dim\Hom^*_{\Db(\mathcal C)}(E_i, E_j) = \begin{smallcases}1 & \text{if $|i-j|=1$},\\0 & \text{otherwise}\end{smallcases}$.
\end{definition}

\begin{theorem}[{\autocite[thm.\ 2.17]{seidel-thomas:braid-group-actions}}]
	Given an $\mathrm{A}_n$-configuration $\{E_1,\dotsc, E_n\}$ of $d$-spherical objects,
	the assignment
	\begin{equation*}
		B_n\to\operatorname{Auteq}(\Db(\mathcal C)),\quad s_i \mapsto T_{E_i}
	\end{equation*}
	defines a weak action of the braid group.
\end{theorem}

\subsection{Objective}
The category $\CatO_0$ is a $\Complex$-linear category of finite global dimension
and hence satisfies the requirements of \autocite{seidel-thomas:braid-group-actions}.
For $\mathfrak g=\SL_n$, we want to understand
whether there is an $\mathrm{A}_n$-configuration in $\Db(\CatO_0)$
such that the associated twist functors relate to the shuffling functors.

We shall show in \cref{thm:shuffling-is-twisting:sl2}
that for $\mathfrak g = \SL_2$, derived shuffling functors are in instance
of spherical twists for suitable objects,
and we shall develop the following generalisation.

For a parabolic subalgebra $\p \subseteq \mathfrak{sl}_n$
corresponding to a parabolic subgroup $W_\p \leq S_n$,
we write $\CatO^\p_0$ for the principal block
of the parabolic category $\CatO$,
see \cref{sec:parabolic-subalgebras}.
We shall prove the following:

\begin{theorem*}[{\ref{thm:main-result}}]
	For a maximal parabolic subalgebra $\p$ of $\mathfrak{sl}_n$
	corresponding to the parabolic subgroup $S_{n-1}×S_1\leq S_n$,
	the derived category $\Db(\CatO^\p_0)$ contains
	an $\mathrm{A}_{n-1}$-configuration of $0$-spherical objects
	such that the associated spherical twist functors
	and the restriction of $\LSh{s_i}\hShift{1}$ to $\Db(𝓞^𝔭_0)$
	are naturally isomorphic auto\-/equivalences of $\Db(𝓞^𝔭_0)$.
\end{theorem*}

\subsection*{Outline}
\Cref{sec:tools-for-O} gives an overview of some of the most important properties
of blocks of $\CatO$
and the tools we employ.
We do not require any prior knowledge about $\CatO$.
We include a short refresher on Kazhdan-Lusztig theory, quivers and graded algebras.

In \Cref{sec:sl2-case}, we explain how to compute images of the shuffling functors
for the special case $\mathfrak g=\SL_2$ explicitly.
The proof of \cref{thm:shuffling-is-twisting:sl2} 
will serve as a model for the proof of the more general \cref{thm:main-result},
which is worked out in \cref{sec:sln-case}.

\section{Some theory of the category \texorpdfstring{$\CatO_\lambda$}{𝓞}}
\label{sec:tools-for-O}
In this section, we collect the most important properties of $\CatO_\lambda$.
For simplicity, we always assume that $\lambda$ is an integral $\rho$-dominant weight.

\subsection{Composition series}
\label{sec:composition-series}
Every module $M\in\CatO_\lambda$ admits a composition series
\begin{equation*}
	0\subsetneq M_1\subsetneq\cdots\subsetneq M_\ell=M
\end{equation*}
with subquotients $M_i/M_{i-1}$ isomorphic to the simple modules $L(w\cdot \lambda)$ 
\autocite[Satz 1.13]{Jantzen:Moduln-mit-hoechstem-Gewicht}.
According to the Jordan-Hölder theorem, any composition series for $M$
involves the same isomorphism classes of the simple factors,
up to their order of appearance.
The multiplicity of $L(w\cdot \lambda)$ in any composition series for $M$
is denoted by $[M:L(w\cdot\lambda)]$.
There is another important filtration:

\begin{definition}
	\label{def:loewy-filtration}
	A \term{Loewy filtration} of a module $M\in\CatO_0$ is a filtration
	of minimal length with semisimple subquotients $M_i$.
	A module $M$ is called \emph{rigid} if it has a unique Loewy filtration;
	for instance, Verma modules are rigid \autocite[thm.\ 1]{Irving:socle-filtration}.
	The \term{socle} $\soc M$ and the \term{head} $\operatorname{hd} M$ of a module $M$
	are its maximal semisimple submodule and quotient, respectively.
\end{definition}

\begin{notation}
	We write Loewy filtrations diagrammatically 
	by putting all simple summands of each semisimple subquotient into a common horizontal layer,
	from the socle in the bottom line to the head in the top line.
\end{notation}

\begin{example}
		For $\mathfrak g=\SL_2$, the non-simple Verma module has a composition series (and also Loewy filtration)
		$M(e)=\begin{psmallmatrix}L(e)\\L(s)\end{psmallmatrix}$,
		which is just another way to say that there is a short exact sequence
		$0\to\nolinebreak[2] L(s)\to M(e)\to L(e)\to 0$.
		
		For $\mathfrak g=\SL_3$, 
		the Verma module $M(s)=\begin{psmallmatrix}L(s)\\L(st)\quad L(ts)\\L(w_0)\end{psmallmatrix}$
		has the simple quotient $L(s)$ and the simple submodule $L(w_0)$.
		Its non-trivial submodules have Loewy filtrations
		$L(w_0)$,
		$\begin{psmallmatrix}L(st)\\L(w_0)\end{psmallmatrix}$, $\begin{psmallmatrix}L(ts)\\L(w_0)\end{psmallmatrix}$
		and $\begin{psmallmatrix}L(st)\quad L(ts)\\L(w_0)\end{psmallmatrix}$.
\end{example}
\begin{remark}
	For a morphism $f\colon M \to N$ in $\CatO$
	there exist composition series of $M$ and $N$
	contiguously containing a composition series of $\im f$
	at the top and the bottom, respectively.
	The cokernel and kernel of $f$ comprise the composition factors
	of $M$ and $N$ not belonging to $\im f$, respectively.
	For example, for $\mathfrak g = \mathfrak{sl}_2$,
	we can write the short exact sequence $0 \to M(e) \to P(s) \to M(s) \to 0$ diagrammatically as
	\begin{equation*}%
		\begin{matrix}%
			&& L(s) & \onto & \smash{\underbrace{L(s)}_{M(s).}}\\
			\underbrace{ \begin{array}{@{}c@{}} L(e)\\L(s) \end{array} }_{M(e)}%
			& \into & 
			\underbrace{ \begin{array}{@{}c@{}} L(e)\\ L(s) \end{array} }_{P(s)}%
		\end{matrix}%
	\end{equation*}
	Since $\dim\Hom_\CatO(L(v), L(w))=\delta_{vw}$,
	a morphism is determined by a scalar
	for each pair of composition factors from its domain and its codomain.
\end{remark}

Certain modules, such as the indecomposable projectives $P(w\cdot\lambda)$,
also admit a \term{standard filtration} whose subquotients are isomorphic to Verma modules $M(v\cdot\lambda)$.
We write $(P(w\cdot\lambda) : M(v\cdot\lambda))$ for the (unique) multiplicity
of $M(v\cdot\lambda)$ in a standard filtration of $P(w\cdot\lambda)$.

\begin{theorem}[\thmtitle{BGG reciprocity} \autocite{BGG}]
	$(P(v\cdot\lambda):M(w\cdot\lambda))=[M(w\cdot\lambda):L(v\cdot\lambda)]$.
\end{theorem}

\subsection{Kazhdan-Lusztig theory}
\label{sec:kl-theory}
There is a $W^2$-parametrised collection $\{p_{vw}\}$
of polynomials in $\mathbf Z[q^{\pm 1}]$,
called \term{Kazhdan-Lusztig polynomials},
which occur as base change coefficients between the standard basis and the Kazhdan-Lusztig basis
(usually denoted, respectively, by $\{T_w\}$ and $\{C_w\}$ for $w \in W$)
of the \term{Iwahori-Hecke algebra} $\mathsf{H}_q(W)$
\autocites{KL:Representations-of-Coxeter-groups}{Soergel:KL-polynomials},
which is a $\mathbf{Z}[q^{\pm 1/2}]$-algebra and a deformation of the group algebra $\mathbf{Z}W$.
The following relation to composition factor multiplicities
has been conjectured in \autocite[conjecture 1.5]{KL:Representations-of-Coxeter-groups}:

\begin{theorem}[\thmtitle{Kazhdan-Lusztig conjecture} {\autocites[§4]{BB:localisation-de-g-modules}[§8]{BK:KL-conjecture}}]
	\label{thm:BGG-reciprocity}
	The composition factor multiplicities in a block $\CatO_\lambda$
	for a regular weight $\lambda$ are given by
	\begin{equation}
		\label{eqn:BGG-reciprocity}
		(P(v\cdot\lambda):M(w\cdot\lambda)) = [M(w\cdot\lambda):L(v\cdot\lambda)] = p_{vw}(1).
	\end{equation}
\end{theorem}

\begin{remark}
	\label{rmk:conventions-kl-polynomials}
	If one indexes blocks by $\rho$-antidominant rather than $\rho$-dominant weights,
	\cref{eqn:BGG-reciprocity} takes the form $[M(w\cdot\lambda):L(v\cdot\lambda)] = p_{w_0v,w_0w}(1)$.
	Depending on whether one uses the basis elements $C_w$ or $C'_w$ for the Kazhdan-Lusztig basis%
	\footnote{See \autocite[after thm.\ 1.1]{KL:Representations-of-Coxeter-groups} for their relation.},
	$T_w$ or $H_w = q^{\ell(w)/2} T_w$ for the natural basis,
	$q$ or $v = q^{-1/2}$ for the formal indeterminate,
	and $p_{vw}$ or $h_{vw} = q^{-2(\ell(w) - \ell(v))} p_{vw}$ for the Kazhdan-Lusztig polynomials,
	the precise statement exhibiting the $p_{vw}$'s as base change coefficients may look different
	(see also \autocite[2.6]{Soergel:KL-polynomials}).
\end{remark}

\subsection{Gradings}
\label{sec:gradings}
We summarise how to pass from $\CatO_\lambda$
to a graded category $\gCatO_\lambda$.

The category $\CatO_\lambda$ has a projective generator $P_\lambda ≔ \bigoplus_{w\in W/W_\lambda} P(w\cdot\lambda)$,
for which we define the algebra $A_\lambda \coloneqq \End_\CatO(P_\lambda)$.
By Morita's theorem \autocite[Thm.\ II.1.3]{Bass:K-theory},
there is an equivalence of categories
\begin{equation}
	\label{eqn:morita-eq}
	\CatO_\lambda\xrightarrow{\simeq}\Mod{}[A_\lambda],\quad M \mapsto\Hom_\CatO(P_\lambda, M).
\end{equation}

There exists a natural grading on $A_\lambda$, given as follows
\autocite[\S\S 1.3, 2]{stroppel:gradings-and-translation}.

\begin{theorem}[\thmtitle{Struktursatz}{ \autocite[thm.\ 2]{Soergel:perverse-Garben}}]
	\label{thm:struktursatz}
	The functor $\mathbf V_\lambda\colon\CatO_\lambda \to\Mod{}[\End_\CatO(P(w_0 \cdot \lambda))]$,
	called the \emph{combinatoric functor}, is fully faithful on projectives.
\end{theorem}

The algebra $\End_\CatO(P(w_0 \cdot \lambda))$ can be described explicitly:
The ordinary (non-dotted) action of the Weyl group $W$ on $\mathfrak{h}$
gives rise to an action of $W$ on the symmetric algebra $\Complex[\mathfrak h^*]$.
Let $(\Complex[\mathfrak h^*]_+^W)$ denote its ideal generated 
by $W$-invariant polynomials without constant term
and $C≔\Complex[\mathfrak h^*]/(\Complex[\mathfrak h^*]_+^W)$ 
be the \emph{coinvariant algebra} of $\mathfrak g$.

\begin{theorem}[\thmtitle{Endomorphismensatz} {\autocite[thm.\ 3]{Soergel:perverse-Garben}}]
	$\End_\CatO(P(w_0\cdot\lambda))$ is isomorphic to the invariant subalgebra $C^{W_\lambda}$.
\end{theorem}

The images $\mathbf{V}P(w\cdot\lambda)$ of the indecomposable projectives
are particular direct summands of the $C$-module $C \otimes_{C^{s_{i_n}}} \cdots \otimes_{C^{s_{i_1}}} \Complex$
for a reduced expression $w = s_{i_1} \dotsm s_{i_n}$
\autocite[thm.\ 10]{Soergel:perverse-Garben}.

By putting $\mathfrak{h}^*$ in degree two, $C[\mathfrak h^*]$ becomes an evenly graded algebra.
Since $(\Complex[\mathfrak h^*]_+^W)$ is a homogeneous ideal,
also $\End_\CatO(P(w_0\cdot\lambda)) \isom C$ is a graded algebra.
Tensor products and direct summands of graded modules with indecomposable complement inherit a grading, 
hence the images $\mathbf{V} P(w\cdot\lambda)$
and the spaces $\Hom_{\CatO}(P(v\cdot\lambda), P(w\cdot\lambda)) \isom \Hom_C(\mathbf{V}P(v\cdot\lambda), \mathbf{V}P(w\cdot\lambda))$ are graded as well,
which establishes a grading on $A_\lambda \isom \End_C(\mathbf{V} P_\lambda)$
\autocite[thm\ 2.1]{stroppel:gradings-and-translation}.
The equivalence in \eqref{eqn:morita-eq} motivates the following:

\begin{definition}
	Let $\gCatO_\lambda ≔ \Mod[g]{}[A_\lambda]$ 
	be the category of graded $A_\lambda$-modules.
\end{definition}

\begin{notation}
	We denote the $k$-th graded component of a graded $A$-module $M$ by $M_k$.
	We employ the convention that the grading shift $\langle-\rangle$ shifts upwards;
	\ie, $M⟨i⟩_k ≔ M_{k-i}$.
\end{notation}

\begin{remark}
	\label{rmk:graded-lifts}
	A module $M\in \CatO_\lambda$ is \term{gradable} if there is a graded module $\tilde M\in\gCatO_\lambda$
	such that forgetting the grading yields $\mathrm f\tilde M=M$.
	In particular, simple, Verma and indecomposable projective modules
	are gradable \autocite[§§2f.]{stroppel:gradings-and-translation}.
	Up to grading shifts and isomorphisms, an indecomposable module has a unique graded lift.
	In the following, we shall not distinguish notationally between these modules and their graded lifts.
\end{remark}

We may fix the shift for the graded lifts of $M(w)$, $L(w)$ and $P(w)$
such that each of these lifts has its non-zero graded component
of least degree in degree zero.
For a graded module $M$
the multiplicity $[M_i : L(w)]$ of $L(w)$ in $M_i$, seen as ungraded modules,
therefore is precisely the multiplicity $[M : L(w)⟨i⟩]$ of graded modules.

Since $A_\lambda$ is non-negatively graded,
a graded module $M$ is filtered by the submodules $M_{\geq k} \coloneqq \bigoplus_{l \geq k} M_l$,
and its graded components $M_k = M_{\geq k}/M_{\geq k+1}$ are subquotients.
For Verma and projective modules, these are precisely the Loewy layers
\autocites[302]{stroppel:gradings-and-translation}[§1.1]{Irving:socle-filtration}.
The graded component a simple composition factor belongs to
is encoded in the exponents of Kazhdan-Lusztig polynomials:

\begin{theorem}[\thmtitle{generalised KL-theorem} {
	\autocites[thm.\ 2]{Irving:socle-filtration}%
		[cor.\ 7.1.3]{Irving:filtered-category-OS}%
		[thm.\ 3.1.4%
		\protect\footnote{
			Blocks by are indexed by dominant weights,
			$n_{xy}$ is used for Soergel's $h_{xy}$.
		}]{BGS:Koszul-Duality}%
}]
	\label{thm:generalised-kl}
	The composition factor multiplicities
	in the graded components of Verma modules satisfy
	\begin{equation}
		\label{eq:generalised-kl-theorem}
		p_{vw}=\textstyle \sum_\ell [M(v)_{\ell} : L(w) ] q^{(\ell(w) - \ell(v) - \ell)/2};
	\end{equation}
	in other words, each summand $q^k$ of $p_{vw}$ corresponds to a factor $L(w)$
	in the $\ell = (\ell(w) - \ell(v) - 2k)$-th layer of the Loewy filtration for $M(v)$, 
	with the zeroth layer at the top. 
\end{theorem}

\begin{remark}
	If one prefers to work with Soergel's $h_{vw}$
	(see \cref{rmk:conventions-kl-polynomials}), one gets
	$h_{vw} = \sum_\ell [M(v)_\ell : L(w)] v^\ell$.
	If one indexes blocks of $\CatO$ by antidominant weights instead,
	$v$ and $w$ have to be replaced by $w_0v$ and $w_0w$, respectively,
	and \cref{eq:generalised-kl-theorem} becomes
	$p_{w_0v,w_0w} = \sum_\ell [M(w_0v)_{\ell} : L(w_0w) ] q^{(\ell(w_0v) - \ell(w_0w) - \ell)/2}$.
	This is equivalent to \cref{eq:generalised-kl-theorem}
	since $w_0^2 = e$ and $\ell(w_0w) = \ell(w_0) - \ell(w)$ for all $w$.
\end{remark}
A graded analogue of the BGG reciprocity theorem \ref{thm:BGG-reciprocity} holds
\autocite[thm.\ 7.6]{stroppel:gradings-and-translation}.
\begin{definition}
	\label{def:graded-translation}
	On $\gCatO_\lambda$,
	the graded \term{translation through the $s$-wall} $\Theta_s$
	is uniquely defined by short exact sequences
	\begin{equation}
		\label{eq:definition-of-graded-translation}
		0\to M(w)⟨1⟩ \xto{\unit}\Theta_s M(w) \to M(ws)\to 0
		\quad\text{and}\quad
		\Theta_s M(ws) = \Theta_s M(w)⟨-1⟩
	\end{equation}
	for $w<ws$.
	Both adjunction maps are of degree one,
	and $\Theta_s^2 \isom \Theta_s⟨-1⟩ ⊕ \Theta_s⟨1⟩$.
	Analogously to the non-graded case (see \cref{def:shuffling-functor}),
	the graded (co)shuffling functors $\Sh{s}$ and $\Csh{s}$ are defined by
	$\Sh{s} M  = \coker(\unit\colon \id⟨1⟩ \Rightarrow \Theta_s)$ and 
	$\Csh{s} M = \coker(\Theta_s \Rightarrow \id⟨-1⟩)$.
\end{definition}

\begin{remark}
	The grading shift $\langle 1\rangle$ renders
	the Grothendieck group $K_0(\CatO^\mathbf Z_0)$ becomes a $\mathbf Z[v^{\pm 1}]$-module
	by $v[M]≔[M⟨1⟩]$.
	With $v = q^{-1/2}$,
	there is an isomorphism of $𝐙[v^{±1}]$-modules
	\begin{equation*}
			K_0\bigl(𝓞^\mathbf{Z}_0(\SL_n)\bigr) \to  \mathsf{H}_v(S_n),\quad
			[M(w)⟨v⟩]	\mapsto vH_w,\quad
			[P(w)⟨v⟩]	\mapsto vC_w,
	\end{equation*}
	for the Iwahori-Hecke algebra $\mathsf{H}_v(S_n)$ (see \cref{sec:kl-theory}).
	Explicitly, the basis $\{H_w\}_{w \in W}$ is subject to the relations
	$H_w H_s = H_{ws}$ and $H_{ws} H_s = H_w - (v - v^{-1})H_{ws}$ for $w < ws$.
	
	From the defining short exact sequence of $\Theta_s$
	in \cref{eq:definition-of-graded-translation} we see for $w < ws$ that
	\begin{align*}
		[\Theta_s M(w)] &= v[M(w)] + [M(ws)]\\
		[\Sh{s} M(w)] &= [\Theta_s M(w)] - [M(w)⟨v⟩] = M(ws)\\
		[\Sh{s} M(ws)] &= [\Theta_s M(w)⟨-1⟩] - [M(ws)⟨1⟩] = [M(ws)] - (v-v^{-1})[M(w)],
	\end{align*}
	hence $[\Theta_s]$ and $[\Sh{s}]$ respectively correspond
	to the right multiplication $\cdot(v + H_s)$ and $\cdot H_s$
	under the above isomorphism
	\autocite[thm.\ 7.1]{stroppel:gradings-and-translation}
\end{remark}

\subsection{Quivers}
\label{sec:quivers}
Recall that we defined the algebra $A_\lambda \coloneqq \End_{\CatO}(P_\lambda)
= \End_{\CatO}(\bigoplus_w P(w\cdot\lambda))$
in \cref{sec:gradings}.
This algebra $A_\lambda$ arises as the path algebra of a quiver as follows.

\begin{definition}
	The \term{Ext-quiver} $Q_A$ associated to a basic algebra $A$
	over an alg.\ closed field
	has vertices corresponding to the isoclasses of simple $A_\lambda$-modules $L$
	and $\dim \Ext^1(L, L') $ many arrows $L' \from L$.
\end{definition}

\begin{remark}
	Equivalently, for a complete set $\{e_i \mid i \in I\}$ of primitive orthogonal idempotents of $A$,
	$Q_A$ can be defined as the quiver
	with vertex set $I$ corresponding to the idempotents
	and $\dim (e_i (\rad A)e_j / e_i(\rad^2 A) e_j)$ many arrows $i \from j$
	\autocites[prop.\ 1.14]{ARS:1995}[prop.\ 2.4.3]{Benson:1995}[lem.\ 2.12]{ASS:2006}.
\end{remark}

By Gabriel's theorem 
\autocites[thm.\ 4.1.7]{Benson:1995}[thm.\ 3.7]{ASS:2006}[§4.3]{Gabriel:Auslander-Reiten-Sequences},
there exists an admissible ideal $\mathfrak{a}$
such that $k Q_A / \mathfrak{a} \cong A$.
Under this isomorphism,
the idempotent $e_i$ of $A$ corresponds to the trivial path $\trivpath{i}$ in $kQ_A$.
More generally, every finite dimensional algebra over an alg.\ closed field
is Morita-equivalent to a basic algebra \autocite[cor.\ 6.10]{ASS:2006}
and thus to a path algebra quotient.

For the basic $\Complex$-algebra $A_\lambda$
with complete set $\{\id_{P(w)}\}$ of primitive orthogonal idempotents,
we denote the Ext-quiver quiver by $Q_\lambda \coloneqq Q_{A_\lambda}$
and the admissible ideal by $\mathfrak{a}_\lambda$,
such that $A_\lambda \cong Q_\lambda / \mathfrak{a}_\lambda$.

\begin{notation}
	We denote the composition of arrows $v_2\xleftarrow{a}v_1$ and $v_3\xleftarrow{b}v_2$ of a quiver by
	$v_3\xleftarrow{ba}v_1$.
	We denote trivial the path associated to a vertex $v$ by $\trivpath{v}$.
\end{notation}

The path algebra of any quiver $Q_\lambda$ is non-negatively graded by the length of a path in terms of arrows.
One can show that $\mathfrak a_\lambda$ is a homogeneous ideal;
hence $A_\lambda$ is graded as well.
This grading coincides with the grading induced by $\mathbf V_\lambda$ 
\autocite{stroppel:gradings-and-translation} that we explained in \cref{sec:gradings}.

To summarise, we have that $\CatO_\lambda\simeq\Mod{}[A_\lambda]$ and $\CatO_\lambda^\mathbf Z\simeq\Mod[g]{}[A_\lambda]$.
The equivalence $\CatO_\lambda\simeq\Mod{}[A_\lambda]$ maps indecomposable projectives $P(w)$ in $\CatO_\lambda$
to the indecomposable projectives in $\Mod{}[A_\lambda]$;
these are precisely the right ideals $\trivpath{w} A_\lambda$ of all paths ending in $w$
\autocite[cor.\ 4.18, rmk.\ 4.20]{Barot:Rep-Theory}.

One can choose a $\Complex$-basis of the projective module $\trivpath{w} A_\lambda$
consisting of paths ending in $w$,
each of which corresponds to a simple composition factor $L(v)⟨i⟩$
for its source vertex $v$ and length $i$.
Hence as a vector space, $\trivpath{w}A$ is a direct sum
$\trivpath{w}A = \bigoplus_{v \in W} \trivpath{w}A\trivpath{v}$ of graded vector spaces,
whose $i$-th graded component has dimension
$\dim (\trivpath{w}A\trivpath{v})_i = [P(w)_i : L(v)⟨i⟩]$.
	
From now on, we shall stick to the principal block $\CatO_0$
and omit all the subscript-$\lambda$'s from $A$ and $Q$.

\subsection{Parabolic subalgebras}
\label{sec:parabolic-subalgebras}
Let $(W,S)$ be a Coxeter system
and $S_\p\subseteq S$ be any subset of the simple reflections of $W$.

\begin{definition}
	The associated \term{parabolic subgroup} $W_\p\leq W$
	is the subgroup $W_\p=⟨S_\p⟩$ of $W$.
	Every left coset in $W_\p\backslash W$ has a unique representative of minimal length \autocite[§1.10]{Humphreys:Coxeter-Groups}.
	We denote the set of such representatives by $W^\p$.
\end{definition}

To the quiver $Q$ we associate the full subquiver $Q^\p$ of $Q$
with vertex set $Q^\p=W^\p$
and define the respective algebra
$A^\p≔A/(\trivpath{v})_{v\in W_\p}$.
The category $\CatO^\p_0 ≔ \Mod{}[A^\p]$
is equivalent to the smallest Serre subcategory of $\CatO^\p_0$
containing all simple modules $L(w)$ for $w\in W^\p$
\autocite[§2]{kildetoft-mazorchuk:parabolic-projective-functors}.
The quotient map $A\onto A^\p$ induces an induction-restriction-adjunction
\begin{equation}
	\Ind^{\p}\colon \CatO_0 \xtofrom{\;\dashv\;}  \CatO^\p_0 \noloc \Res_{\p},\quad
	M \mapsto M \otimes_A A^\p, \quad
	N_A \mapsfrom N,
\end{equation}
where $N_A$ denotes the module $N$ on which $A$ acts via $A\onto A^\p$.
The functor $\Res_\p$ is fully faithful
and thus renders $\CatO^\p_0$ a subcategory of $\CatO_0$.
It has a left adjoint $Z^\p≔\Ind^\p$, called \term{Zuckerman functor},
and a right adjoint $Z_\p$, called \term{dual Zuckerman functor},
which respectively assign to a module $M$ its largest quotient and its largest submodule
that have simple composition factors $L(w)$ for $w \in W^\p$ 
\autocite[Thm.\ 6.1]{mazorchuk:categorification}.
Let $P^\p(w)≔Z^\p(P(w))=\trivpath{w}A^\p$ and $M^\p(v) ≔ Z^\p(M(v))$.
As notation suggests, the $P^\p(w)$'s are the indecomposable projectives in $\CatO^\p_0$ \autocite[§4.6]{mazorchuk:categorification}.

\begin{remark}
	The constructions and statements from \crefrange{sec:composition-series}{sec:quivers}
	have parabolic analogues; namely:
	\begin{thmlist}
		\item The category $\CatO^\p_0$ has a projective generator $P^\p=\bigoplus_{w\in W^\p}P^\p(w)$.
		\item The analogously constructed graded version $\CatO_0^{\mathbf Z,\p}$ of $\CatO_0^\p$
			contains graded lifts of simples, parabolic Vermas and indecomposable projectives.
		\item The $P^\p(w)$'s have parabolic standard filtrations with subquotients isomorphic to $M^\p(v)$'s.
		\item The respective multiplicities satisfy a parabolic BGG reciprocity theorem
			\autocite[Prop.\ 4.5, Thm.\ 6.1]{RC:splitting.criteria}.
		\item Parabolic Kazhdan-Lusztig polynomials $p^\p_{vw}\in\mathbf Z[q^{\pm 1}]$ 
			occur as base change coefficients for parabolic Hecke modules
			\autocites[§3]{Deodhar:parabolic-KL}[rmk.\ 3.2]{Soergel:KL-polynomials}.
		\item The generalised Kazhdan-Lusztig theorem \ref{thm:generalised-kl}
			has the following parabolic analogue
			for graded lifts of simple, parabolic Verma and indecomposable projective modules
			\autocites[Cor.\ 7.1.3]{Irving:filtered-category-OS}[Thm.\ 1.3]{CC:parabolic-KL}:
			\begin{equation}
				p^\p_{vw}
				= \sum_{k\geq 0}(P^\p(w)_\ell : M^\p(v)) q^{(\ell(w) - \ell(v) - \ell)/2} = \sum_{k\geq 0} [M^\p(v)_\ell : L(w)] q^{(\ell(w) - \ell(v) - \ell)/2}.
			\end{equation}
			In other words, each summand $q^k$ corresponds,
			for $\ell - (\ell(w) - \ell(v) - 2k)$,
			to an $M^\p(v)\langle\ell\rangle$ in $P^\p(w)$
			and an $L(w)\langle\ell\rangle$ in $M^\p(v)$,
			with the top of the modules in degree zero.
	\end{thmlist}
\end{remark}

\begin{remark}
	\label{rmk:parabolic-translation-and-shuffling}
	The functor $Z^\p$ commutes with projective functors,
	in particular with $\Theta_s$
	\autocite[Thm.\ 6.1]{mazorchuk:categorification}.
	This implies that both $\Theta_s$ and $\Sh{s}$ restrict to $\CatO^\p_0$
	and $\CatO^{\mathbf{Z},\p}_0$,
	and the restriction of $\Theta_s$ is uniquely characterised 
	by the short exact restrictions
	\begin{equation}
		\label{eq:ses-parabolic-wall-crossing}
		0 \to M^\p(w)⟨1⟩ \xto{\unit} \Theta_s \to M^\p(ws)\to 0
		\quad\text{and}\quad
		\Theta_s M^\p(ws) = \Theta_s M^\p(w)⟨-1⟩,
	\end{equation}
	for $w<ws$,
	of the short exact sequences \cref{eqn:translation-defining-ses} defining $\Theta_s$.
\end{remark}

\begin{caveat}
	\label{caveat:parabolic-projectives-not-projective}
	The inclusion $\CatO^\p_0\subseteq \CatO_0$ does not preserve projectives,
	as we shall see in \cref{rmk:inclusion-of-parabolic-subalgebradoesnt-preserve-projectives}.
\end{caveat}

\section{\texorpdfstring{$B_n$-actions for $\SL_2$}{𝐵n-actions for 𝔰𝔩₂}}
\label{sec:sl2-case}
Consider the Lie algebra $\mathfrak{g}=\SL_2$ and its Weyl group $S_2=\{e, s\}$.
Since the Kazhdan-Lusztig polynomials for $S_2$ are
$p_{vw} = \begin{smallcases}
	1 & \text{if $v \leq w$,}\\
	0 & \text{otherwise,}
\end{smallcases}$
the Verma modules and the indecomposable projectives
have composition series and standard filtrations
\begin{equation}
	P(e) 
		= M(e) 
		= \begin{psmallmatrix}L(e)\\L(s)\end{psmallmatrix},
	\qquad
	P(s) 
		= \begin{psmallmatrix}M(s)\\M(e)\end{psmallmatrix} 
		= \begin{psmallmatrix} L(s)\\L(e)\\L(s)\end{psmallmatrix},
	\qquad M(s) = L(s).
\end{equation}
The category $\CatO_0$ thus is equivalent to $\Mod{}[A]$
for the path algebra quotient $A=\Complex Q/(ba)$ 
of the quiver $Q = (a\colon e \rightleftarrows s \noloc b)$.%
The arrows $a,b$ of $Q$ and their relations
correspond to the unique (up to scalars) non-trivial morphisms $a\colon P(e)\into P(s)$ and $b\colon P(s)\to P(e)$
\autocite{stroppel:gradings-and-translation}.

\subsection{The action of \texorpdfstring{$\Sh{s}$}{Shₛ}}
\label{sec:shuffling-sl2}
We obtain the images of the indecomposable projectives and Verma modules under $\Sh{s}$
as follows:
\begin{enumerate}
	\item $M(e)=P(e)$: from the short exact sequence \cref{eqn:translation-defining-ses} we obtain that
		$\Theta_s M(e) = \begin{psmallmatrix}M(s)\\M(e)\end{psmallmatrix}=P(s)$
		and $\Sh{s} M(e)= M(s)$.
	\item $M(s)$: Up to scalars, there is a unique morphism $M(s) \into M(e)$;
		hence we obtain that
		$\Sh{s} M(s)
		\isom \begin{psmallmatrix}M(s)\\M(e)/M(s)\end{psmallmatrix} 
		=     \begin{psmallmatrix}L(s)\\L(e)\end{psmallmatrix}=M^\vee(s)$,
		\ie, the \term{dual Verma module};
		see \autocite[\S 3.2]{Humphreys:CatO} for the definition.
	\item $P(s)$: using $\Theta_s^2\isom\Theta_s\oplus\Theta_s$, 
		it follows from $P(s)=\Theta_s M(e)$
		that $\Sh{s} P(s) \isom P(s)$.
\end{enumerate}
From the naturality diagram
\begin{equation}
		\label{eqn:adj-of-Ps}
	\begin{tikzcd}[ampersand replacement=\&]
		P(e) \rar["a"]\dar["\eta_{P(e)}"']\& P(s) \rar["b"]\dar["\Mtrx{1\\x}", "\eta_{P(s)}"'] \& P(e) \dar["\eta_{P(s)}"]\\
		P(s) \rar["\Mtrx{1\\0}"] \& P(s)⊕P(s) \rar["\Mtrx{0 & 1}"]\rar \& P(s)
	\end{tikzcd}
\end{equation}
of $\eta\colon\id\Rightarrow\Theta_s$ we get that
the morphisms $a\colon P(e)\into P(s)$ and $b\colon P(s)\to P(e)$ have images 
\begin{equation}
	\label{eq:shuffling-sl2:images-of-generating-morphisms}
	\Sh{s} a\colon M(s)\into P(s)\quad\text{and}\quad\Sh{s} b\colon P(s)\onto M(s).
\end{equation}
Similar arguments show that $\Csh{s} P(s) \isom P(s)$
and $\RCsh{s} P(e)$ is the mapping cone $\RCsh{s} P(e)\qis\{\degZero{P(s)} \to P(e)\}$
(recall the notation of mapping cones from \cref{notation:total-complexes});
so we have
\begin{equation}
	\label{eq:shuffling-sl2:imaegs-of-projectives}
	\begin{aligned}
		\LSh{s} P(e) &= \{P(e)\to \degZero{P(s)}\} \simeq M(s)  &
		\RCsh{s} P(e) &= \{\degZero{P(s)}\to P(e)\} \\
		\LSh{s} P(s) &= P(s) &
		\RCsh{s} P(s) &= P(s),
	\end{aligned}
\end{equation}
and in terms of the projective resolution $\{P(e) \to P(s)\}$ of $M(s)$,
the morphisms from \cref{eq:shuffling-sl2:images-of-generating-morphisms} become
\begin{align}
	\label{eq:shuffling-sl2:images-of-generating-morphisms:lifts}
	\LSh{s} a&\colon \left\{
		\begin{tikzcd}[cramped, sep=small, baseline=(ab.base), ampersand replacement=\&]
			\{P(e) \rar \& P(s) \dar["ab" name=ab]\} \\
			\& P(s),
		\end{tikzcd}
	\right. &
	\LSh{s} b&\colon \left\{
		\begin{tikzcd}[sep=small, cramped, baseline=(id.base), ampersand replacement=\&]
			\& P(s) \dar["\id" name=id]\\
			\{P(e) \rar \& P(s)\},
		\end{tikzcd}
	\right.\\[1ex]
	\RCsh{s} a&\colon \left\{
		\begin{tikzcd}[cramped, sep=small, baseline=(ab.base), ampersand replacement=\&]
			\{P(s) \dar["\id" name=id] \rar \& P(s) \} \\
			P(s),
		\end{tikzcd}
	\right.
	&
	\RCsh{s} b&\colon \left\{
		\begin{tikzcd}[sep=small, cramped, baseline=(id.base), ampersand replacement=\&]
			P(s) \dar["ab" name=ab]\\
			\{P(s) \rar \& P(s) \}.
		\end{tikzcd}
	\right.
\end{align}
Since $\CatO_0$ has finite global dimension \cite[prop.\ 6.9]{Humphreys:CatO},
$\Db(\CatO_0)$ is equivalent to the bounded homotopy category of projectives
(cf.\ \autocite[thm.\ 10.4.8]{Weibel:Hom-Alg})
and therefore generated as a triangulated category
by the indecomposable projectives;
\ie, the smallest triangulated subcategory of $\Db(\CatO_0)$
containing the indecomposable projectives
is $\Db(\CatO_0)$ itself \autocite[\S 3.3]{Krause:Derived-Categories}.
These above data thus suffices to determine $\LSh{s}$.
Phrasing the observations differently, we have shown the following:
\begin{proposition}
	The action on $\Db(\CatO_{0,\SL_2})$ of the Artin braid group $B_1$
	(which is an infinite cyclic group)
	by $s \mapsto \LSh{s}$ and $s^{-1} \mapsto \RCsh{s}$
	is given explicitly by
	\begin{align*}
		B_1 &\longto \Aut(\Db(\CatO_0))\\
		s^n &\longmapsto
			\underset{\strut}{
			\begin{aligned}[t] 
				& P(s)\mapsto P(s),\\
				& P(e)\mapsto 
					\begin{cases}
						\bigl\{P(e)\into \smash[t]{\overbrace{P(s)\xto[ab]{} P(s)\to \dotsb\to \degZero{P(s)}}^{\text{$n$ times}}}\bigr\} & \text{if $n\geq 0$},\\
						\bigl\{\smash{\underbrace{\degZero{P(s)}\xto{ab} P(s)\to \dotsb\to P(s)}_{\text{$-n$ times}}} \to P(e) \bigr\} & \text{if $n\leq 0$,}
					\end{cases}
			\end{aligned}
			}
	\end{align*}
	with homological degree $0$ as indicated.
\end{proposition}

\begin{example}
	To illustrate why this proposition is clear from what we derived before,
	let us compute the action of, say, $s^2$, which acts by $\LSh{s}^2$.
	In $\Db(\CatO_0)$, the module $M(s)$ is isomorphic
	to its projective resolution $\{P(e) \xto{a} P(s)\}$,
	to which we apply $\LSh{s}^2$.
	From the images in \cref{eq:shuffling-sl2:imaegs-of-projectives,eq:shuffling-sl2:images-of-generating-morphisms}
	we get that $\LSh{s} M(s)$ is isomorphic to the double complex
	\begin{equation*}
		\LSh{s} M(s) \simeq
		\left\{
			\begin{tikzcd}[sep=small, slightly cramped, baseline=(b.base)]
				P(e) \dar["b" name=b] \\
				P(s) \rar{\id} & P(s)
			\end{tikzcd}
		\right\}
		\simeq 
		\{P(e) \to P(s) \to \degZero{P(s)}\}.
	\end{equation*}
	Applying $\LSh{s}$ once again shows that
	$\LSh{s} M(s) \simeq \{P(e) \to P(s) \to P(s) \to \degZero{P(s)}\}$.
\end{example}

\subsection{Spherical objects}
\label{sec:spherical-objects-sl2}
To check that a spherelike object is spherical indeed,
\ie, that the composition pairing from the Calabi-Yau-property
(\cref{def:spherical-object:calabi-yau}) is non-degenerate,
it suffices to check non-degeneracy \wrt\ indecomposable projectives:

\begin{lemma}
	Let $\mathcal{C}$ be an abelian category of finite global dimension
	and $E \in \Db(\mathcal{C})$ a $d$-spherelike object.
	If the composition pairing
	$\Hom^i_{\Db}(P, E) ⊗ \Hom^{d-i}_{\Db}(P, E) \to kx$
	is non-degenerate for all $i$
	and all indecomposable projective objects $P \in \mathcal{C}$,
	seen as complexes in $\Db(\mathcal{C})$ concentrated in a single degree,
	then
	$\Hom^i_{\Db}(F, E) ⊗ \Hom^{d-i}_{\Db}(F, E) \to kx$
	is non-degeneate for all $i$ and $F \in \Db(\mathcal{C})$.
\end{lemma}
\begin{proof}
	Since $\mathcal{C}$ is of finite global dimension,
	every object of $\Db(\mathcal C)$ is quasi-isomorphic to a bounded complex of projectives;
	in other words, for the full subcateogires
	$S_0 \coloneqq \Proj{\mathcal C}$
	and $S_s \coloneqq \{\cone f \mid f \in S_{s-1}\}$ of $\Db(\mathcal C)$,
	we have $\Db(\mathcal C) = \bigcup_{s \geq 0} S_s$.
	Assume for some $s \geq 0$
	that the pairing $\Hom^i_{\Db}(F, E) ⊗ \Hom^{d-i}_{\Db}(F, E) \to kx$ is non-degenerate
	whenever $F \in S_s$.
	For a distinguished triangle $F' \xto{f} F \to F'' \to F'\hShift{-1}$
	with $F, F' \in S_s$,
	we thus get that in the diagram
	\begin{equation*}
		\begin{tikzcd}[sep=small, cramped]
			\cdots \rar & \Hom_{\Db}(E, F') \dar{\cong} \rar & \Hom_{\Db}(E, F) \dar \rar & \Hom_{\Db}(E, F'') \dar{\cong} \rar & \cdots\\
			\cdots \rar & \Hom_{\Db}(F', E\hShift{-d})^* \rar & \Hom_{\Db}(F, E\hShift{-d})^* \rar & \Hom_{\Db}(F'', E\hShift{-d})^* \rar & \cdots
		\end{tikzcd}
	\end{equation*}
	whose rows are long exact sequences,
	the indicated vertical maps are isomorphisms.
	By the five lemma, it follows that also
	$\Hom^i_{\Db}(F'', E) ⊗ \Hom^{d-i}_{\Db}(F'', E) \to kx$ is non-degenerate.
	The claim follows by induction.
\end{proof}
  
\begin{lemma}
	$P(s)$ is a $0$-spherical and $L(e)$ is a $2$-spherical object
	of $\Db(\CatO_0)$ for $\SL_2$.
\end{lemma}
\begin{proof}
	To see that $P(s)$ is $0$-spherical,
	note that the non-trivial endomorphism $x\coloneqq ab$ of $P(s)$
	satisfies $x^2 = 0$ and $\End_{\CatO}(P(s))\isom\Complex[x]/(x)^2$,
	so $P(s)$ is 0-spherelike.
	Since $P(e)$ and $P(s)$ generate $D^\mathrm b(\mathcal O_0)$
	and since $x$ factors through $P(e)$,
	the composition pairing is non-degenerate, hence $P(s)$ is spherical.
	
	To see that $L(e)$ is $2$-spherical, consider its projective resolution
	$\{P(e)\to P(s)\to P(e)\}$.
	Apart from the identity, the only non-trivial graded endomorphisms 
	of this resolution are (scalar multiples of)
	\bgroup
	\tikzcdset{
		every diagram/.append style={
		sep=0.6em, cramped, font=\footnotesize, baseline=(B.base),
		every cell/.append style={inner sep=0.1ex}
	}}
	\begin{gather}
			\biggl(\,
				\begin{tikzcd}[ampersand replacement=\&]
					P(e) \rar \& P(s) \dar["\strut" name=B] \rar \& P(e) \dar \ar[dl, equal]\\
					\& P(e) \rar \& P(s) \rar \& P(e)
				\end{tikzcd}
			\,\biggr), 
			\biggl(\,
				\begin{tikzcd}[ampersand replacement=\&]
					\& P(e) \rar \dar["\strut" name=B] \& P(s) \dar \ar[dl, equal] \rar \& P(e) \\
					P(e) \rar \& P(s) \rar \& P(e)
				\end{tikzcd}
			\,\biggr),
			\biggl(
			\underbrace{\,
				\begin{tikzcd}[ampersand replacement=\&]
					\&\& P(e) \rar \dar[equal, "\strut" name=B] \& P(s) \rar \& P(e), \\
					P(e) \rar \& P(s) \rar \& P(e)
				\end{tikzcd}
			\,}_{{}\eqqcolon x}
			\biggr)
	\end{gather}
	\egroup
	of which the first two are null-homotopic
	with the indicated chain homotopies.
	Therefore, $\End_{\Db(\CatO)}(L(e)) = \Complex[x]/(x^2)$ with $x$ of degree two,
	which renders $L(e)$ 2-spherelike.
	The endomorphism $x$ factors through $P(e)$,
	and $\Hom^*_{\Db(\CatO)}(P(s), L(e)) = 0 = \Hom^*_{\Db(\CatO)}(L(e), P(s))$,
	hence the composition pairing is non-degenerate.
\end{proof}

\begin{remark}
	Simple modules in $\CatO_0$ only have trivial self-extensions \autocite[Prop.\ 3.1]{Humphreys:CatO}
	and thus cannot be $0$ or $1$-spherical.
\end{remark}

\subsection{Spherically twisting by \texorpdfstring{$P(s)$}{P(s)}}

The images of projectives under the cotwisting functor $T'_{P(s)}$ are
\begin{equation*}
	\begin{alignedat}{4}
		T'_{P(s)}P(s) &= \{\Mtrx{1\\x}\colon \degZero{P(s)} → P(s)⊕P(s)\}  
			&&\qis P(s)\hShift{1} 
			&&\stackrel{\text{\cref{eq:shuffling-sl2:imaegs-of-projectives}}}{=} 
			  \Sh{s} P(s)\hShift{1},\\
		T'_{P(s)}P(e) &= \{a\colon \degZero{P(e)} → P(s)\} 
			&&\qis M(s)\hShift{1} 
			&&\stackrel{\phantom{\text{\cref{eq:shuffling-sl2:imaegs-of-projectives}}}}{=} 
			  \Sh{s} P(e)\hShift{1}.
	\end{alignedat}
\end{equation*}
This is a first step towards a proof of the following:
\begin{theorem}%
	\label{thm:shuffling-is-twisting:sl2}
	There is a natural isomorphism $\LSh{s}\hShift{1}\isom T'_{P(s)}$ of autoequivalences of $\Db(𝓞_0(\SL_2))$.
\end{theorem}
\begin{lemma}[\thmtitle{Morita}]
	Let $A$ and $B$ be rings.
	Any right exact functor $F\colon  \Mod{}[A] → \Mod{}[B]$ that preserves arbitrary direct sums
	is isomorphic to tensoring with the $A$-$B$-bimodule $FA$
	\autocite[Thm.\ II.2.3]{Bass:K-theory}.
\end{lemma}
\newcommand{\FBim}[1]{M_{#1}}
\begin{corollary}
	\label{cor:morita-equivalence-and-induced-module}
	For abelian categories $𝓐$ and $𝓑$ with projective generators $P_𝓐$ and $P_𝓑$,
	Morita's theorem \eqref{eqn:morita-eq} 
	allows us to identify $𝓐$ and $𝓑$ with $\Mod{}[\End_𝓐(P_𝓐)]$ and $\Mod{}[\End_𝓑(P_𝓑)]$ respectively.
	Then any right exact functor $F\colon  𝓐 \to 𝓑$ that commutes with arbitrary direct sums 
	can be identified with the functor
	\begin{equation*}
		-⊗_{\End_𝓐(P_𝓐)} \FBim{F} \colon \Mod{}[\End_𝓐(P_𝓐)] \to \Mod{}[\End_𝓑(P_𝓑)],
	\end{equation*}
	where $\FBim{F} \coloneqq \Hom_𝓑(P_𝓑, FP_𝓐)$ 
	is an $\End_𝓐(P_𝓐)$-$\End_𝓑(P_𝓑)$-bimodule
	on which elements $a\in \End_𝓐(P_𝓐)$ and $b\in \End_𝓑(P_𝓑)$
	act by $a\ldot f\ldot b = Fa\circ f\circ b$
	for $f\in \FBim{F}$.
\end{corollary}

\begin{proof}[{Proof of \cref{thm:shuffling-is-twisting:sl2}}]
	By the corollary there are natural isomorpisms
	\begin{alignat*}{2}
		\LSh{s}\hShift{1} &\stackrel{\text{def}}{=} \{ \degZero{\id_𝓞} \Rightarrow 𝛩_s \} &&≅ -⊗_A \{ A \to \FBim{𝛩_s} \},\\*
		T'_{P(s)} &\stackrel{\text{def}}{=} \{ \degZero{\id_𝓞} \Rightarrow \Xi'_{P(s)} \} &&≅ -⊗_A \{ A \to \FBim{\Xi'_{P(s)}} \},
	\end{alignat*}
	such that it suffices to show $\FBim{𝛩_s} \isom \FBim{\Xi'_{P(s)}}$.
	Recall that $\Xi'_{P(s)} = \lin(\hom^\bullet_\CatO(-, P(s)), P(s))$.
	By finite dimensionality,
	there is an isomorphism $\FBim{\Xi'_{P(s)}} ≅ P(s)^* ⊗_𝐂 P(s)$ of $A$-$A$-bimodules.
	Consider the canonical vector space basis $\{\trivpath{s}, s\from e, s\from e\from s\}$ of $P(s)=\trivpath{s}A$
	and the dual basis of $P(s)^*$.
	A basis of $\FBim{\Xi'_{P(s)}}$ is given by
	the nine pairwise tensor products in the schematic
	\begin{equation}
		\label{eq:shuffling-twisting-isomorphic:sl2-tensor}
		M_{\Xi'} \isom P(s)^* ⊗_𝐂 P(s)\colon \quad
		\begin{tikzpicture}[mth, baseline=(B.base)]
			\matrix (A) [
				matrix of math nodes, 
				row sep=0mm, 
				column sep=-2mm, 
				text height=2ex, text depth=.5ex, 
				column 2/.style={anchor=base east},
				column 4/.style={anchor=base west}
			] {
					~& (s←e←s)^* 			&				& \trivpath{s}	&~\\
					~& |[alias=B]| (s←e)^* 	& \quad⊗\quad	& (s←e)			&~\\
					~& \trivpath{s}^*			&				& (s←e←s),		&~\\
				};
				\node[braced box={T}, fit=(A-1-2) (A-3-2)] {} ;
				\node[braced box'={T'},fit=(A-1-4) (A-3-4)] {} ;
				\draw
					(A-1-1) edge[|->, out=195, in=165] node[swap]{$(e←[b]s)_*$} (A-2-1.160)
					(A-2-1.200) edge[|->, out=195, in=165] node[swap]{$(s←[a]e)_*$} (A-3-1)
					(A-1-5)	edge[|->, out=-15, in=15] node{$∘(s←[a]e)$} (A-2-5.20)
					(A-2-5.340) edge[|->, out=-15, in=15] node{$∘(s←[b]e)$} (A-3-5);
		\end{tikzpicture}
	\end{equation}
	with the indicated action on basis vectors.
	To describe the $A$-$A$-bimodule action on $\FBim{\Theta_s} = \Hom_A(P, \Theta_s P)$
	in terms of a vector space basis,
	we introduce the following notation.
	Recall that $P=P(e)⊕P(s)$ and $\Theta_s P\isom P(s)^{⊕3}$.
	We enumerate the summands of $\Theta_s P = P(s)_1⊕P(s)_2⊕P(s)_3$.
	We then abbreviate, \eg, the morphism $\Mtrx{0&0\\0&0\\0&x} \in \Hom_A(P, \Theta_s P)$,
	by  $P(s)_3 \xleftarrow{x} P(s)$.
	For $\Theta_s$, the naturality diagram \eqref{eqn:adj-of-Ps} of $\eta\colon\id\Rightarrow\Theta_s$ 
	shows that the images under $\Theta_s$ of the morphisms $a$ and $b$ generating $\End_\mathcal{C}(P)$ are
	\settowidth{\algnRef}{${}⊕{}$}
	\begin{equation}
		\begin{tikzcd}[column sep={\the\algnRef}, nodes={inner xsep=0pt}, row sep=small]
			P(e)\rar[phantom, "⊕"] & 
			P(s)\ar[dl, "b"'] &[4em] 
			P(s)_1 \rar[phantom, "⊕"] & 
			P(s)_2 \rar[phantom, "⊕"] & 
			P(s)_3 \ar[dll, "\id_{P(s)}"' {near end, inner sep=0pt, outer sep=0pt}]
			\\
			P(e)\rar[phantom, "⊕"] \ar[dr, "a"'] & 
			P(s) \ar[phantom, "\xmapsto{\Theta_s}", r]& 
			P(s)_1 \rar[phantom, "⊕"] \ar[dr, "\id_{P(s)}"' {near start, inner sep=0pt, outer sep=0pt}] & 
			P(s)_2 \rar[phantom, "⊕"] & P(s)_3
			\\
			P(e)\rar[phantom, "⊕"] & 
			P(s) & 
			P(s)_1 \rar[phantom, "⊕"] & 
			P(s)_2 \rar[phantom, "⊕"] & 
			P(s)_3\mathrlap{.}
		\end{tikzcd}
	\end{equation}
	A vector space basis of $\FBim{𝛩_s}$ is given 
	by the nine ways to map a summand $P(w)$ of $P$ from the right of the following schematic
	to a summand $P(s)_i$ of $\Theta_s P(s)$ on the left of
	\begin{equation}
		\label{eq:shuffling-twisting-isomorphic:sl2-hom}
		\FBim{𝛩_s}\colon\quad
		\begin{tikzpicture}[mth, ampersand replacement=\&]
			\matrix (A) [
			matrix of math nodes, 
			row sep={3.5ex,between origins},
			column sep=10mm,
			column 1/.style={anchor=base east, text width=width("$P(s)_2$")},
			column 3/.style={anchor=base west, text width=width("$P(s)$")}
			] {
				P(s)_3 \&  \& P(s)\\
				P(s)_1 \& P(s) \& P(e)\\
				P(s)_2 \& \& P(s) \\
			};
			\draw[->] (A-2-2)
			edge[->, out=180, in=0] (A-1-1)
			edge[->, out=180, in=0] (A-2-1)
			edge[->, out=180, in=0] (A-3-1)
			edge[<-, in=180, out=0, "$\id$" very near end] (A-1-3)
			edge[-, in=180, out=0] (A-2-3)
			edge[-, in=180, out=0, "$x$"' very near end] (A-3-3);
			\draw
			(A-1-1.west)  edge[|->, out=195, in=165] node[swap]{$\Theta_s(P(e)\xleftarrow{b} P(s))\circ $} (A-2-1.170)
			(A-2-1.190)   edge[|->, out=195, in=165] node[swap]{$\Theta_s(P(s)\xleftarrow{a} P(e))\circ$} (A-3-1.west)
			(A-1-3.east)  edge[|->, out=-15, in=15] node{$∘(P(s)←[a]P(e))$} (A-2-3.10)
			(A-2-3.-10)   edge[|->, out=-15, in=15] node{$∘(P(e)←[b]P(s))$} (A-3-3.east);
		\end{tikzpicture}
	\end{equation}
	with the bimodule action as indicated.
	Comparison of \eqref{eq:shuffling-twisting-isomorphic:sl2-tensor} and \eqref{eq:shuffling-twisting-isomorphic:sl2-hom} shows
	that the obvious isomorphism $\FBim{\Theta_s} \isom \FBim{\Xi'_{P(s)}}$ of vector spaces
	is an isomorphism of $A$-$A$-bimodules.
\end{proof}

\subsection{Spherically twisting by \texorpdfstring{$L(e)$}{L(e)}}
\label{sec:twisting-by-Le}
We first note that $\LSh{s} L(e) = L(e)\hShift{-1}$;
namely, from the images in \cref{eq:shuffling-sl2:imaegs-of-projectives}
we get that the projective resolution $\{P(e) \to P(s) \to \degZero{P(e)}\}$ of $L(e)$
is mapped under $\LSh{s} L(e)$ to $\{M(s) \to P(s) \to \degZero{M(s)}\}$,
which is quasi-isomorphic to $L(e)\hShift{-1}$.

Since $L(e)$ is 2-spherical as we have seen in \cref{sec:spherical-objects-sl2},
\cref{rmk:twisting-E-by-itself} implies
that if there is any isomorphism $T'_{L(e)} \isom \LSh{s}\hShift{k}$,
then the shift $k$ must be zero.
We shall now show that $T'_{L(e)} \isom \LSh{s}\hShift{0}$ indeed.

Computing $T'_{L(e)}$ involves $\hom^\bullet_{\CatO}(-, L(e))$,
for which we employ the following notation.
Between angle brackets, we write down
$\Complex$-bases for the homological components of $\hom^\bullet_{\CatO}(-, L(e))$.
Every basis element, which is a morphism of complexes, is written
with the codomain and the (shifted) domain written horizontally 
and the map of complexes vertically.
The horizontal arrows between the $\langle\cdots\rangle$'s
carry matrix representations of the boundary map of $\hom^\bullet_{\CatO}(-, L(e))$
\wrt\ these bases.
From the projective resolution $L(e)\qis\{P(e)\to P(s)\to P(e)\}$,
we obtain
{%
\tikzset{
	ampersand replacement=\&,
	commutative diagrams/diagrams={
		row sep=small,
		column sep=tiny,
		nodes={inner sep=1pt, font=\scriptsize}
	}
}%
\renewcommand{\xto}[1]{\xrightarrow{\mathmakebox[\widthof{\scriptsize$\!\!\!\Mtrx{0\\1}\!\!\!$}]{#1}}}%
\begin{align}
	\mathmakebox[1em][l]{\hom^\bullet_\CatO(P(s), L(e))} \notag\\*%
		&= \left\{
			\left\langle
				\begin{tikzcd}[baseline=(B.base)]
					P(s) \ar[d, "b"' name=B] \\
					P(e) \ar[r] \& P(s) \ar[r] \& P(e)
				\end{tikzcd}
	 		\right\rangle
			\xto{\Mtrx{0\\1}}
			\left\langle
				\begin{tikzcd}[baseline=(B.base)]
					\& P(s) \ar[d, "{\id, x}"' name=B]\\
					P(e) \ar[r] \& P(s) \ar[r] \& P(e)
				\end{tikzcd}
			\right\rangle
			\xto{(1, 0)}
			\left\langle
				\begin{tikzcd}[baseline=(B.base)]
					\&\& P(s) \ar[d, "b" name=B]\\
					P(e) \ar[r] \& P(s) \ar[r] \& P(e)
				\end{tikzcd}
			\right\rangle
		\right\}
	\notag\\*
		&\qis 0, \label{eq:twisting-sl2-by-Le:Hom-Ps-Le}
	\\[2.5pt]
	\mathmakebox[1em][l]{\hom^\bullet_\CatO(L(e), P(s))} \notag\\*
		&= \left\{
		\left\langle
			\begin{tikzcd}[baseline=(B.base)]
				P(e) \ar[r] \& P(s) \ar[r] \& P(e) \ar[d, "b"' name=B]\\
				\&\& P(s)
			\end{tikzcd}
		\right\rangle
		\xto{\Mtrx{0\\1}}
		\left\langle
			\begin{tikzcd}[baseline=(B.base)]
				P(e) \ar[r] \& P(s) \ar[r] \ar[d, "{\id, x}"' name=B] \& P(e) \\
				\& P(s)
			\end{tikzcd}
		\right\rangle
		\xto{(1, 0)}
		\left\langle
			\begin{tikzcd}[baseline=(B.base)]
				P(e) \ar[d, "b"' name=B] \ar[r] \& P(s) \ar[r] \& P(e)\\
				P(s)
			\end{tikzcd}
		\right\rangle
		\right\}
	\notag\\*
	&\qis 0 \label{eq:twisting-sl2-by-Le:Hom-Le-Ps},
\intertext{
	We thus get immediately from the  definition \ref{def:twist-cotwist} of $T'$ and $T$
	that $\Xi'_{L(e)} P(s) = 0 = \Xi_{L(e)} P(s)$
	and thus that $T'_{L(e)} P(s) = P(s) = T_{L(e)} P(s)$.
	For $P(e)$, we obtain
}
		\mathmakebox[1em][l]{\hom^\bullet_\CatO(P(e), L(e))} \notag\\*%
	&= \Biggl\{
	\left\langle
	\underbrace{
		\begin{tikzcd}[baseline=(B.base)]
			P(e) \ar[d, "\id"' name=B]\\
			P(e) \ar[r] \& P(s) \ar[r] \& P(e)
		\end{tikzcd}
	}_{\eqqcolon i^{(\bar{2})}}
	\right\rangle
	\xto{1}
	\left\langle
	\underbrace{
		\begin{tikzcd}[baseline=(B.base)] 
			\& P(e) \ar[d, "a"' name=B]\\
			P(e) \ar[r] \& P(s) \ar[r] \& P(e)
		\end{tikzcd}
	}_{\eqqcolon a^{(\bar{1})}}
	\right\rangle
	\xto{0}
	\left\langle
	\underbrace{
		\begin{tikzcd}[baseline=(B.base)]
			\&\& P(e) \ar[d, name=B, "\id"]\\
			P(e) \ar[r] \& P(s) \ar[r] \& P(e)
		\end{tikzcd}
	}_{\eqqcolon i^{(0)}}
	\right\rangle
	\Biggr\}
	\label{eq:twisting-sl2:Hom-Pe-Le}
	\\*[-2ex]
	&\qis \langle i^{(0)}\rangle,
	\notag
	\\[2.5pt]
	\mathmakebox[1em][l]{\hom^\bullet_\CatO(L(e), P(e))} \notag\\*%
	&= \Biggl\{
	\left\langle
	\underbrace{
		\begin{tikzcd}[baseline=(B.base)]
			P(e) \ar[r] \& P(s) \ar[r] \& P(e) \ar[d, "\id"' name=B]\\
			\&\& P(e) 
		\end{tikzcd}
	}_{\eqqcolon i^{(0)}}
	\right\rangle
	\xto{1}
	\left\langle
	\underbrace{
		\begin{tikzcd}[baseline=(B.base)]
			P(e) \ar[r] \& P(s) \ar[r] \ar[d, "b"' name=B] \& P(e) \\
			\& P(e)
		\end{tikzcd}
	}_{\eqqcolon b^{(1)}}
	\right\rangle
	\xto{0}
	\left\langle
	\underbrace{
		\begin{tikzcd}[baseline=(B.base)]
			P(e) \ar[r] \ar[d, "\id"' name=B]\& P(s) \ar[r] \& P(e) \\
			P(e)
		\end{tikzcd}
	}_{\eqqcolon i^{(2)}}
	\right\rangle
	\Biggr\}
	\label{eq:twisting-sl2:Hom-Le-Pe}
	\\*[-2ex]
	&\qis \langle i^{(2)}\rangle \hShift{2},
	\notag
\end{align}
}%
We denote the morphisms of complexes as indicated in the above formulae.
Before we compute the images of $P(e)$ under $T'_{L(e)}$ and $T_{L(e)}$,
we state the following:

\begin{lemma}[(Elimination of trivial summands)]
	\label{lem:gauss-elimination}
	In any abelian category, 
	if the map $d$ in the chain complex in the first line of the diagram
	\[
	\begin{tikzcd}[ampersand replacement=\&, row sep=scriptsize]
		\cdots \rar 
		\& U \rar{\Mtrx{a\\b}}\dar[equal] 
		\& V \oplus W \rar{\Mtrx{c&d\\e&f}}\dar{(1, 0)} 
		\& X \oplus Y \rar{(g, h)} \dar{(-fd', 1)}
		\& Z \rar \dar[equal] \& \cdots\\
		\cdots \rar
		\& U \rar[swap]{b}
		\& V \rar[swap]{e - fd'c}
		\& Y \rar[swap]{h}
		\& Z \rar \& \cdots
	\end{tikzcd}
	\]
	is an isomorphism with inverse $d'$,
	the second line is a chain complex and the vertical map is a quasi-isomorphism.
\end{lemma}

{%
\predisplaypenalty=1000
We now compute the images $T'_{L(e)}P(e)$ and $T_{L(e)}P(e)$ parallelly in two columns.
These are the respective total complexes of the triple complexes
\tikzcdset{every diagram/.append style = {
	row sep={.7cm,between origins}, 
	column sep={1.0cm,between origins}, 
	nodes={inner sep=1pt},
	ampersand replacement=\&,
	cramped,
	baseline=(B.base)
}}
\begin{align}
	\notag
	&\eqsp T'_{L(e)} P(e) &&\eqsp T_{L(e)} P(e)\\*
	\notag
	&=\bigl\{
		\bgroup\color{gray}
			\degZero{P(e)}
			\xto{\mkern -5mu \ev' \mkern-5mu}
		\egroup
		\underbrace{\lin\bigl(\hom^\bullet_\CatO(P(e),L(e)), L(e)\bigr)}_{\Xi'_{L(e)}}\bigr\}
	&&= \underbrace{\bigl\{\hom^\bullet_\CatO\bigl(L(e),P(e)\bigr) \otimes L(e)}_{\Xi_{L(e)} P(e)}
		\bgroup\color{gray}
			\xto{\mkern -5mu \ev \mkern-5mu}
			\degZero{P(e)}
		\egroup
		\bigr\}\\*
\intertext{
	with the gray $P(e)$ in degree $0$.
	Recall from \cref{rmk:lin-and-tensor}
	that the double complexes
	$\lin(\hom^\bullet_{\CatO}(P(e),L(e)), L(e))$ and $\hom^\bullet_{\CatO}(P(e),L(e),P(e)) \otimes L(e)$,
	are sums of shifted copies of $L(e)$, indexed by basis elements
	from \cref{eq:twisting-sl2:Hom-Pe-Le,eq:twisting-sl2:Hom-Le-Pe}.
	Working this out gives
	that $T'_{L(e)} P(e)$ and $T_{L(e)} P(e)$ respectively
	are the triple complexes
}
		&=
	\label{eq:cotwist-of-P(e)-wrt-L(e)}
	\left\{
		\begin{tikzcd}[row sep={.7cm,between origins}, column sep={1.0cm,between origins}, nodes={inner sep=1pt},ampersand replacement=\&]
			     P(e)_{i^{(0)}}       \ar[rr, "a"]\ar[dd,"0"']
			\&\& P(s)_{i^{(0)}}       \ar[rr,"b"]
			\&\& P(e)_{i^{(0)}}       \ar[dd,"0"]\\
			\& |[color=gray]| P(e)    \ar[gray, dddl, "\id"]\ar[gray, dr, "a"]\ar[gray, urrr, "\id" very near start, dashed]\\
			     P(e)_{a^{(\bar{1})}}   \ar[rr, "a", cross line]\ar[dd,"\id"', dashed]	
			\&\& P(s)_{a^{(\bar{1})}}   \ar[from=uu, cross line, "0"]\ar[rr,"b"]\ar[dd,"\id", dashed] 	
			\&\& P(e)_{a^{(\bar{1})}}   \ar[dd,"\id", dashed]\\\\
			     P(e)_{i^{(\bar{2})}} \ar[rr, "a"']
			\&\& P(s)_{i^{(\bar{2})}} \ar[rr,"b"']
			\&\& P(e)_{i^{(\bar{2})}}
		\end{tikzcd}
	\right\}
	&&=
	\left\{
	\begin{tikzcd}
			     P(e)_{i^{(0)}}      \ar[rr, "a"] \ar[dd,"\id"', dashed]
			\&\& P(s)_{i^{(0)}}      \ar[rr,"b"] 
			\&\& P(e)_{i^{(0)}}      \ar[dd,"\id", dashed]\\\\ |[alias=B]|
			     P(e)_{b^{(1)}}      \ar[rr, "a"] \ar[dd,"0"']
			\&\& P(s)_{b^{(1)}}      \ar[from=uu, "\id", dashed] \ar[rr,"b"] \ar[dd,"0"']  	
			\&\& P(e)_{b^{(1)}}      \ar[dd,"0"]\\
		\&\&\& |[color=gray]| P(e) \ar[gray, from=dlll, cross line, "\id", dashed]\ar[gray, from=ul, "b"]\ar[gray, from=uuur, cross line, "\id"' near start]\\
			     P(e)_{i^{(2)}}      \ar[rr, "a"']
			\&\& P(s)_{i^{(2)}}      \ar[rr,"b"'] 
			\&\& P(e)_{i^{(2)}}
	\end{tikzcd}
	\right\} \\
\intertext{
	still with the gray $P(e)$ in degree zero.
	Here, the black $P(-)$'s are indexed by 
	basis elements from \cref{eq:twisting-sl2:Hom-Pe-Le,eq:twisting-sl2:Hom-Le-Pe},
	according to the copy of $L(e)$ they belong to.
	Since $\hom^\bullet_{\CatO}(P(e),L(e))$, $\hom^\bullet_{\CatO}(L(e),P(e))$
	and $L(e)$ are respectively concentrated 
	in degrees $-2$ to $0$, $0$ to $2$ and $-2$ to $0$,
	the two double complexes $\Xi'_{L(e)} P(e)$ and $\Xi_{L(e)} P(e)$ printed in black
	have their degree-zero summands on the secondary diagonal.
	According to \cref{rmk:evaluation-in-terms-of-shifted-copies},
	the non-zero components of the (co)evaluation maps therefore consist of the gray morphisms.
	Using the elimination lemma \ref{lem:gauss-elimination},
	we may eliminate all summands adjacent to one of the dashed identity morphisms;
	this shows that there are quasi-isomorphisms
}
	\notag
	&\qis \{P(e) \xto{a} \degZero{P(s)}\},
	&&\qis \{\degZero{P(s)} \xto{b} P(e)\} \\*
	\notag
	&\qis \LSh{s} P(e),
	&& \qis \RCsh{s} P(e),
\end{align}
}%
by comparison with \eqref{eq:shuffling-sl2:imaegs-of-projectives}.
Hence, we already have shown half of the following:

\needspace{3\baselineskip}
\begin{theorem}
	There is a natural isomorphism $T'_{L(e)} \isom \LSh{s}$.
\end{theorem}
\begin{proof}
	To show that both functors are isomorphic
	we have yet to show that the morphisms $\id_{P(e)}, \id_{P(s)}, a$ and $b$ generating $\End_{\CatO}(P)$
	have isomorphic images under both functors.
	Recall the images $\Sh{s} a\colon M(s)\into P(s)$ and $\Sh{s} b\colon P(s)\onto M(s)$ from \eqref{eq:shuffling-sl2:images-of-generating-morphisms}.
	Recall from \cref{rmk:induced-map-in-terms-of-shifted-copies}
	how the maps $a\colon P(e) \to P(s)$ and $b\colon P(s) \to P(e)$ induce morphisms
	$a_*: \Xi'_{L(e)} P(e) \to \Xi'_{L(e)} P(s)$ and $b_*: \Xi'_{L(e)} P(s) \to \Xi'_{L(e)} P(s)$.
	Since
	\begin{align*}
		a^*\colon \hom^\bullet_{\CatO}(P(s), L(e)) &\longrightleftarrows \hom^\bullet_{\CatO}(P(e), L(e)) \noloc b^*\\
		\begin{aligned}[t]
			b^{(0)}, x^{(\bar{1})}, b^{(\bar{2})}  &\longmapsto 0,\\
			i^{(\bar{1})}                          &\longmapsto a^{(\bar{1})}\\
		\end{aligned}
		&\phantom{{}\longrightleftarrows{}}
		\begin{aligned}[t]
			b^{(0)}          & \longmapsfrom i^{(0)}\\
			x^{(\bar{1})}    & \longmapsfrom a^{(\bar{1})}\\
			b^{(\bar{2})}    & \longmapsfrom i^{(\bar{2})},
		\end{aligned}
	\end{align*}
	the morphisms $a_*$ and $b_*$ act
	on an element $m_{f}$ of a summand $L(e)\hShift{\deg f}_f$
	of $\Xi'_{L(e)}P(e)$ or $\Xi'_{L(e)}P(e)$ indexed by a basis element $f$ by
	\begin{align*}
		a_*: \Xi'_{L(e)} P(e) &\longrightleftarrows \Xi'_{L(e)} P(s)\noloc b_*\\
		\begin{aligned}[t]
			m_{i^{(0)}}, m_{i^{(\bar{2})}}   &\longmapsto 0\\
			m_{a^{(\bar{1})}}                &\longmapsto m_{i^{(\bar{1})}}
		\end{aligned}
		&\phantom{{}\longrightleftarrows{}}
		\begin{aligned}[t]
			m_{i^{(0)}}                      &\mapsfrom m_{b^{(0)}}\\
			m_{a^{(\bar{1})}}                &\mapsfrom m_{x^{(\bar{1})}}\\
		\end{aligned}\quad
		\begin{aligned}[t]
			0                                &\mapsfrom m_{i^{(\bar{1})}}\\
			m_{i^{(\bar{2})}}                &\mapsfrom m_{b^{(\bar{2})}}.
		\end{aligned}
	\end{align*}
	The maps $T'_{L(e)}a$ and $T'_{L(e)}b$ are therefore quasi-isomorphic to
	\[
		\begin{tikzcd}[ampersand replacement=\&, row sep=normal]
			T'_{L(e)} P(s)
				\ar[r, phantom, "\qis"]
				\ar[d, shift right, "T'_{L(e)} b"'] \&[-3.5ex]
			\Bigl\{ \textcolor{gray}{P(s)}
				\ar[r, "\ev'", color=gray] 
				\ar[d, shift right=0.56ex-.6ex, "b"', color=gray] 
				\ar[from=d, shift right=0.56ex+.6ex, "a"', color=gray] 
				\&
			\overbrace{L(e)_{b^{(0)}} \oplus L(e)\hShift{1}_{i^{(\bar{1})}} \oplus L(e)\hShift{1}_{x^{(\bar{1})}} \oplus L(e)\hShift{2}_{b^{(\bar{2})}}}^{\Xi'_{L(e)} P(s)} 
			\mathrlap{\Bigr\}}
				\ar[d, shift right, "b_*=\Mtrx{1\\&0&1\\&&&1}"'] 
			\\
			T'_{L(e)} P(e)
				\ar[r, phantom, "\qis"]
				\ar[u, shift right, "T'_{L(e)} a"']  \&
			\Bigl\{ \textcolor{gray}{P(e)}
				\ar[r, "\ev'"', color=gray] 
				\&
			\underbrace{L(e)_{i^{(0)}} \oplus L(e)\hShift{1}_{a^{(\bar{1})}} \oplus L(e)\hShift{2}_{i^{(\bar{2})}}}_{\Xi'_{L(e)} P(e)}
			\mathrlap{\Bigr\}.}
				\ar[u, shift right, "\Mtrx{0\\&1\\&0\\&&0}=a_*"'] 
		\end{tikzcd}
	\]
	Consider again the left triple complex in \eqref{eq:cotwist-of-P(e)-wrt-L(e)},
	whose black printed “front layer” represents $\Xi'_{L(e)} P(e)$,
	and use the elimination lemma \ref{lem:gauss-elimination}
	to eliminate only the bottom three identities
	$P(-)_{a^{(\bar{1}})} \to P(-)_{i^{(\bar{2})}}$.
	Similarly, a triple complex representing $T'_{L(e)} P(s)$ 
	can be obtained from \eqref{eq:twisting-sl2-by-Le:Hom-Ps-Le}.
	We obtain that $T'_{L(e)}a$ and $T'_{L(e)}b$ are represented by the maps
	\begin{equation}
		\label{eq:twisting-by-Le:action-on-a-b:triple-complexes}
		\begin{tikzcd}[cramped, ampersand replacement=\&, row sep={.7cm,between origins}, column sep={2cm,between origins}]
			\&[-0mm] \mathmakebox[\widthof{$P(e)_{i^{(0)}}$}]{P(e)_{b^{(0)}}}
				\ar[rr]
				\ar[dd, equal, cross line]
				\ar[brace', -, dd, start anchor=north west, end anchor=south west, xshift=-1ex, "T'_{L(e)} P(s) \qis{}" {name=TPs, anchor=east, xshift=-.8ex, font=\normalsize}]
				\ar[dd, equal, cross line]
			\&[-3mm]\&[-0mm] P(s)_{b^{(0)}} \ar[rr]\ar[dd, equal] 
			\&[-3mm]\&[-0mm] \mathmakebox[\widthof{$P(e)_{i^{(0)}}$}]{P(e)_{b^{(0)}}} \ar[dd, equal]
			\\
			\&\& |[gray]| P(s)
				\ar[urrr, gray]
				\ar[dr, gray, equal]
			\\
			\& P(e)_{i^{(\bar{1})}} \ar[rr]	
			\&\& P(s)_{i^{(\bar{1})}}
				\ar[from=uu, equal, cross line]
				\ar[rr] 
			\&\& P(e)_{i^{(\bar{1})}}
				\ar[brace, -, from=uu, start anchor=north east, end anchor=south east]
			\\
			P(e)_{i^{(0)}}
				\ar[from=uuur, cross line, "1"'  near end, shift right]
				\ar[to=uuur, cross line, "0"' very near end, shift right]
			\&\& P(s)_{i^{(0)}}
				\ar[from=uuur, cross line, "1"', shift right]
				\ar[to=uuur, cross line, "0"' very near end, shift right]
				\ar[rr]
			\&\& P(e)_{i^{(0)}}
				\ar[from=uuur, cross line, "1"', shift right]
				\ar[to=uuur, cross line, "0"' very near end, shift right]\\
			\phantom{P(e)_{i^{(0)}}} 
				\ar[brace', -, xshift=-1ex, from=u, start anchor=north west, end anchor=south west, "T'_{L(e)} P(e) \qis{}" {name=TPe, anchor=east, xshift=-.8ex, font=\normalsize}]
			\& |[gray]| P(e) 
				\ar[urrr, gray, equal]
				\ar[from=uuur, cross line, gray, "b"'  at end, shift right]
				\ar[to=uuur, cross line, gray, "a"'  at end, shift right]
				\ar[from=ul, to=ur, cross line]
			\&\&\& \phantom{P(e)_{i^{(0)}}} 
				\ar[brace, -, from=u, start anchor=north east, end anchor=south east]
				\ar[from=TPs, to=TPe, "T'_{L(e)}b"', shift right, xshift=-2ex]
				\ar[to=TPs, from=TPe, "T'_{L(e)}a"' near end, shift right, xshift=-2ex]
		\end{tikzcd}
	\end{equation}
	between the triple complexes representing $T'_{L(e)}P(s)$ and $T'_{L(e)}P(e)$.
	We pass to the total complexes of \eqref{eq:twisting-by-Le:action-on-a-b:triple-complexes}
	and use the elimination lemma \ref{lem:gauss-elimination} 
	to choose quasi-isomorphic replacements
	\begin{equation*}
		\scriptstyle
		\begin{tikzcd}[ampersand replacement=\&, cramped]
			\&[-.6cm]
			\& P(s) \ar[from=d, "\Mtrx{0\\0\\1}", xshift=-.5em] \ar[d, "\Mtrx{0\\-1\\1}", xshift=.5em] \ar[d, phantom, "\qis" rotate=90]
			\\
			T'_{L(e)}P(s) \ar[r, phantom, "\qis"] \ar[d, "T'_{L(e)}b"', shift right] \ar[from=d, "T'_{L(e)}a"', shift right]\&
			\smash{\Bigl\{}
			P(e)_{b^{(0)}} 
				\ar[d, "1"', shift right]
				\ar[from=d, "0"', shift right]
				\ar[r, "\Mtrx{-1\\a\\0}"] \& 
			\begin{matrix}P(e)_{i^{(\bar{1})}} ⊕{}\\{} P(s)_{b^{(0)}} ⊕ \textcolor{gray}{P(s)}\end{matrix}
				\ar[r, "\Mtrx{a&1&1\\0&b&\textcolor{gray}{b}}"] 
				\ar[d, "\Mtrx{0&1\\&0&\textcolor{gray}{b}}"', shift right]
				\ar[from=d, "\Mtrx{0\\0&0\\&\textcolor{gray}{a}}"', shift right] \&
			\begin{matrix}\phantom{{}\oplus{}} P(s)_{i^{(\bar{1})}} \\{} ⊕ P(e)_{b^{(0)}}\end{matrix}
				\ar[r, "\Mtrx{b & -1}"] 
				\ar[d, "\Mtrx{0&1}"', shift right]
				\ar[from=d, "0"', shift right] \& 
			P(e)_{i^{(\bar{1})}}
			\mathrlap{\Bigr\}}
			\\
			T'_{L(e)}P(e) \ar[r, phantom, "\qis"] \&
			\smash{\Bigl\{}
			P(e)_{i^{(0)}} \ar[r, "\Mtrx{a\\0}"] %
			\&
			P(s)_{i^{(0)}} ⊕ \textcolor{gray}{P(e)} \ar[r, "\Mtrx{b & 1}"'] \& 
			P(e)_{i^{(0)}}
			\mathrlap{\smash{\Bigr\}}}
			\\
			\&
			\& M(s)\mathrlap{.} \ar[from=u, "\Mtrx{\can&0}", xshift=.5em] \ar[u, phantom, "\qis" rotate=90] \ar[u, "\Mtrx{\can\\-\can}", xshift=-.5em]
		\end{tikzcd}
	\end{equation*}
	This shows that $T'_{L(e)} b = \LSh{s}b\colon P(s)\onto M(s)$
	and $T'_{L(e)}a = \Sh{s} a\colon M(s)\into P(s)$ indeed.
	Since all morphisms in $\Db(\CatO_0)$ are generated by $a$ and $b$,
	this proves the statement.
\end{proof}

\section{\texorpdfstring{$B_n$-actions for $\SL_3$ and $\SL_n$}{𝐵ₙ-actions for 𝔰𝔩₃ and 𝔰𝔩ₙ}}
\label{sec:sln-case}
The Lie algebra $\SL_3$ has as its Weyl group the symmetric group $S_3=\{e, s, t, st, ts, w_0\}$.
A quiver $Q_{\SL_3}$,
which has vertices indexed by $S_3$
and unique edges $w \leftrightarrows ws$ for all $w < ws$,
and a homogeneous ideal $\mathfrak a_{\SL_3}$ of $\Complex Q_{\SL_3}$
such that $\CatO_{0,\SL_3}\simeq\Mod{}[A_{\SL_3}]$ 
for the algebra $A_{\SL_3}=\Complex Q_{\SL_3}/\mathfrak a_{\SL_3}$
is provided in \autocites{stroppel:quivers,Marko:quivers}.

One sees that $Q_{\SL_2}$ is a full subquiver of $Q_{\SL_3}$
and $\mathfrak a_{\SL_3}\cap\Complex Q_{\SL_2}=\mathfrak a_{\SL_2}$.
The inclusion $A_{\SL_2}\into A_{\SL_3}$ thus induced 
gives rise to an adjoint pair of functors
\begin{equation}
	\label{eq:sl2-sl3:restriction-induction-category-O}
	\begin{aligned}
		\Res_{\SL_3}^{\SL_2}\colon\CatO_{0,\SL_3} 
			&\xtofrom{\ \dashv\ } \CatO_{0,\SL_2}\noloc \Ind_{\SL_2}^{\SL_3},
		\\
		\begin{aligned}
			P(e), P(t) &\mapsto P(e),\\
			P(s), P(st), P(ts), P(w_0) &\mapsto P(s),\\	
		\end{aligned}
		&
		\phantom{{}\xtofrom{\dashv}{}}
		\begin{aligned}
			P(e) &\mapsfrom P(e),\\
			P(s) &\mapsfrom P(s),
		\end{aligned}
	\end{aligned}
\end{equation}
which turns $\CatO_{0,\SL_2}$ into a full subcategory of $\CatO_{0,\SL_3}$.
In particular, $\End_{\CatO_{0,\SL_3}}(P(s)) \isom \Complex[x]/(x^2)$ and
$P(s)$ is $0$-spherelike also in $\CatO_{0,\SL_3}$.

\begin{caveat}
	The Calabi-Yau property from \cref{def:spherical-object:calabi-yau} is not “local”,
	in the sense that an object can lose this property in a larger ambient category.
	For instance, there are non-trivial morphisms 
	$P(s)\to P(t)$ and $P(t)\to P(s)$ in $\CatO_{0,\SL_3}$
	whose composition is zero, so $P(s)$ cannot be spherical.
	We shall present two possible remedies in this section.
\end{caveat}

\subsection{Spherical subcategories}
Consider a $k$-linear triangulated category $\mathcal T$.
\begin{definition}
	An object $E\in\mathcal T$ is said to have a \term{Serre dual} $\mathrm{S} E$
	if the contravariant functor $\Hom_\mathcal T(E,-)^*$%
	---the star stands for vector space dual---%
	is represented by $\mathrm{S} E$.
	If a Serre dual can be chosen functorially
	via an auto-equivalence $\mathrm S$,
	this functor $\mathrm S$ is said to be a \term{Serre functor} of $\mathcal T$.
\end{definition}

\begin{remark}
	The Calabi-Yau condition in \cref{def:spherical-object:calabi-yau}
	for a $d$-spherlike object $E$ to be spherical
	demands that for all $F$, the composition pairing induce an isomorphism
	$\Hom^d_{\Db}(E, F) \isom \Hom_{\Db}(F,E)^*$ or, equivalently, 
	that $E\hShift{-d}$ be a Serre dual for $E$;
	see also \autocite[lem.\ 2.15]{seidel-thomas:braid-group-actions}.
\end{remark}

Let $E\in\mathcal T$ be a $d$-spherelike (but not necessarily spherical) object
that has \emph{some} Serre dual $\mathrm SE$.
Since in particular $\End^d_\mathcal T(E)^* = \Hom_\mathcal T(E, E\hShift{-d})^*
\isom \Hom_\mathcal T(E\hShift{-d}, \mathrm SE)$,
there is a morphism $x^*\colon E\hShift{-d} \to \mathrm SE$
dual to the non-trivial endomorphism $x \in \End^d_\mathcal T(E)$.

\begin{definition}
	The \term{asphericality} of a spherelike object $E$ is $Q(E)≔\cone(x^*)$.
	Its \term{left complement} ${}^\bot Q(E) ≔ \{X\in\mathcal T \mid \Hom_{\mathcal T}(X, Q(E))=0 \}$
	is a full triangulated subcategory of $\mathcal T$.
\end{definition}
\begin{theorem}[(Hochenegger, Kalck, Ploog) {\autocite[Thm.\ 4.4, Appendix A]{HKP:spherical-subcategories}}]
	The \emph{spherical subcategory} $\Sph(E)≔{}^\bot Q(E)$ of $E$
	is the largest triangulated subcategory of $\Db(\mathcal T)$
	in which $E$ is spherical.
\end{theorem}
\begin{example}
	For $\mathfrak g$ a semisimple complex Lie algebra and $\lambda$ a regular weight
	(for instance, $\lambda=0$),
	the auto-equivalence $\mathrm S\coloneqq\LSh{w_0}^2$ is a Serre functor of $\Db(\CatO_\lambda)$
	\autocite[Prop.\ 4.1]{MS:Serre-functors}.
\end{example}
\begin{proposition}
	The $0$-spherelike module $P(s)\in\Db(\CatO_{0,\SL_3})$ has Serre dual $\mathrm SP(s)\isom P(s)^\vee$.
\end{proposition}
\begin{proof}
	For this proof, we take the graded structure on $\CatO_0^{\mathbf Z}$ into account.
	Recall from \cref{def:graded-translation}
	that $\Theta_s^2 ≅ \Theta_s⟨-1⟩⊕\Theta_s⟨1⟩$, and hence $\Sh{s}\Theta_s\isom\Theta_s⟨-1⟩$.
	Recall from \cref{eqn:translation-defining-ses,eq:shuffling-on-Vermas} 
	that for $w<ws$, the functor $\Sh{s}$ maps standard factors $M(w)$ to $M(ws)$
	and $M(ws)$ to $\begin{psmallmatrix}M(ws)\\M(w)/M(ws)\end{psmallmatrix}⟨-1⟩$.
	Applying the factors of $\mathrm S = \LSh{w_0}^2$ successively
	shows that $P(s)$ is mapped under $\mathrm S$ to
	\begin{align}
		\mathmakebox[1em][l]{P(s) = \Theta_s M(s)}
		\notag\\
		&\xmapsto{\Sh{s}} 
			P(s)\langle -1\rangle = \begin{psmallmatrix}M(s)\\M(e)\end{psmallmatrix}\langle -1\rangle
		\notag\\
		&\xmapsto{\Sh{t}} 
			\begin{psmallmatrix}M(ts)\\M(t)\end{psmallmatrix}\langle -1\rangle 
		\notag\\
		&\xmapsto{\Sh{s}} 
			\begin{psmallmatrix}M(w_0)\\M(ts)\end{psmallmatrix}\langle -1\rangle
			= \Theta_t M(ts)\langle -1\rangle 
		\notag\\
		&\xmapsto{\Sh{t}} 
			\begin{psmallmatrix}M(w_0)\\M(ts)\end{psmallmatrix}\langle -2\rangle 
			&
		\notag\\
		&\xmapsto{\Sh{s}}
			\begin{psmallmatrix}
					           & M(w_0)\\
					M(ts)      & M(st)/M(w_0)\\
					M(t)/M(ts)
			\end{psmallmatrix}⟨-3⟩ 
		\notag\\
		&\xmapsto{\Sh{t}}
			\arraycolsep=3pt
			\left(
				\mkern-8mu
				\begin{smallmatrix}
					&&& \begin{smallmatrix} M(w_0)\\M(st)/M(w_0) \end{smallmatrix}
					\\[-0.9em]
					&\begin{smallmatrix}\phantom{()} \\ M(w_0) \end{smallmatrix}
					& \left( \begin{smallmatrix} M(st)\\M(s)/(st) \end{smallmatrix} \middle/ \begin{smallmatrix} M(w_0)\\M(ts)/M(w_0) \end{smallmatrix} \right) \\[-1.225em]
					\\[0.125em]
					\left(
						\begin{smallmatrix} M(t)\\M(e)/M(t) \end{smallmatrix} \middle/ \begin{smallmatrix} \phantom{()}\\ M(w_0) \end{smallmatrix}
				 	\right)
				\end{smallmatrix}
				\mkern-8mu
			\right)
			\langle -4\rangle
		\notag\\
		&\mathmakebox[\widthof{${}\xmapsto{\Sh{t}}{}$}][r]{{}={}} 
			\begin{psmallmatrix}
				& L(w_0)\\\
				& L(st)\quad L(ts)\\
				 L(w_0) &  L(s)\quad L(t)\\
				L(st) \quad L(ts) & L(e)\\
				L(s)
			\end{psmallmatrix}
			\langle -4\rangle,
		\label{eq:composition-series-of-P(s)-dual}
	\end{align}
	which is precisely the composition series
	of the dual module $P(s)^\vee$.
\end{proof}

The space $\Hom(P(s), P(s)^\vee) = \Hom(P(s), \mathrm{S}P(s)) \cong {\End(P(s))}^*$
is spanned by the morphisms $\id^*$ and $x^*$
dual to $\id, x \in \End_\CatO(P(s))$.
According to the grading of $P(s)$ in $\CatO_0^\mathbf{Z}$,
the maps $\id$ and $x$ respectively have degrees $0$ and $2$,
so $\id^*$ and $x^*$ are of degrees $0$ and $-2$.
In terms of composition series of $P(s)$ and $P(s)^\vee$, the maps $\id^*$ and $x^*$ are
\begin{equation}
	\label{eq:P(s)-x-dual-made-explicit}
	\begin{tikzpicture}[baseline=(Ps-1-5.base),
		,
		every matrix/.append style={
			column sep={1.4em,between origins},
			anchor={base east},
			row sep={1em,between origins},
		},
		every edge quotes/.append style={
			execute at begin node={$},
			execute at end node={$}
		}
	]
		\matrix(Ps)[
			matrix of math nodes, 
			matrix anchor=Ps-1-5.base,
			font={\scriptsize}
		]{
			&&&&L(s) &\\
			& L(e) && L(st) && L(ts)\\
			L(s) && |[gray]| L(t) &&|[gray]|L(w_0)\\
			|[gray]| L(st) && |[gray]| L(ts)\\
			&|[gray]| L(w_0)\\
		};
		\node[braced box'={Ps-label}, fit=(Ps-1-5) (Ps-3-1) (Ps-5-2)] {};
		\node[left=-.25cm of Ps-label] {$P(s)={}$};
		\node [yshift=2pt, braced box''={T}, fit=(Ps-1-5) (Ps-3-1) (Ps-2-6)] {};
		\matrix(SPs)[
			right=9em of Ps-1-5.base,
			matrix of math nodes, 
			matrix anchor=SPs-5-2.base, 
			font={\scriptsize}
		]{
			&&&& |[gray]|L(w_0) &\\
			&&& |[gray]|L(st) && |[gray]|L(ts)\\
			& |[gray]|L(w_0) && L(s) && |[gray]|L(t)\\
			L(st) && L(ts) && L(e)\\
			& L(s)\\
		};
		\node[braced box={SPs-label}, fit=(SPs-1-5) (SPs-3-6) (SPs-5-2)] {};
		\node[right=-.25cm of SPs-label] {$ {} = \mathrm{S}P(s) = P(s)^\vee$.};
		\node [braced box'''={T'}, fit=(SPs-5-2) (SPs-4-1) (SPs-4-5)] {};

		\draw[limit bb, ->] (T) to[out=80, in=-100, looseness=2, "x^*" at start] (T');
		\draw[cross line, ->] (Ps-1-5) to["\id^*"' {pos=.4, fill=white}] (SPs-5-2);
		\matrix(G)[matrix of math nodes,matrix anchor=G-5-1.base,xshift=-2.5cm] at  (Ps-1-5.base -| Ps-3-1) {
			-4\\-3\\-2\\-1\\\phantom{-}0\\\phantom{-}1\\\phantom{-}2\\\phantom{-}3\\\phantom{-}4\\
		};
	\end{tikzpicture}
\end{equation}
The grading of the modules is as indicated on the left,
such that the two factors $L(s)$ connected by $\id^*$ are in degree zero
and $x^*$ is a degree $-2$-map indeed.
The grey composition factors belong to the kernel and cokernel of $x^*$, respectively.
Using this explicit description for $x^*$, we can prove the following:

\begin{theorem}
	The inclusion $\Db(\CatO_{0,\SL_2})\into\Db(\CatO_{0,\SL_3})$
	from \cref{eq:sl2-sl3:restriction-induction-category-O}
	factors through $\Sph(P(s))$.
\end{theorem}
\begin{proof}
	From the composition series \cref{eq:composition-series-of-P(s)-dual} of $P(s)^\vee$,
	we can derive that $P(s)^\vee$ has a projective resolution
	$\{P(s)\to P(w_0)⟨-2⟩\to P(w_0)⟨-4⟩\}$.
	Expressiong the map $x^*$ from \cref{eq:P(s)-x-dual-made-explicit}
	in terms of the projective resolution of $P(s)^\vee$ yields
	that the asphericality $Q(P(s)) = \cone(x^*)$ of $P(s)$ is the total complex
	\begin{equation}
		\label{eq:P(s):asphericality}
		Q \qis
		\left\{
			\begin{tikzcd}[nodes={inner sep={2pt}},sep=small, slightly cramped, baseline=(x.base)]
			&& P(s)⟨-2⟩\dar["\strut" name=x]\\
			P(s) \rar & P(w_0)⟨-2⟩ \rar & P(w_0)⟨-4⟩
			\end{tikzcd}
		\right\}
	\end{equation}
	with the canonical inclusions and
	with the bottom right $P(w_0)⟨-4⟩$ is in homological degree $0$.
	We claim that 
	\[
		\Hom_{\Db(\CatO_{0,\SL_3})}(P(w), Q)
		\begin{cases}
			= 0 & \text{if $w\in\{e, s\}$,}\\
			\neq 0 & \text{if $w\in\{t, st, ts, w_0\}$.}
		\end{cases}
	\]
	Consider the composition series of the projective modules 
	involved in \cref{eq:P(s):asphericality}.
	The $P(w_0)⟨-4⟩$ in degree $0$ has composition series
	\begin{equation}
		\label{eq:Ps-dual:resolution:composition-factors-of-Pw0}
		P(w_0)⟨-4⟩ = 
		\begin{psmallmatrix}
			& L(w_0)\\
			& L(st)\quad L(ts)\\
			 \textcolor{gray}{L(s)} & \textcolor{gray}{L(w_0)}\quad L(w_0) & L(t)\\
			 \textcolor{gray}{L(st)}\quad  \textcolor{gray}{L(ts)} &  \textcolor{gray}{L(e)} & \textcolor{gray}{L(st)}\quad \textcolor{gray}{L(ts)}\\
			 \textcolor{gray}{L(w_0)} &  \textcolor{gray}{L(s)}\quad  \textcolor{gray}{L(t)} & \textcolor{gray}{L(w_0)}\\
			&  \textcolor{gray}{L(st)}\quad  \textcolor{gray}{L(ts)}\\
			&  \textcolor{gray}{L(w_0)}
		\end{psmallmatrix}⟨-4⟩;
	\end{equation}
	the gray factors are a composition series of the image of
	\[(P(s)⟨-2⟩\oplus P(w_0)⟨-2⟩ \to P(w_0)⟨-4⟩,\]
	which is the last non-trivial map of the total complex \cref{eq:P(s):asphericality}.
	Every map $P(e)⟨-⟩ \to P(w_0)⟨-4⟩$ and $P(s)⟨-⟩ \to P(w_0)⟨-4⟩$ factors through $P(s)⟨-2⟩\oplus P(w_0)⟨-2⟩$
	and hence is null-homotopic.
	This shows that $\Hom_{\Db(\CatO(\SL_3))}(P(w), Q) = 0$ if $w\in\{e, s\}$.
	
	For $w \in \{t, st, ts, w_0\}$ on the other hand,
	the black composition  factors $L(w)$ in \cref{eq:Ps-dual:resolution:composition-factors-of-Pw0}  
	generate images of non-zero morphisms $P(w)⟨-⟩ \to P(w_0)⟨-2⟩$
	that cannot be factored through $P(s)⟨-2⟩\oplus P(w_0)⟨-2⟩$
	and thus are not null-homotopic.
	These morphisms therefore represent non-trivial morphisms
	$P(w)⟨-⟩ \to Q$ in $\Db(\CatO_{0,\SL_3})$
	and thus show that $\Hom_{\Db(\CatO(\SL_3))}(P(w), Q) \neq 0$ if $w \in \{t, st, ts, w_0\}$.
	
	We see that $\Sph(P(s))$  of $\Db(\CatO_{0,\SL_3})$
	contains the triangulated subcategory generated by $P(s)$ and $P(e)$,
	which proves the claim.
\end{proof}

\subsection{Maximal parabolic subalgebras}
We shall now address parabolic subalgebras $\p$ of $\SL_3$ (\cref{sec:parabolic-subalgebras})
as another remedy for the failure of the Calabi-Yau property
of $P(s)$ in $\CatO_{0,\SL_3}$.
Consider the category $\CatO^\p_0$
corresponding to the parabolic subgroup $W_{\p} = ⟨t⟩ \isom S_1 \times S_2$ of $W = S_3 = ⟨s,t⟩$.
The minimal-length representatives of cosets in $W/W_\p$ are $W^\p=\{e, s, st\}$.
There is an equivalence $\CatO^\p_0\simeq\Mod{A_\p}$
for the path algebra quotient
\begin{equation}
	\label{eq:quiver-sl3}
	A_\p≔A_{\SL_3}/(\trivpath{t}, \trivpath{ts}, \trivpath{w_0}) = \Complex\bigl[e \rightleftarrows s \rightleftarrows st\bigr]\Bigm/ 
		\Bigl(
		\begin{scriptaligned}
			e\from s\from e &= 0, & 	e\from s\from st &= 0,\\[-0.5em]
			st\from s\from e &= 0, &	s\from e \from s &= s\from st\from s
		\end{scriptaligned}
	\Bigr)
\end{equation}
of the path algebra $A_{\SL_3}$.
The parabolic Verma modules and projectives have the following composition series,
whose factors now are simple $A_\p$-modules:
\begin{equation}
	\label{eq:parabolic-sl3-composition-factors}
	\begin{array}{ccc|cc@{{}={}}cc@{{}={}}c}
		\toprule
		M^𝔭(e) &
		M^𝔭(s)  &
		M^𝔭(st) &
		P^𝔭(e) & 
		\multicolumn{2}{c}{P^𝔭(s)} &
		\multicolumn{2}{c}{P^𝔭(st)}\\
		\midrule
		\begin{smallmatrix}\vphantom{L(e)}\\L(e)\\L(s)\end{smallmatrix} &
		\begin{smallmatrix}\vphantom{L(e)}\\L(s)\\L(st)\end{smallmatrix} &
		\begin{smallmatrix}\vphantom{L(e)}\\\vphantom{L(e)}\\L(st)\end{smallmatrix} &
		\begin{smallmatrix}\vphantom{L(e)}\\M(e)\end{smallmatrix} &
		\begin{smallmatrix}M(s)\\M(e)\end{smallmatrix} &
		 \begin{smallmatrix}L(s)\\L(st)\quad L(e)\\L(s)\end{smallmatrix} &
		\begin{smallmatrix}M(st)\\M(s)\end{smallmatrix} &
		 \begin{smallmatrix}L(st)\\L(s)\\L(st)\end{smallmatrix}\\
		\bottomrule
	\end{array}
\end{equation}
According to \cref{rmk:parabolic-translation-and-shuffling},
the defining short exact sequences \cref{eqn:translation-defining-ses} of $\Theta_s$ and $\LSh{s}$
restrict to sequences \cref{eq:ses-parabolic-wall-crossing} in $\CatO_0^\p$
so the indecomposable projective modules
have the following images under $\Theta_s$ and $\LSh{s}$:
\begin{equation}
	\label{eq:parabolc-sl3:images-under-translation-and-shuffling}
	\begin{array}{c|cr@{}>{{}}lcr@{}>{{}}l}
		\toprule
		M                 & \Theta_s M                        &              \multicolumn{2}{c}{\LSh{s} M}               & \Theta_t M             &                                          \multicolumn{2}{c}{\LSh{t} M}                                           \\ \midrule
		P^\p(e)  & P^\p(s)                  & \{P^\p(e)  & \to \degZero{P^\p(s)}\}   & 0                      & \mathmakebox[\widthof{$P^\p(e)$}][c]{P^\p(e)\hShift{-1}} &                                     \\
		P^\p(s)  & P^\p(s)⊕P^\p(s) &                     & \phantom{{}\to{}} P^\p(s) & P^\p(st)      & \{P^\p(s)                                                         & \to \degZero{P^\p(st)}\}   \\
		P^\p(st) & P^\p(s)                  & \{P^\p(st) & \to \degZero{P^\p(s)}\}   & P^\p(st)^{⊕2} &                                                                            & \phantom{{}\to{}} P^\p(st) \\ \bottomrule
	\end{array}
\end{equation}

\begin{remark}
	\label{rmk:inclusion-of-parabolic-subalgebradoesnt-preserve-projectives}
	A module $M$ is $\Sh{w}$-acyclic if and only if $\LSh{w}$
	is quasi-isomorphic to a complex concentrated in degree zero.
	The results form \cref{eq:parabolc-sl3:images-under-translation-and-shuffling}
	therefore are examples for \cref{caveat:parabolic-projectives-not-projective}:
	the objects $P^\p(-)$, 
	albeit projective in the category $\CatO_0^\p$,
	are not $\Sh{s}$-acyclic and hence in particular not projective in $\CatO_0$.
\end{remark}
\begin{lemma}
	The set $\{P^\p(s), P^\p(st)\}$
	is an $\mathrm{A}_2$-collection of $0$-spherical objects in $\Db(\CatO^\p_0)$.
\end{lemma}
\begin{proof}
	From the composition series in \cref{eq:parabolic-sl3-composition-factors},
	we see that $P^\p(s)$ and $P^\p(st)$ have endomorphism algebras
	isomorphic to $\Complex[x]/(x^2)$,
	with the non-trivial endomorphism
	$x\colon P^\p(w) \onto L(w) \into P^\p(w)$ for $w \in \{s, st\}$.
	All Hom-spaces in the following are one-dimensional,
	and we see that for the indecomposable projectives, the composition pairings
	\begin{align}
		\label{eq:composition-pairing:SL3:Ps-Pe}
		\Hom_\mathcal O(P^𝔭(e), P^𝔭(s)) \otimes \Hom_\mathcal O(P^𝔭(s), P^\p(e)) 
		&\longto \End(P^\p(s))/\langle\id\rangle\\* \notag
		\Biggl(
			\begin{tikzpicture}[x=.4cm,y=.4cm, baseline=(Fb.base),every node/.append style={font=\scriptsize}]
				\node (Fa) at (0,1) {$L(s)$};
				\node (Fb) at (1,0) {$L(e)$};
				\draw[brace] ($(Fa)+(0.6,.7)$) to node[brace tip](Pa){} ($(Fb)+(1.1,0.1)$);
				\node (Fc) at (-1,0) {$L(st)$};
				\node (Fd) at (0,-1) {$L(s)$};
				\node[font=\normalsize] at (0,-2.5) {$P^\p(s)$};
				\draw[brace'] (Fc.west |- Fd.south) to (Fd.south -| Fb.east);
				\node[font=\normalsize] at (6,-2.5) {$P^\p(e)$};
				\node (Ta) at (6,0.5) {$L(e)$};
				\node (Tb) at (6,-0.5) {$L(s)$};
				\draw[brace] ($(Tb)-(1,.3)$) to node[brace tip](Pb){} ($(Ta)-(1,-.3)$);
				\draw[brace'] (Ta.west |- Fd.south) to (Fd.south -| Ta.east);
				\node[font=\normalsize] at (12,-2.5) {$P^\p(s)$};
				\node at (12,1) {$L(s)$};
				\node (Sb) at (11,0) {$L(e)$};
				\node (Sc) at (13,0) {$L(st)$};
				\node (T') at (12,-1) {$L(s)$};
				\draw[->, limit bb] (Pa.center) to[out=25,in=180] (Pb.center);
				\draw[->, limit bb] (Ta) to[out=0,in=180] (T');
				\draw[brace'] (Sb.west |- Fd.south) to (Fd.south -| Sc.east);
			\end{tikzpicture}
		\Biggr)&\longmapsto 
		x_{P^\p(s)},\\[3pt]
		\label{eq:composition-pairing:SL3:Ps-Pst}
		\Hom(P^𝔭(st), P^𝔭(s)) \otimes \Hom(P^𝔭(s), P^\p(st))
		&\longto \End(P^\p(s))/\langle\id\rangle \\* \notag
		\Biggl(
			\begin{tikzpicture}[x=.4cm,y=.4cm, baseline=(Fb.base),every node/.append style={font=\scriptsize}]
				\node[font=\normalsize] at (0,-2.5) {$P^\p(s)$};
				\node (Fa) at (0,1) {$L(s)$};
				\node at (-1,0) {$L(e)$};
				\node (Fb) at (1,0) {$L(st)$};
				\node at (0,-1) {$L(s)$};
				\draw[brace] ($(Fa)+(0.6,.7)$) to node[brace tip](Pa){} ($(Fb)+(1.1,.1)$);
				\draw[brace'] (Fc.west |- Fd.south) to (Fd.south -| Fb.east);
				\node[font=\normalsize] at (6,-2.5) {$P^\p(st)$};
				\node (Ta) at (6,1) {$L(st)$};
				\node (Tb) at (6,0) {$L(s)$};
				\node (Tc) at (6,-1) {$L(st)$};
				\draw[brace] ($(Tc)-(1,.3)$) to node[brace tip](Pb){} ($(Tb)-(1,-.3)$);
				\draw[brace] ($(Ta)+(1,.3)$) to node[brace tip](Pc){} ($(Tb)+(1,-.3)$);
				\draw[brace'] (Ta.west |- Fd.south) to (Fd.south -| Ta.east);
				\node[font=\normalsize] at (12,-2.5) {$P^\p(s)$};
				\node (Sb) at (12,1) {$L(s)$};
				\node (Za) at (11,0) {$L(st)$};
				\node (Sc) at (13,0) {$L(e)$};
				\node (Zb) at (12,-1) {$L(s)$};
				\draw[brace] ($(Zb)-(0.8,.2)$) to node[brace tip](Pd){} ($(Za)-(1.0,0)$);
				\draw[->,limit bb] (Pa.center) to[out=45,in=180] (Pb.center);
				\draw[->,limit bb] (Pc.center) to[out=0,in=235] (Pd.center);
				\draw[brace'] (Sb.west |- Fd.south) to (Fd.south -| Sc.east);
			\end{tikzpicture}
		\Biggr)
		&\longmapsto x_{P^\p(s)},\\[3pt]
		\label{eq:composition-pairing:SL3:Pst-Pe}
		\Hom_\mathcal O(P^𝔭(s), P^𝔭(st)) \otimes \Hom_\mathcal O(P^𝔭(s), P^\p(st))
		&\longto \End(P^\p(ps))/\langle\id\rangle\\* \notag
		\Biggl(
			\begin{tikzpicture}[x=.4cm,y=.4cm, baseline=(Fb.base),every node/.append style={font=\scriptsize, inner sep=2pt}]
				\node[font=\normalsize] at (0,-2.5) {$P^\p(st)$};
				\node (Fa) at (0,1) {$L(st)$};
				\node (Fb) at (0,0) {$L(s)$};
				\node[braced box={Pa}, fit=(Fa) (Fb)]{};
				\node at (0,-1) {$L(st)$};
				\draw[brace'] (Fb.west |- Fd.south) to (Fd.south -| Fb.east);
				\node[font=\normalsize] at (6,-2.5) {$P^\p(s)$};
				\node (Ma) at (6,1) {$L(s)$};
				\node (Mb) at (5,0) {$L(st)$};
				\node (Md) at (7,0) {$L(e)$};
				\node (Mc) at (6,-1) {$L(s)$};
				\draw[brace'] (Mb.-170) to node[brace tip'](Pb){} (Mc.south west);
				\draw[brace]  (Mb.170)  to node[brace tip] (Ta){} (Ma.north west);
				\draw[brace'] (Mb.west |- Fd.south) to (Fd.south -| Md.east);
				\node[font=\normalsize] at (12,-2.5) {$P^\p(st)$};
				\node (Sb) at (12,1) {$L(st)$};
				\node (Sc) at (12,0) {$L(s)$};
				\node (Tb) at (12,-1) {$L(st)$};
				\draw[->, limit bb] (Pa.center) to[out=0,in=-135] (Pb.center);
				\draw[-, limit bb] (Ta) to[out=135, out looseness=1.5, in=135] (8,1.5);
				\draw[->, limit bb] (8,1.5) to[out=-45, in=180] (Tb);
				\draw[brace'] (Sb.west |- Fd.south) to (Fd.south -| Sc.east);
			\end{tikzpicture}
		\Biggr)
		&\longmapsto x_{P^\p(st)}
	\end{align}
	are non-degenerate and that
	$\Hom_{\CatO}(P^{\p}(e), P^{\p}(e)(s)) = 0 = \Hom_{\CatO}(P^{\p}(st), P^{\p}(e)(e))$.
	Hence, $P^\p(s)$ and $P^\p(st)$ are $0$-spherical objects
	in $\Db(\CatO_0^\p)$.
	In particular,
	$\dim \Hom(P^\p(s), P^\p(st)) = 1 = \Hom(P^\p(st), P^\p(s))$
	as required;
	so the set $\{P^\p(s), P^\p(st)\}$ is an $\mathrm A_2$-configuration.
\end{proof}

With the dimensions of the hom-spaces in
\crefrange{eq:composition-pairing:SL3:Ps-Pe}{eq:composition-pairing:SL3:Ps-Pst},
it follows immediately from \cref{def:twist-cotwist}
that the indecomposable projective modules under $T'_{P(s)}$ and $T'_{P(st)}$ are
\begin{equation}
	\label{eq:parabolc-sl3:images-under-twist}
	\begin{array}[t]{r @{} >{{}}c<{{}} @{} r @{} >{{}}c<{{}} @{} l}
		T'_{P^\p(s)}\colon 
		P^\p(e) & \mapsto & \{\degZero{P^\p(e)}  & \to & P^\p(s)\}         \\
		P^\p(s)                                & \mapsto &                               &     & P^\p(s)\hShift{1} \\
		P^\p(st)                               & \mapsto & \{\degZero{P^\p(st)} & \to & P^\p(s)\}
	\end{array}
	\qquad
	\begin{array}[t]{r @{} >{{}}c<{{}} @{} r @{} >{{}}c<{{}} @{} l}
		T'_{P^\p(st)}\colon 
		P^\p(e) & \mapsto & P^\p(e)             &     &  \\
		P^\p(s)                                 & \mapsto & \{\degZero{P^\p(s)} & \to & P^\p(st)\}         \\
		P^\p(st)                                & \mapsto &                              &     & P^\p(st)\hShift{1}.
	\end{array}
	\hspace*{-2em}
\end{equation}

\begin{proposition}
	\label{thm:parabolic-sl3:main-statement}
	For $\p$ as above,
	there are natural isomorphisms
	$T'_{P^\p(s)}  \isom \LSh{s}\hShift{1}$ and
	$T'_{P^\p(st)} \isom \LSh{t}\hShift{1}$ 
	of autoequivalences of $\CatO^\p_0$.
\end{proposition}
\begin{proof}
	We see from \cref{eq:parabolc-sl3:images-under-translation-and-shuffling,eq:parabolc-sl3:images-under-twist}
	that $T'_{P^\p(s)}  P^\p(w) \qis \LSh{s} P^\p(w)\hShift{1}$
	and  $T'_{P^\p(st)} P^\p(w) \qis \LSh{t} P^\p(w)\hShift{1}$
	for all $w \in W^\p$.
	It remains to show that the shuffling and spherical twist functors
	also map elements of $\Hom_{\CatO^\p}(P^\p(v), P^\p(w))$
	(for $v, w \in W^\p$)
	to isomorphic maps.
	This will be carried out analogously to \cref{thm:shuffling-is-twisting:sl2}.
	
	The proof for $\Theta_t \isom T'_{P^\p(st)}$ is done, mutatis mutandis, the same way as for $\Theta_s \isom T'_{P^\p(s)}$;
	we thus only show the latter.
	To that end, we show that
	the $A_𝔭$-$A_𝔭$-bimodules $\FBim{\Xi'_{P^\p(s)}}$ and $M_{\Theta_s}$,
	for which the functors $-\otimes \FBim{\Xi'_{P^\p(s)}}$ and $-\otimes M_{\Theta_s}$
	respectively correspond to $\Xi'_{P^\p(s)}$ and $\Theta_s$
	under $\CatO_0^\p \simeq \Mod{}[A_\p]$,
	are isomorphic.
	
	Since as an $A_\p$-module, $P^\p(s) = \trivpath{s} A_\p$
	has an explicit finite basis by paths of the quiver \cref{eq:quiver-sl3} ending in the vertex $s$,
	the module $\FBim{\Xi'_{P^\p(s)}} = P^\p(s)^* ⊗_\Complex P^\p(s)$%
	---the star stands for vector space dual---%
	has a $\Complex$-basis given by the sixteen pairwise tensor products in the schematic
	\begin{equation}
		\label{eq:sl3-p:M-Xi}
		\begin{tikzpicture}[mth, ampersand replacement=\&]
			\matrix (A) [
				matrix of math nodes, 
				row sep=0mm, 
				column sep=2mm, 
				text height=2ex, 
				text depth=.5ex,
				column 1/.style={text width=width("$(s←e←s)^*.$"), align=right},
				column 3/.style={text width=width("$(s←e←s)^*.$")-1em, align=left},
				nodes={font=\normalsize}
			] {
				(s←e←s)^*	\& \phantom{\otimes_\Complex} \& e    	\\
				(s←st)^*	\& \& s←e    	\\
				(s←e)^*		\& \& s←st 	\\
				e^*		\& \& s←e←s	\\
			};
			\node[braced box={T}, fit=(A-1-1) (A-4-1)] {} ;
			\node[braced box'={T'},fit=(A-1-3) (A-4-3)] {} ;
			\path[draw=none] (T) to node[anchor=center]{$\displaystyle ⊗_\Complex$} (T');
			\draw
				(A-1-1.west)					edge[|->, out=195, in=165] node[swap] {$(e←s)_*$} 	(A-2-1.west)
				(A-2-1.base west)		 		edge[|->, out=195, in=165] node[swap] {$(e←s)_*$} 	(A-4-1.west)
				([xshift=-5em]A-1-1.west) 		edge[|->, out=195, in=165] node[swap] {$(st←s)_*$} 	([xshift=-5em]A-3-1.west)
				([xshift=-5em]A-3-1.base west) 	edge[|->, out=195, in=165] node[swap] {$(s←st)_*$} 	([xshift=-5em]A-4-1.west)
				(A-1-3.east)					edge[|->, out=-15, in=15] node {$∘(s←e)$} 			(A-2-3.east)
				(A-2-3.base east)				edge[|->, out=-15, in=15] node {$∘(e←s)$} 			(A-4-3.east)
				([xshift=4.5em]A-1-3.east)		edge[|->, out=-15, in=15] node {$∘(s←st)$} 			([xshift=4.5em]A-3-3.east)
				([xshift=4.5em]A-3-3.base east)	edge[|->, out=-15, in=15] node {$∘(st←s).$} 		([xshift=4.5em]A-4-3.east);
		\end{tikzpicture}
	\end{equation}
	The right $A_𝔭$-action is given by precomposition of paths;
	under the left $A_𝔭$-action, a
	$\gamma \in A_𝔭$ acts on a $\lambda \in P^\p(s)^*$
	by $\gamma_*(\lambda) \colon v \mapsto \lambda(v \circ \gamma)$.
	The $A_𝔭$-$A_𝔭$-bimodule action on $\FBim{\Xi'_{P^\p(s)}}$ hence is
	such that the generating paths of $A_𝔭$ act from the left and right as indicated.
	
	To construct the module $\FBim{\Theta_s}$,
	we consider the endomorphism algebra $\End_{\CatO}(P^\p)$ 
	of $P^\p = P^\p(e) \oplus P^\p(s) \oplus P^\p(st)$,
	which is generated by the elements
	\begin{equation}
		\label{eq:End-Ps-sl3:table}
		\Mtrx{1\\&0\\&&0},
		\Mtrx{0\\ &1\\&&0},
		\Mtrx{0\\&0\\&&1},
		\Mtrx{0 \\ (s \from e) & 0 \\ && 0},
		\Mtrx{0 \\ & 0 & (s \from st) \\&& 0},
		\Mtrx{0 & (e \from s) \\ & 0 \\ && 0},
		\Mtrx{0 \\ & 0 \\ & (st \from s) & 0}.
	\end{equation}
	Since according to \cref{eq:parabolc-sl3:images-under-translation-and-shuffling}
	we have $\Theta_s P^\p = P^\p(s)^{\oplus 4}$,
	we endow these $P^\p(s)$'s with indices to make them distinguishable,
	identifying $P^\p(s)_1$ with $\Theta_s P^\p(e)$,
	$P^\p(s)_2 \oplus P^\p(s)_3$ with $\Theta_s P^\p(s)$
	and $P^\p(s)_4$ with $\Theta_s P^\p(st)$.
	A diagram chase of morphisms through the relevant naturality diagrams
	shows that $\Theta_s$ respectively maps 
	the above generators of $\End_{\CatO}(P^\p)$ from \cref{eq:End-Ps-sl3:table}
	to the endomorphisms
	\begin{equation}
		\label{eq:Theta-s-action-on-End-Ps-sl3:table}
		\Mtrx{1\\&0\\&&0\\&&&0}, 
		\Mtrx{0\\&1\\&&1\\&&&0}, 
		\Mtrx{0\\&0\\&&0\\&&&1} , 
		\Mtrx{0&& 1\\&0\\&&0\\&&&0}, 
		\Mtrx{0&&\\&0\\&&0\\&&1&0}, 
		\Mtrx{0\\1&0\\&&0\\&&&0}, 
		\Mtrx{0\\&0&&1\\&&0\\&&&0}.
	\end{equation}
	of $\Theta_s P^\p = P^\p(s)^{\oplus 4}$.
	Somewhat more suggestively, 
	we may write these morphisms $f$ as arrows connecting the two $P^\p(-)$'s
	on which they have non-zero kernel or cokernel.
	$\Theta_s$ then maps the last four of the generators $f$ from the table,
	which we depict by the following morphisms on the left,
	to the respective solid or dashed morphism on the right:
	\settowidth{\algnRef}{${}⊕{}$}
	\begin{equation}
		\label{eq:Theta-s-action-on-End-Ps-sl3}
		\begin{tikzcd}[column sep={\the\algnRef}, nodes={inner xsep=0pt}, row sep=small, ampersand replacement=\&]
			P^\p = P(e)\rar[phantom, "⊕"] 
			\& P(s)\ar[dl]\ar[dr, dashed] \rar[phantom, "⊕"]
			\& P(st) \ar[d, phantom, "" name=C]
			\&[4em] P(s)_1 \rar[phantom, "⊕"]  \ar[d, phantom, "" name=C']
			\& P(s)_2 \rar[phantom, "⊕"] 
			\& P(s)_3 \rar[phantom, "⊕"]\ar[dll, "\id"' very near end]\ar[dr, dashed, "\id" very near end] 
			\& P(s)_4 =\Theta_s P^\p
			\\
			\phantom{P^\p = {}} P(e)\rar[phantom, "⊕"] \ar[dr] 
			\& P(s) \rar[phantom, "⊕"]
			\& P(st) \ar[dl, dashed] \ar[d, phantom, "" name=D]
			\& P(s)_1 \rar[phantom, "⊕"] \ar[dr, "\id"' very near start] \ar[d, phantom, "" name=D']
			\& P(s)_2 \rar[phantom, "⊕"] 
			\& P(s)_3 \rar[phantom, "⊕"]
			\& P(s)_4 \ar[dll, "\id" very near start, dashed] \phantom{{}=\Theta_s P^\p}
			\\
			\phantom{P^\p = {}} P(e)\rar[phantom, "⊕"] 
			\& P(s) \rar[phantom, "⊕"]
			\& P(st)
			\& P(s)_1 \rar[phantom, "⊕"] 
			\& P(s)_2 \rar[phantom, "⊕"] 
			\& P(s)_3 \rar[phantom, "⊕"]
			\& P(s)_4 \mathrlap{.}\phantom{{}=\Theta_s P^\p}  \ar[phantom, from=C, to=C', "\xmapsto{\Theta_s}"]\ar[phantom, from=D, to=D', "\xmapsto{\Theta_s}"]
		\end{tikzcd}
	\end{equation}
	A vector space basis of $M_{\Theta_s} = \Hom(P^\p, \Theta_s P^\p)$
	is given by morphisms
	sending $P^\p$ to one of its summands $P^\p(w)$ (where $w \in \{e, s, st\}$)
	and further to a summand $P^\p(s)_i$ of $\Theta_s P^\p$.
	In other words, this basis consists of the 16 possible ways
	to map a summand $P^\p(w)$ of $P^\p$ on the right
	to the $P^\p(s)$ in the middle 
	and embed this into $\Theta_s P^\p$
	as one of the summands on the left:
	\begin{equation}
		\label{eq:sl3-p:M-Theta}
		\begin{tikzpicture}[mth, ampersand replacement=\&]
			\matrix (A) [
			matrix of math nodes, 
			row sep={3.5ex,between origins},
			column sep=10mm,
			column 1/.style={anchor=base east, text width=width("$P^\p(s)_4$")},
			column 3/.style={anchor=base west, text width=width("$P^\p(st)$")}
			] {
				P^𝔭(s)_3 \&  \& P^𝔭(s)\\
				P^𝔭(s)_1 \&  \& P^𝔭(e)\\[-1.75ex]
				\& P^𝔭(s) \\[-1.75ex]
				P^𝔭(s)_4 \&  \& P^𝔭(st)\\
				P^𝔭(s)_2 \&  \& P^𝔭(s)\\
			};
			\draw[->] (A-3-2)
				edge[->, out=180, in=0] (A-1-1)
				edge[->, out=180, in=0] (A-2-1)
				edge[->, out=180, in=0] (A-4-1)
				edge[->, out=180, in=0] (A-5-1)
				edge[<-, in=180, out=0, "$\id$" very near end] (A-1-3)
				edge[-, in=180, out=0] (A-2-3)
				edge[-, in=180, out=0] (A-4-3)
				edge[-, in=180, out=0, "$x$"' very near end] (A-5-3);
			\draw[looseness=1]
				(A-1-1.west)                     edge[|->, out=180+15, in=180-15] node[swap] {$𝛩_s(e←s)∘$}             (A-2-1.west)
				(A-2-1.base west)                edge[|->, out=180+15, in=180-15] node[swap, pos=.4] {$𝛩_s(e←s)∘$}     (A-5-1.west)
				([xshift=-5em]A-1-1.west)        edge[|->, out=180+15, in=180-15] node[swap] {$𝛩_s(st←s)∘$}            ([xshift=-5em]A-4-1.west)
				([xshift=-5em]A-4-1.base west)   edge[|->, out=180+15, in=180-15] node[swap] {$𝛩_s(s←st)∘$}            ([xshift=-5em]A-5-1.west)
				(A-1-3.east)                     edge[|->, out=   -15, in=    15] node {$∘(s←e)$}                      (A-2-3.east)
				(A-2-3.base east)                edge[|->, out=   -15, in=    15] node[pos=.4] {$∘(e←s)$}              (A-5-3.east)
				([xshift=4.5em]A-1-3.east)       edge[|->, out=   -15, in=    15] node {$∘(s←st)$}                     ([xshift=4.5em]A-4-3.east)
				([xshift=4.5em]A-4-3.base east)  edge[|->, out=   -15, in=    15] node {$∘(st←s).$}                    ([xshift=4.5em]A-5-3.east);
		\end{tikzpicture}
	\end{equation}
	To understand the $A^\p$-$A^\p$-bimodule action on $\FBim{\Theta_s}$,
	recall from \cref{cor:morita-equivalence-and-induced-module}
	that $\phi, \psi \in A^{\p}$ act on $m\in \FBim{\Theta_s}$
	by $\phi \ldot m \ldot \psi = \Theta_s(\phi) \circ m \circ \psi$,
	with the images $\Theta_s(\phi)$ from
	\cref{eq:Theta-s-action-on-End-Ps-sl3:table,eq:Theta-s-action-on-End-Ps-sl3}.
	On the vector space basis of $\FBim{\Theta_s}$ from \cref{eq:sl3-p:M-Theta},
	the generating paths of the algebra $A^\p$ therefore
	act from the left and right as indicated.
	
	Comparing \cref{eq:sl3-p:M-Xi} and \cref{eq:sl3-p:M-Theta}
	shows that the obvious isomorphism $M_{\Xi'_{P^\p(s)}} \isom M_{\Theta_s}$
	of vector spaces is an isomorphism of $A^\p$-$A^\p$-bimodules.
	It follows that $T'_{P^\p(s)} \simeq \{\id \Rightarrow -\otimes M_{\Xi'_{P^\p(s)}}\}$
	and $\Sh{s} = \{\id \Rightarrow -\} \cong M_{\Theta_s}$ are naturally isomorphic functors.
\end{proof}

\subsection*{Parabolic subalgebras of $\SL_n$}
We now transfer results for $\CatO^\p_0$ from $\SL_3$ to $\SL_n$.
Consider the parabolic subalgebra $\p$ of $\SL_n$
corresponding to the subgroup $W_\p=⟨s_2, \dotsc, s_{n-1}⟩ = S_{1}\times S_{n-1}$
of $S_n = ⟨s_1,\dotsc,s_{n-1}⟩$.
We let $\sigma_0 ≔ e$ and $\sigma_i ≔ s_1\dotsm s_i$ for $i \geq 1$;
the minimal length coset representatives then are $W^\p=\{\sigma_0,\dotsc,\sigma_{n-1}\}$. 

\begin{lemma}
	\label{lem:parabolic-sln:path-algebra}
	The category $\CatO^\p_0(\SL_n)$ is equivalent to $\Mod{}[A^\p(\SL_n)]$ for
	\begin{equation*}
		\label{eq:parabolic-sln:path-algebra}
		A^\p=\Complex\bigl[e\rightleftarrows\sigma_1\rightleftarrows\dotsb\rightleftarrows\sigma_{n-1}\bigr]
		\Bigm/
		\Bigl(
		\setlength\jot{-.3ex}
		\begin{scriptaligned}
			e\from\sigma_1\from e &= 0,\\
			\sigma_i\from\sigma_{i+1}\from\sigma_i &= \sigma_i\from\sigma_{i-1}\from\sigma_i,\\
			\sigma_i\from\sigma_{i\pm 1}\from\sigma_{i\pm 2} &= 0
		\end{scriptaligned}
		\Bigr)_{1 \leq i \leq n-2.}
	\end{equation*}
	We shall denote this quiver by $Q^\p$.
\end{lemma}

\begin{proof}
	We compute the composition series of Verma modules and indecomposable projectives in $\CatO^\p_0$
	using the generalised Kazhdan-Lusztig theorem
	\autocites[Cor.\ 7.1.3]{Irving:filtered-category-OS}[Thm.\ 1.3]{CC:parabolic-KL}.
	For parabolic subgroups of the form $W_\p=S_{k}\times S_{n-k}\leq S_n$
	there is a graphical calculus for computing parabolic Kazhdan-Lusztig polynomials
	\autocites[§5]{Brundan-Stroppel:Highest-weight-categories-I}{JS:Graphical-Kazhdan-Lusztig}.
	The composition series thus obtained are listed in \cref{eq:composition-series-in-Op-Sn}.
	\begin{table}
		\caption{Composition series of Verma modules and indecomposable projectives
			indexed by the $\sigma_i \in W^\p$.}
		\label{eq:composition-series-in-Op-Sn}
		\centering
		$\begin{array}{c|ccccc}
			\toprule
			& \sigma_0 & \sigma_1 & \sigma_2 &\cdots & \sigma_{n-1}\\
			\midrule
			M^𝔭(-)
			& \begin{smallmatrix}L(e)\\L(s_1)\end{smallmatrix}
			& \begin{smallmatrix}L(\sigma_1)\\L(\sigma_2)\end{smallmatrix}
			& \begin{smallmatrix}L(\sigma_2)\\L(\sigma_3)\end{smallmatrix}
			& \cdots
			& L(\sigma_{n-1})\\[1em]
			P^𝔭(-)
			& \begin{smallmatrix}L(e)\\L(s_1)\end{smallmatrix}
			& \begin{smallmatrix}L(\sigma_1)\\L(e)\quad L(\sigma_2)\\L(\sigma_1)\end{smallmatrix}
			& \begin{smallmatrix}L(\sigma_2)\\L(s_1)\quad L(\sigma_3)\\L(\sigma_2)\end{smallmatrix}
			& \cdots
			& \begin{smallmatrix}L(\sigma_{n-1})\\L(\sigma_{n-2})\\L(\sigma_{n-1})\end{smallmatrix}\\
			\bottomrule
		\end{array}$
	\end{table}
	These composition series show that the unique (up to scalar)
	morphisms $P^\p(\sigma_i) \to P^\p(\sigma_{i \pm 1})$
	are irreducible,
	that they generate $\End_{\CatO^\p} \bigoplus_{i=0}^{n-1} P(\sigma_i)$,
	and that their relations generate precisely
	the ideal quotiented out in the statement of the lemma.
\end{proof}

Every Verma module $M^\p(\sigma_i)$
fits uniquely into a short exact sequence
\[
	M^\p(\sigma_i)\into P^\p(\sigma_{i+1})\onto M^\p(\sigma_{i+1});
\]
in particular, $\Theta_{s_i} M^\p(\sigma_i)=P^\p(\sigma_i)$
is always projective,
which is not true in the non-parabolic category $\CatO_0$.
From these sequences we obtain
the images of the indecomposable projective and Verma modules $M$
under translation and shuffling functors
listed in \crefrange{eq:parabolic-sln:translation-of-vermas-and-projectives}{eq:parabolic-sln:shuffling-of-projectives}
(two-term complexes are understood to have the right entry in degree zero).
\bgroup
\newcommand{\M}[1]{M^\p(\sigma_{#1})}
\renewcommand{\P}[1]{P^\p(\sigma_{#1})}
\begin{table}
	\caption{Images of Verma modules and indecomposable projectives under translation functors.}
	\label{eq:parabolic-sln:translation-of-vermas-and-projectives}
	$\begin{array}{@{} c|cccc@{}}
		\toprule
		M     & \Theta_{s_1} M & \Theta_{s_2} M & \Theta_{s_3} M & \cdots         \\ \midrule
		\M{0} & \P{1}          &                &  \\
		\M{1} & \P{1}          & \P{2}          &                &  \\
		\M{2} & \P{1}          & \P{2}          & \P{3}          &                \\
		\M{3} &                & \P{2}          & \P{3}          &                \\
		\M{4} &                &                & \P{3}          & \smash{\ddots} \\ \bottomrule
	\end{array}$
	\hfill
	$\begin{array}{@{} c|cccc @{}}
		\toprule
		M     & \Theta_{s_1} M & \Theta_{s_2} M & \Theta_{s_3} M & \cdots         \\ \midrule
		\P{0} & \P{1}          &                &  \\
		\P{1} & \P{1}^2        & \P{2}          &                &  \\
		\P{2} & \P{1}          & \P{2}^2        & \P{3}          &                \\
		\P{3} &                & \P{2}          & \P{3}^2        &                \\
		\P{4} &                &                & \P{3}          & \smash{\ddots} \\ \bottomrule
	\end{array}$
\end{table}
\begin{table}
	\caption{Images of Verma modules under derived shuffling functors.}
	\label{eq:parabolic-sln:shuffling-of-vermas}
	\centering
	$\begin{array}{@{} c | *{3}{r@{}>{{}}l<{{}}@{}l} l @{} }
		\toprule
		M & \multicolumn{3}{c}{\LSh{s_1} M}  &  \multicolumn{3}{c}{\LSh{s_2} M}  & \multicolumn{3}{c}{\LSh{s_3} M}  & \cdots         \\ \midrule
		\M{0}                   &         &             & \M{1}    & \M{0}    & \hShift{-1} &          & \M{0}    & \hShift{-1} &         &                \\
		\M{1}                   & \{\M{1} & \longto     & \P{1})\} &          &             & \M{2}    & \M{1}    & \hShift{-1} &         &                \\
		\M{2}                   & \M{2}   & \hShift{-1} &          & \{ \M{2} & \longto     & \P{2} \} &          &             & \M{3}   &                \\
		\M{3}                   & \M{3}   & \hShift{-1} &          & \M{3}    & \hShift{-1} &          & \{ \M{3} & \to         & \P{3}\} &                \\
		\M{4}                   & \M{4}   & \hShift{-1} &          & \M{4}    & \hShift{-1} &          & \M{4}    &             &         &  \\ 
		\M{5}                   & \M{5}   & \hShift{-1} &          & \M{5}    & \hShift{-1} &          & \M{5}    &             &         & \smash{\ddots} \\ \bottomrule
	\end{array}$
\end{table}
\begin{table}
	\caption{Images of indecomposable projectives under derived shuffling functors.}
	\label{eq:parabolic-sln:shuffling-of-projectives}
	\centering
	$\begin{array}{@{} c | *{3}{r@{}>{{}}l<{{}}@{}l} l @{} }
		\toprule
		M & \multicolumn{3}{c}{\LSh{s_1} M}  & \multicolumn{3}{c}{\LSh{s_2} M}  & \multicolumn{3}{c}{\LSh{s_3} M} & \cdots         \\ \midrule
		\P{0}                   &         &             & \M{1}    & \P{0}   & \hShift{-1} &          & \P{0}   & \hShift{-1} &         &  \\
		\P{1}                   &         &             & \P{1})   & \{\P{1} & \longto     & \P{2}\}  & \P{1}   & \hShift{-1} &         &  \\
		\P{2}                   & \{\P{2} & \longto     & \P{1})\} &         &             & \P{2}    & \{\P{2} & \longto     & \P{3}\} &                \\
		\P{3}                   & \P{3}   & \hShift{-1} &          & \{\P{3} & \longto     & \P{2} \} &         &             & \P{3}   &                \\
		\P{4}                   & \P{4}   & \hShift{-1} &          & \P{4}   & \hShift{-1} &          & \{\P{4} & \longto     & \P{2}\} & \\
		\P{5}                   & \P{5}   & \hShift{-1} &          & \P{5}   & \hShift{-1} &          & \P{5}   & \hShift{-1} &         &  \smash{\ddots} \\ \bottomrule
	\end{array}$
\end{table}
\egroup

\begin{lemma}
	$\{P^\p(\sigma_1),\dotsc,P^\p(\sigma_{n-1})\}$
	is an $\mathrm A_{n-2}$-configuration of $0$-spherical objects.
\end{lemma}
\begin{proof}
	The composition series from \cref{eq:composition-series-in-Op-Sn}
	exhibit that $\Hom_{\CatO_0^\p}(P^\p(\sigma_i), P^\p(\sigma_i)) \isom \Complex[x]/(x^2)$
	for all $1 \leq i \leq n-1$,
	that the nontrivial endomorphism $x$ is the degree $2$-map
	\[
		x\colon P^\p(\sigma_i) \onto L^\p(\sigma_i) \into P^\p(\sigma_i),
	\]
	and that
	\[
		\dim\Hom_{\CatO_0^\p}(P^\p(\sigma_j), P^\p(\sigma_i)) =
		\begin{smallcases}
			2 & \text{if $i=j$},\\
			1 & \text{if $|i-j|=1$}\\
			0 & \text{otherwise.}
		\end{smallcases}
	\]
	It is sufficient to check non-degeneracy of the composition pairing
	\[
		\Hom_{\CatO_0^\p}(-, P^\p(\sigma_i)) 
		\otimes \Hom_{\CatO_0^\p}(P^\p(\sigma_i), -)
		\to \langle x_{P^\p(\sigma_i)} \rangle
	\]
	only for the indecomposable projectives $P^\p(\sigma_{i\pm 1})$
	that are connected to $P^\p(\sigma_{i})$ by an arrow in $Q^\p$
	because for $j \notin \{i-1, i, i+1\}$,
	the composition pairing
	\[
		\underbrace{\Hom_{\CatO_0^\p}(P^\p(\sigma_j), P^\p(\sigma_i))}_{0}
		\otimes 
		\underbrace{\Hom_{\CatO_0^\p}(P^\p(\sigma_i), P^\p(\sigma_j))}_{0}
		\to \langle x_{P^\p(\sigma_i)} \rangle
	\]
	is non-degenerate trivially.
	
	For $P^\p(\sigma_{i\pm 1})$, we see,
	analoguously to \crefrange{eq:composition-pairing:SL3:Ps-Pe}{eq:composition-pairing:SL3:Pst-Pe},
	that the composition of the only (up to scalars) non-zero morphisms to and from $P^\p(\sigma_{i})$,
	which are written down in terms of composition series in the following,
	is $x_{P^\p(\sigma_i)}$.
	This shows that for all $1\leq i<n-2$
	(and mutatis mutandis also for $i = n-1$),
	the composition pairing
	\begin{align}
		\Hom(P^\p(\sigma_{i\pm 1}), P^\p(\sigma_{i}))
			\otimes\Hom(P^\p(\sigma_i), P^\p(\sigma_{i\pm 1}))
			& \longto \End(P^\p(\sigma_i))/\langle \id\rangle\\
			\Biggl(
			\tikzset{x=.4cm,y=.4cm, baseline=(Fb.base),every node/.append style={font=\scriptsize, inner sep=0pt}}
			\underbrace{
				\begin{tikzpicture}[remember picture]
					\node (Fa) at (0,1) {$L(\sigma_i)$};
					\node at (-1.25,0) {$L(\sigma_{i+1})$};
					\node (Fb) at (1.75,0) {$L(\sigma_{i-1})$};
					\node (Fc) at (0,-1) {$L(\sigma_i)$};
					\draw[brace] (Fa.50) to node[brace tip](Pa){} (Fb.10);
				\end{tikzpicture}
			}_{\textstyle P^\p(\sigma_{i})}
			\qquad
			\underbrace{
				\begin{tikzpicture}[remember picture]
					\node (Ma) at (0,1) {$L(\sigma_{i+1})$};
					\node (Mb) at (-1.25,0) {$L(\sigma_{i\pm 2})$};
					\node at (1.25,0) {$L(\sigma_{i})$};
					\node (Mc) at (0,-1) {$L(\sigma_{i+1})$};
					\draw[brace'] (Mb.-170) to node[brace tip'](Pb){} (Mc.south west);
					\draw[brace]  (Mb.170)  to node[brace tip] (Ta){} (Ma.north west);
					\begin{scope}[overlay]
						\draw[->] (Pa.center) to[out=70, out looseness=1.6, in=-135] (Pb.center);
					\end{scope} 
				\end{tikzpicture}
			}_{\textstyle P^\p(\sigma_{i\pm 1})}
			\qquad
			\underbrace{
				\begin{tikzpicture}[remember picture, trim left=(Gd.west)]
					\node (Ga) at (0,1) {$L(\sigma_i)$};
					\node (Gd) at (-1.75,0) {$L(\sigma_{i+1})$};
					\node (Gb) at (1.25,0) {$L(\sigma_{i-1})$};
					\node (Gc) at (0,-1) {$L(\sigma_i)$};
					\draw[brace']  (Gc.-80)to node[brace tip'](Tb){} (Gb.-10) ;
						\draw[-, limit bb] (Ta) to[out=135, in looseness=1.5, in=115] (-4, 0);
						\draw[->, limit bb] (-4, 0) to[out=-65, out looseness=1.5, in=-70] (Tb);
				\end{tikzpicture}
			}_{\textstyle P^\p(\sigma_{i})}
			\Biggr)
			& \longmapsto x_{P^\p(\sigma_i)} \notag
		\end{align}
		is non-degenerate.
\end{proof}
\begin{theorem}
	\label{thm:main-result}
	For the parabolic subalgebra $\p\subseteq\SL_n$
	corresponding to the subgroup $W_\p = S_{n-1} \times S_1 < S_n$,
	the auto-equivalences $\LSh{s_i} \hShift{1}$ and $T'_{P^\p(\sigma_i)}$ of $\Db(\CatO^\p_0)$
	are naturally isomorphic
	for every $1 \leq i \leq n-1$.
\end{theorem}
\begin{proof}
	Let $A^\p_n$ and $Q^\p_n$
	respectively be the path algebra quotient and the quiver from \cref{lem:parabolic-sln:path-algebra}
	for $\SL_n$.
	The assignment $p\colon A^\p_n \to A^\p_n/(\trivpath{\sigma_0}) \xto{\cong} A^\p_{n-1}$
	induces fully faithful functor
	\begin{equation*}
		p^* \colon \CatO^\p_0(\SL_{n-1})\to\CatO^\p_0(\SL_{n}),\quad
		P^ \p(e) \mapsto M^\p(\sigma_1),\quad
		P^\p(\sigma_{i-1})\mapsto P^\p(\sigma_{i})\ \text{for $1\leq k \leq n-2$,}
	\end{equation*}
	which exhibits $\CatO^\p_0(\SL_{n-1})$ as a full subcategory of $\CatO^\p_0(\SL_{n})$;
	by induction, it follows that on the triangulated category of $\Db(\SL_{n})$
	generated by $P^\p(\sigma_2), \dotsc, P^\p(\sigma_{n-1})$
	there are natural isomorphisms $\LSh{s_i} \hShift{1} \cong T'_{P^\p(\sigma_i)}$
	for $2 \leq i \leq n-1$.
	We shall see, however, that it is easier to show the statement
	also for $i = 1$ and all of $\Db(\SL_{n})$ directly,
	without resorting to induction,
	employing the proof of \cref{thm:parabolic-sl3:main-statement}
	with the following alterations:
	
	Since
	\begin{equation}
		\label{eq:parabolic-sln:proof-of-main-thm:zero-hom}
		\Hom_\CatO(P^\p(\sigma_i), P^\p(\sigma_j)) = 0
		\qquad
		\text{for $0 \leq i,j \leq n-1$ with $\lvert i-j\rvert \geq 2$}
	\end{equation}
	it follows from the definition of $T'$
	that $T'_{P^\p(\sigma_i)}$ (for $1 \leq i \leq n-1$)
	acts as identity on all $P^\p(\sigma_j)$
	for $0 \leq j \leq n-1$ with $\lvert i-j\rvert \geq 2$,
	as does, according to \cref{eq:parabolic-sln:shuffling-of-projectives},
	the functor $\LSh{s_i}\hShift{1}$.
	One checks that both functors act by identity also on morphisms
	between these modules;
	this shows that $T'_{P(\sigma_i)} \cong \LSh{s_i} \hShift{1}$ for all $1 \leq i \leq n-1$
	as functors on the triangulated subcategory of $\Db(\CatO^\p_0(\SL_n))$
	generated by these $P^\p(\sigma_j)$'s.
	
	The category $\CatO(\SL_n)$ has a projective generator
	$P^\p_n \coloneqq \bigoplus_{k=0}^{n-1} P^\p(\sigma_k)$,
	with image $\Theta_{s_i} P^\p_n = P^\p(\sigma_i)^4$
	for all $1 \leq i \leq n-1$.
	Again due to \cref{eq:parabolic-sln:shuffling-of-projectives},
	we obtain that
	\[\FBim{\Theta_{s_i}} = \Hom_{\CatO_0}( P^\p_n, \Theta_{s_i} P^\p_n)
		= \Hom_{\CatO_0}\bigl(
			P^\p_n(\sigma_i-1)\oplus P^\p_n(\sigma_i) \oplus P^\p_n(\sigma_i+1), 
			P^\p_n(\sigma_i)^4\bigr
		)
	\]
	Replacing $e$ by $\sigma_{i-1}$, $s$ by $\sigma_i$ and $st$ by $\sigma_{i+1}$
	in the proof of \cref{thm:parabolic-sl3:main-statement}
	then gives a proof of $T'_{P(\sigma_i)} \cong \LSh{s_i}\hShift{1}$ for all $1 \leq i \leq n-1$
	as functors on $\Db(\CatO^\p_0(\SL_n))$.
\end{proof}

\Needspace{4\baselineskip}
\section{Further observations and final remarks}
\subsection{\texorpdfstring{$\CatO^\p_0$}{O} as a spherical subcategory}
We know that the object $P^\p(s)\in\CatO_0^\p(\SL_3)$ is spherical, so
one might ask whether $\CatO^\p_0(\SL_3)$ arises
as the spherical subcategory $\Sph(P^\p(s))$
of $P^\p(s)\in\CatO_{0,\SL_3}$.
However, $P^\p(s)$ is not spherelike in $\Db(\CatO(\SL_3))$,
\ie, we cannot assign a meaningful spherical subcategory to it.

To see this, consider the projective resolution $P^\p(s) \qis \{P(s)\to P(ts)\to P(s)\}$
in $\Db(\CatO(\SL_3))$.
Using this resolution, we obtain the chain complex
(see \cref{sec:twisting-by-Le} for an explanation of the notation)
\bgroup
\tikzset{
	ampersand replacement=\&,
	commutative diagrams/diagrams={
		row sep=small,
		column sep=tiny,
		nodes={inner sep=1pt, font=\scriptsize}
	}
}%
\begin{multline*}
	\hom^\bullet_{\Db(\CatO_0)}(P^\p(s), P^\p(s))
	\\
	\qis\left\{
	\left\langle
	\begin{tikzcd}[baseline=(R.base), ampersand replacement=\&]
		P(s) \rar\dar \& P(ts)\rar\dar \& P(s)\ar[d, "{\id, x}" name=R]\\
		P(s) \rar \& P(ts)\rar \& P(s)
	\end{tikzcd}
	\right\rangle
	\to
	0
	\to
	\left\langle
	\begin{tikzcd}[baseline=(R.base), ampersand replacement=\&]
		\&\& P(s) \rar\dar["{\id}"']\dar[phantom, "\phantom{\id, x}" name=R] \& P(ts) \rar \&  P(s) \\
		P(s) \rar \& P(ts) \rar \&  P(s)
	\end{tikzcd}
	\right\rangle
	\right\},
\end{multline*}
\egroup
whose leftmost bracket is in degree zero.
The complex $\hom^\bullet_{\Db(\CatO_0)}(P^\p(s), P^\p(s))$ 
thus has total dimension $3$, so $P^\p(s)$ is not spherelike.
As a side note, we notice that the inclusion
$\Db(\CatO_0^\p) \subset \Db(\CatO_0)$,
given by mapping projectives to projectives,
is not full.

\subsection{Necessity of extremal partitions}
Is it necessary to choose a parabolic subalgebra $𝔭$ which corresponds to the “extremal” partition $(n-1, 1)$,
\ie\ to the parabolic subgroup $S_{n-1}\times S_1<S_n$?
The maximal parabolic subalgebra
\begin{equation}
	\label{eq:sl4:subalgebra-without-spherical-objects}
	𝔭 = \begin{psmallmatrix}* & * & * &*\\ *&*&*&*\\0&0&*&*\\0&0&*&*\end{psmallmatrix} ⊂ \SL_4,
\end{equation}
corresponding to the parabolic subgroup $W_𝔭=⟨s⟩×⟨u⟩=S_2\times S_2<S_4=⟨s,t,u⟩$
with minimal length coset representatives $W^𝔭=\{e, t, tu, ts, tsu, tsut\}$
has parabolic Verma modules and indecomposable projectives
with composition series listed in \cref{tab:sl4:subalgebra-without-spherical-objects:composition-series};
these can be obtained from parabolic Kazhdan\-/Lusztig polynomials
(see also \autocite{Brundan-Stroppel:Highest-weight-categories-III}).
One checks from the composition series 
that $P^𝔭(t)$, $P^𝔭(ts)$ and $P^𝔭(tsu)$ are the only spherelike indecomposable projectives,
and there no possible $\mathrm A_3$ configuration of indecomposable projectives.
One might wonder if one can find at least an $\mathrm A_2$-configuration or a single spherical object.

Unfortunately, we encounter the same problem of the present spherelike objects
as in the non\-/parabolic $\SL_3$-case;
this can be seen from the composition series as follows.
The projectives $P^𝔭(ts)$, $P^𝔭(tu)$ corresponding to the two incomparable weights $ts\not\gtrless tu$
have non-trivial morphisms $P^𝔭(ts)→P^𝔭(tu)$ and $P^𝔭(tu)→P^𝔭(ts)$ whose composition
\begin{equation}
	\label{eq:sl4:P(ts)-P(tu)-not-spherical}
	\begin{tikzpicture}[x=.4cm,y=.4cm, node distance=1mm, baseline=(current bounding box.center),every node/.append style={inner sep=0mm}]
		\begin{scope}[every node/.append style={font=\scriptsize}]
			\node (I-ts) at (0,0) {$L(ts)$};
			\node[right=of I-ts] (I-tsut) {$L(tsut)$};
			\node[right=of I-tsut] (I-tu) {$L(tu)$};
			\node[above=of I-tsut] (I-t) {$L(t)$};
			\node[below=of I-tsut] (I-tsu) {$L(tsu)$};
			\node[above=of I-tu] (I-tsu') {$L(tsu)$};
			\node[above=of I-tsu'] (I-ts') {$L(ts)$};
			\node[braced box={Pa}, fit=(I-tsu') (I-ts')] {} ;
			\node[right=4 of I-tu] (II-ts) {$L(ts)$};
			\node[right=of II-ts] (II-tsut) {$L(tsut)$};
			\node[right=of II-tsut] (II-tu) {$L(tu)$};
			\node[above=of II-tsut] (II-t) {$L(t)$};
			\node[below=of II-tsut] (II-tsu) {$L(tsu)$};
			\node[above=of II-tu] (II-tsu') {$L(tsu)$};
			\node[above=of II-tsu'] (II-tu') {$L(tu)$};
			\draw[brace] ([xshift=1mm]II-tsu.south west) coordinate (c1) to node[brace tip] (Pb){} ([xshift=-1mm]II-ts.base west) coordinate (c2);
			\draw[->, limit bb] 
			let \p1=(c1),\p2=(c2),\n1={atan2(\y2-\y1,\x2-\x1)}
			in (Pa) to[out=0,in=\n1+90] (Pb.center);
			\node[braced box={Pc}, fit=(II-tsu') (II-tu')] {};
			\node[right=4 of II-tu] (III-ts) {$L(ts)$};
			\node[right=of III-ts] (III-tsut) {$L(tsut)$};
			\node[right=of III-tsut] (III-tu) {$L(tu)$};
			\node[above=of III-tsut] (III-t) {$L(t)$};
			\node[below=of III-tsut] (III-tsu) {$L(tsu)$};
			\node[above=of III-tu] (III-tsu') {$L(tsu)$};
			\node[above=of III-tsu'] (III-ts' ){$L(ts)$};
			\draw[brace] ([xshift=1mm]III-tsu.south west) coordinate (c1) to node[brace tip] (Pd){} ([xshift=-1mm]III-ts.base west) coordinate (c2);
			\draw[->, limit bb] 
			let \p1=(c1),\p2=(c2),\n1={atan2(\y2-\y1,\x2-\x1)}
			in (Pc) to[out=0,in=\n1+90] (Pd.center);
		\end{scope}
		\node[above=1 of I-t] (Pts) {$P^𝔭(ts)$};
		\node[above=1 of II-t] (Ptu) {$P^𝔭(tu)$};
		\node[above=1 of III-t] (Pts') {$P^𝔭(ts)$};
		\draw[->] (Pts) edge (Ptu) (Ptu) edge (Pts');
	\end{tikzpicture}
\end{equation}
is the zero morphism.
The same holds true for $\bigl(P^𝔭(tu)→P^𝔭(t)→P^𝔭(tu)\bigr) = 0$.
Hence neither $P^𝔭(tu)$ nor $P^𝔭(ts)$ is spherical.
For $P^𝔭(t)$ the composition
\begin{equation}
	\label{eq:sl4:P(t)-not-spherical}
	\begin{tikzpicture}[x=.4cm,y=.4cm, node distance=1mm, baseline=(current bounding box.center),every node/.append style={inner sep=0mm}]
		\begin{scope}[every node/.append style={font=\scriptsize}]
			\path
			node(I-ts) at (0,0) {$L(ts)$}
			node[right=of I-ts] (I-tsut) {$L(tsut)$}
			node[right=of I-tsut] (I-tu) {$L(tu)$}
			node[above=of I-tsut] (I-t) {$L(t)$}
			node[below=of I-tsut] (I-tsu) {$L(tsu)$}
			node[left=of I-ts] (I-e') {$L(e)$}
			node[below=of I-e'] (I-t') {$L(t)$}
			node[below=of I-t'] (I-tsut') {$L(tsut)$}
			node[braced box={P-I}, fit=(I-t) (I-tsu) (I-tu)] {}
			node[right=6.5 of I-tsu] (II-t) {$L(t)$}
			node[below=of II-t] (II-tsut) {$L(tsut)$}
			node[left=of II-tsut] (II-ts)  {$L(ts)$}
			node[right=of II-tsut] (II-tu) {$L(tu)$}
			node[below=of II-tsut] (II-tsu) {$L(tsu)$}
			node[above=of II-tu] (II-tsu') {$L(tsu)$}
			node[above=of II-tsu'] (II-ts') {$L(ts)$}
			node[right=of II-tsu']  (II-tsu'') {$L(tsu)$}
			node[above=of II-tsu''] (II-tu'') {$L(tu)$}
			node[right=of II-tu''] (II-tsut''') {$L(tsut)$}
			node[above=of II-tsut'''] (II-tsu''') {$L(tsu)$}
			node[braced box'={P-II}, fit=(II-t) (II-tsu) (II-ts)] {}
			node[braced box={P-II'}, fit=(II-tsu''') (II-tsut''') ] {}
			node[right=6 of II-tu''](III-ts) {$L(ts)$}
			node[right=of III-ts] (III-tsut) {$L(tsut)$}
			node[right=of III-tsut] (III-tu) {$L(tu)$}
			node[above=of III-tsut] (III-t) {$L(t)$}
			node[below=of III-tsut] (III-tsu) {$L(tsu)$}
			node[left=of III-ts] (III-e') {$L(e)$}
			node[below=of III-e'] (III-t') {$L(t)$}
			node[below=of III-t'] (III-tsut') {$L(tsut)$};
			\draw[brace] ([yshift=-1mm]III-tsu.base east) coordinate (C-III') to node[brace tip] (P-III){} ([yshift=-1mm]III-tsut'.south west) coordinate (C-III);
			\draw[->]
			(P-I) to[out=0, in=180,looseness=1.5] (P-II);
			\draw[limit bb] 
			(P-II') to[out=0, in=135] ($(III-tsut')-(2,1.2)$) coordinate (H1);
			\draw[->, limit bb]
			let \p1=(C-III),\p2=(C-III'),\n1={atan2(\y2-\y1,\x2-\x1)}
			in (H1) to[out=-45,in=\n1-90, looseness=1] (P-III.center);
		\end{scope}
		\path
		node[above=1 of I-t] (H2) {\strut}
		node (Pt) at ($(H2) !.5! (I-ts |- H2)$) {$P^𝔭(t)$}
		node (Ptsu) at (Pt -| II-ts') {$P^𝔭(tsu)$}
		node (Pt') at ($(III-t |- H2) !.5! (III-ts |- H2)$) {$P^𝔭(t)$};
		\draw[->] (Pt) -- (Ptsu);
		\draw[->] (Ptsu) -- (Pt');
	\end{tikzpicture}
\end{equation}
shows that $P^𝔭(t)$ it is not spherical either.

Therefore, for $𝔭⊆\SL_4$ the parabolic subalgebra
corresponding to the parabolic subgroup $S_2\times S_2\leq S_4$,
the modules $P^𝔭(t)$, $P^𝔭(ts)$ and $P^𝔭(tu)$ 
are the only spherelike indecomposable projective modules,
and none of them is spherical.

\begin{table}
	\caption{%
		Composition series of parabolic Verma modules and indecomposable projectives
		in $\mathcal{O}^\p_0$ for $\p⊆\SL_4$ the parabolic subalgebra
		corresponding to the parabolic subgroup $W_\p ≔ S_2\times S_2 \leq S_4$
		(see \cref{eq:sl4:subalgebra-without-spherical-objects}).
	}
	\label{tab:sl4:subalgebra-without-spherical-objects:composition-series}
	\tikzset{node distance=1mm, x=3mm, every node/.append style={inner sep=0mm, font=\scriptsize}, baseline}
	\centering
	\begin{tabular}{c|c|c}
		\toprule
		$w∈W^𝔭$ & $M^𝔭(w)$ & $P^𝔭(w)$\\
		\midrule
		$e$
		& 
		\tikz\draw
		node(I-e')  at (0,0) {$L(e)$}
		node[below=of I-e'] (I-t') {$L(t)$}
		node[below=of I-t'] (I-tsut') {$L(tsut)$};
		&
		\scriptsize dto.
		\\
		\midrule
		$t$
		&
		\tikz\draw
		node (I-ts) at (0,0) {$L(ts)$}
		node[right=of I-ts] (I-tsut) {$L(tsut)$}
		node[right=of I-tsut] (I-tu) {$L(tu)$}
		node[above=of I-tsut] (I-t) {$L(t)$}
		node[below=of I-tsut] (I-tsu) {$L(tsu)$};
		&
		\tikz\draw
		node(I-ts) at (0,0) {$L(ts)$}
		node[right=of I-ts] (I-tsut) {$L(tsut)$}
		node[right=of I-tsut] (I-tu) {$L(tu)$}
		node[above=of I-tsut] (I-t) {$L(t)$}
		node[below=of I-tsut] (I-tsu) {$L(tsu)$}
		node[left=1 of I-ts] (I-e') {$L(e)$}
		node[below=of I-e'] (I-t') {$L(t)$}
		node[below=of I-t'] (I-tsut') {$L(tsut)$};
		\\\midrule
		$ts$
		&	
		\tikz\draw
		node (ts'') at (0,0) {$L(ts)$}
		node[below=of ts''] (tsu'') {$L(tsu)$};
		&
		\tikz\path
		node (ts'') at (0,0) {$L(ts)$}
		node[below=of ts''] (tsu'') {$L(tsu)$}		
		node[below left=of tsu''] (tsu''') {$L(tsu)$}
		node[left=of tsu'''](tsut''') {$L(tsut)$}
		node[left=of tsut'''] {$L(ts)$}
		node[above=of tsut'''] {$L(t)$}
		node[below=of tsut'''] {$L(tsu)$};	
		\\	\midrule
		$tu$
		&
		\tikz\path
		node (tu'') at (0,0) {$L(tu)$}
		node[below=of tu''] (tsu'') {$L(tsu)$};
		&
		\tikz\path
		node (tu'') at (0,0) {$L(tu)$}
		node[below=of tu''] (tsu'') {$L(tsu)$}		
		node[below left=of tsu''] (tsu''') {$L(tsu)$}
		node[left=of tsu'''](tsut''') {$L(tsut)$}
		node[left=of tsut'''] {$L(ts)$}
		node[above=of tsut'''] {$L(t)$}
		node[below=of tsut'''] {$L(tsu)$};
		\\\midrule
		$tsu$
		&
		\tikz\draw
		node (tsu) at(0,0) {$L(tsu)$}
		node[below=of tsu] (tsut) {$L(tsut)$};
		&
		\tikz\path
		node (tsu) at(0,0) {$L(tsu)$}
		node[below=of tsu] (tsut) {$L(tsut)$}
		node[left=1 of tsut]  (tu') {$L(tu)$}
		node[below=of tu'] (tsu') {$L(tsu)$}	
		node[left=1 of tu'] (ts'') {$L(ts)$}
		node[below=of ts''] (tsu'') {$L(tsu)$}		
		node[below left=of tsu''] (tsu''') {$L(tsu)$}
		node[left=of tsu'''](tsut''') {$L(tsut)$}
		node[left=of tsut'''] {$L(ts)$}
		node[above=of tsut'''] {$L(t)$}
		node[below=of tsut'''] {$L(tsu)$};
		\\\midrule
		$tsut$
		&
		\tikz\draw
		node at (0,0)  {$L(tsut)$};
		&
		\tikz\path
		node (tsut) at (0,0) {$L(tsut)$}
		node[below left=of tsut] (tsu') {$L(tsu)$}
		node[below=of tsu'] (tsut') {$L(tsut)$}
		node[below left=1mm and 1 of tsut'] (tu'') {$L(tu)$}
		node[left=of tu''] (tsut'') {$L(tsut)$}
		node[left=of tsut''] (ts'') {$L(ts)$}
		node[above=of tsut''] {$L(t)$}
		node[below=of tsut''] {$L(tsu)$}
		node[left=1 of ts''] (e''') {$L(e)$}
		node[below=of e'''] (t''') {$L(t)$}
		node[below=of t'''] {$L(tsut)$};
		\\
		\bottomrule
	\end{tabular}
\end{table}

\subsection*{Acknowledgements}
This article compiles the results of a Master's thesis
written with the advice of Catharina Stroppel,
whom the author wants to thank for her continuous support
as well as the opportunity provided to dive into this topic.

\printbibliography
\end{document}